\documentclass[11pt]{article}
\usepackage{amsmath}
\usepackage[thmmarks,amsmath,amsthm]{ntheorem}
\usepackage{comment}

\usepackage{algorithm,algpseudocode}

\newcounter{algsubstate}
\renewcommand{\thealgsubstate}{\alph{algsubstate}}
\newenvironment{algsubstates}
  {\setcounter{algsubstate}{0}%
   \renewcommand{\State}{%
     \stepcounter{algsubstate}%
     \Statex {\footnotesize\thealgsubstate:}\space}}

\usepackage{amsfonts, psfrag, graphicx, amssymb,enumitem,indentfirst,float,geometry,color,enumerate,graphicx,subfigure,algorithm,algpseudocode,pifont,hyperref,mathtools}
\usepackage{pbox}
\usepackage{booktabs, cellspace, hhline}
\allowdisplaybreaks[4]
\numberwithin{equation}{section}
\numberwithin{figure}{section}

\usepackage{etoolbox}
\patchcmd{\thebibliography}{\chapter*}{\section*}{}{}

\usepackage{lipsum}

\newtheorem{thm}{Theorem}[section]
\newtheorem{lemma}{Lemma}[section]
\newtheorem{rem}{Remark}[section]

\usepackage{chngcntr}
\counterwithin{table}{section}

\newcommand{\commentout}[1]{{}} 

\newcommand{\Frac}[2]{\frac{\textstyle #1}{\textstyle #2}}
\newcommand{\abs}[1]{\left|#1\right|}
\newcommand{\norm}[1]{\left\|#1\right\|}

\newcommand{\bfb}{{\bf b}}

\newcommand{\bfc}{{\bf c}}

\newcommand{\bfF}{{\bf F}}

\newcommand{\bfn}{{\bf n}}

\newcommand{\bfu}{{\bf u}}

\newcommand{\bfalpha}{{\boldsymbol \alpha}}

\newcommand{\bfdelta}{\boldsymbol{\delta}}

\newcommand{\bfgamma}{\boldsymbol{\gamma}}

\begin{document}
\title{A Fixed Mesh Method With Immersed Finite Elements \\
for Solving Interface Inverse Problems \thanks{This research was partially supported  by GRF 15301714/15327816
of HKSAR, and Polyu AMA-JRI}\\
}
\author{
Ruchi Guo \thanks{Department of Mathematics, Virginia Tech, Blacksburg, VA 24061 (ruchi91@vt.edu) }
\and Tao Lin \thanks{Department of Mathematics, Virginia Tech, Blacksburg, VA 24061 (tlin@vt.edu) } \\
\and Yanping Lin \thanks{Department of Applied Mathematics, Hong Kong Polytechnic University, Kowloon, Hong
Kong, China (yanping.lin@polyu.edu.hk)} \\
  }
\date{}
\maketitle

\begin{abstract}
We present a new fixed mesh algorithm for solving a class of interface inverse problems for the typical elliptic interface problems. These interface inverse problems are formulated as shape optimization problems whose objective functionals depend on the shape of the interface.
Regardless of the location of the interface, both the governing partial differential equations and the objective functional are discretized optimally,
with respect to the involved polynomial space, by an immersed finite element (IFE) method on a fixed mesh. Furthermore, the formula for the gradient of the descritized objective function is derived within the IFE framework that can be computed accurately and efficiently through the discretized adjoint procedure. Features of this proposed IFE method based on a fixed mesh are demonstrated by its applications to three representative interface inverse problems: the interface inverse problem with an internal measurement on a sub-domain, a Dirichlet-Neumann type inverse problem whose data is given on the boundary, and a heat dissipation design problem. \\


\end{abstract}

\section{Introduction}
In this article, we present a numerical method for solving a class of interface inverse problems with a fixed mesh by an immersed finite element (IFE) method. Without loss of generality, let $\Omega$ be a domain separated by an interface $\Gamma$ into two subdomains $\Omega^-$ and $\Omega^+$ each occupied by a different material represented by a piecewise constant function $\beta$ discontinuous across $\Gamma$. We consider a group of $K$ forward interface boundary problems posed on the domain
$\Omega$ for the typical second order elliptic equation:
\begin{equation}
\begin{split}
\label{inter_prob_0}
&-\nabla\cdot(\beta\nabla u^{k}) = f^{k}, ~~\textrm{in} \; \Omega^-\cup\Omega^+, \\
&u^{k}=g_D^{k}, ~~\textrm{on}\; \partial\Omega_D^k \subseteq \partial \Omega,~~\frac{\partial u^{k}}{\partial\mathbf{ n}}=g_N^{k}, ~~\textrm{on}\; \partial\Omega_N^k \subseteq \partial \Omega,
\end{split}~~~~~~\text{for~~} k = 1, 2, \cdots, K,
\end{equation}
where $\overline{\partial\Omega_N^k} \cup \overline{\partial\Omega_D^k} =\partial \Omega$ and $\bfn$ is the outward normal of $\partial \Omega$, together with the jump conditions on the interface $\Gamma$:
\begin{eqnarray}
&&\begin{split}
\label{jump_cond_0}
&[u^{k}]|_{\Gamma}:=u^{k,+}-u^{k, -}=0, \\
&[\beta\nabla u^{k} \cdot\mathbf{ n}]|_{\Gamma}:=\big(\beta^+\nabla u^{k,+}-\beta^-\nabla u^{k,-}\big)\cdot\mathbf{ n}=0, \text{~~$\bfn$ is the normal of $\Gamma$},
\end{split}~~~~~~1 \leq k \leq K, \\
&&\text{in which~~~}u^{k,s} = u^k|_{\Omega^s}, ~\beta(X) = \beta^s \text{~~for~~} X \in \Omega^s, ~s = -, +. \label{beta}
\end{eqnarray}
%

An important inverse problem related to the typical second order elliptic equation is to identify the coefficient $\beta$ where one needs to either identify the physical properties of materials, i.e., the values (the {\em parameter estimation} problem) and/or detect the location and shape of inclusions/interfaces (the {\em inverse geometric} problem) using the data measured for $u^{k},~1 \leq k \leq K$ on a subset of the domain or on a subset of the boundary $\partial \Omega$ \cite{1991ChowAnderssen,2013Hegemann,2001ItoKunischLi}. This type of inverse problems arise from many applications in engineering and sciences, such as the electrical impedance tomography (EIT) \cite{2006Calderon,2005HolderDavid} and groundwater or oil reservoir simulation \cite{1983Ewing,1986Yeh}. In the former case, $u^{k}$s and $\beta$ represent the electrical potential and the conductivity, respectively, whereas in the latter case $u^{k}$s and $\beta$ are the piezometric head and transmissivity, respectively. Similar inverse problems related to other partial differential equations also appear in applications, we refer readers to \cite{2008HogeaDavatzikos,2003JiMcLaughlin} for medical imaging problems, \cite{2000LeeParkPark,1992SchnurZabaras} for elasticity problems, and references therein. It is well known that these inverse problems are usually ill-posed especially when the available data is rather limited. Numerical methods based on the output-least-squares formulation are commonly used to handle these types of inverse problems, see
\cite{2003ChanTai,1999ChenZou,1991ChowAnderssen,2013Hegemann,1990ItoKunisch} and references therein.

\commentout{
In the past decades, many works on the parameter estimation problems have been published. Some of these works main focus only on recovering the coefficient values,
and a commonly used approach for such as a case is the variational method, which involves the output-least-squares method, the augmented Lagrangian method \cite{1999ChenZou,2009GockenbachKhan,1990ItoKunisch,2007TaiLi}, and various regularization techniques such as the total variation and Tikhonov regularization \cite{1992CattLionsMorel,2003ChanTai,1985KravarisSeinfeld}.
\commentout{
Roughly speaking, instead of assuming the coefficients take one single value on each sub domain, this group of methods minimize an misfit function quantifying the error between the computed solutions and the observation data in a chosen norm based on the discretization of the coefficients over the entire domain. Regularization techniques can be used to avoid some undesirable results or to handle random errors in data.
}
However, as demonstrated in \cite{2002AmeurChavent,2003ChanTai,1991ChowAnderssen,2013Hegemann} and the references therein, these methods can be adopted to recover not only the values of coefficients but also the discontinuities, i.e., they can solve the parameter estimation problems and the inverse geometric problems simultaneously.
}


In many engineering and science applications, the values of material properties or parameters are known or chosen such as the elastic properties of tissue and bone in medical problems \cite{2010McLaughlinZhang,2011PeregoVeneziani} and the electrical properties in EIT problems \cite{1995Alessandrini,2013BelhachmiMeftahi} to mention just a couple of applications. Thus, the focus of this article is to develop an efficient numerical method based on a fixed mesh for the inverse geometric problem related to the forward interface problem described by \eqref{inter_prob_0} and \eqref{jump_cond_0} in which we assume that the material values $\beta^s
= \beta|_{\Omega^s}, s =-, +$ are known priori and we need to use given measurements about $u$ to recover the location and geometry
of the material interface $\Gamma$.

A widely used approach for an inverse geometric problem is the shape optimization method \cite{2003HaslingerMkinen,2013NovotnSokoowski} by which we seek for the interface $\Gamma^*$ from an optimization problem:
\begin{equation}
\label{interf_optimiz}
\Gamma^* = \textrm{argmin} \; \mathcal{J}(u^{1}(\Gamma), u^{1}(\Gamma), \cdots, u^{K}(\Gamma), \Gamma),
\end{equation}
where
\begin{equation}
\label{ObjFun_4_interf_optimiz}
\mathcal{J}(u^{1}(\Gamma), u^{2}(\Gamma), \cdots, u^{K}(\Gamma), \Gamma) = \int_{\Omega_0} J(
u^{1}(\Gamma), u^{2}(\Gamma), \cdots, u^{K}(\Gamma); X, \Gamma) dX,
\end{equation}
and $u^{k}(\Gamma)$s are the solutions to the forward interface problems \eqref{inter_prob_0} and \eqref{jump_cond_0}, but $\Omega_0 \subseteq \Omega$ and \\ $J(u^{(1)}(\Gamma), u^{(2)}(\Gamma), \cdots, u^{(K)}(\Gamma); X, \Gamma)$ are application dependent, a few specific formulations of $J$ are given in Section \ref{sec:Applications} for a chosen group of representative applications. We note that the shape optimization approach has been applied to numerous applications,
see for example \cite{2013BelhachmiMeftahi,
2005BURGEROSHER,
2010HuangXie,
2004LiStevenXie,
2004MohammadiBijan}.
\commentout{
{\color{blue}The shape optimization approach has been applied to numerous applications such as
the EIT problems \cite{2013BelhachmiMeftahi}, the identification of obstacles \cite{2011BurgerMatevosyan}, and inverse
geometric problems in linear elasticity \cite{2005BURGEROSHER}. In addition, the shape optimization is also an important technique in shape designing problems in engineering such as aerodynamical applications \cite{2004MohammadiBijan}, heat dissipation designing \cite{2004LiStevenXie}, and elasticity compliance designing \cite{2010HuangXie}.}
}

The movement of the structure, boundary or interface is a critical issue in a shape optimization process challenging a solver chosen for the related forward problems. Traditional finite element methods can be used to obtain accurate solutions to the forward interface problems provided that they use a body fitting (or interface conforming) mesh \cite{2014AllaireDapogny,2007CarpentieriKoren}; otherwise, their performance may not be satisfactory
\cite{2000BabuskaOsborn,1998ChenZou}.
The shape optimization methods based on body fitting mesh are referred as the Lagrangian approach \cite{1994ChoiChang} which, however, has a few drawbacks.
The first concerns the mesh updating process from one iteration to the next in the optimization. As the geometry changes, to guarantee the accuracy, the mesh used by a chosen solver for the forward problem needs to be updated to fit the new shape of the boundary or interface \cite{1995Bendose,1991SuzukiKikuchi}, which
not only consumes time but also generates unsatisfactory meshes in many situations, see the illustrations in Figure \ref{fig:interface_independent_mesh}
where the two plots on the left demonstrate an inappropriate mesh movement strategy leading to a mesh with less desirable qualities, especially near the right edge.
\commentout{
{\color{blue}Several approaches were proposed to overcome this issue, such as the mesh parameterization \cite{1985BENNETTBOTKIN,1984BraibantFleury,1994TortorelliTomasko}. But these approaches are difficult to handle the large shape changes, which often appear in shape optimization \cite{1985BENNETTBOTKIN,1989YaoChoiKyung} and could cause excessive mesh distortion and consequently inaccurate finite element solutions.}
}

\begin{figure}[H]
\label{fig_mesh}
  \centering
     \subfigure[The initial body fitting mesh]{
    \label{body_fitting_mesh_0} 
    \includegraphics[width=1.2in]{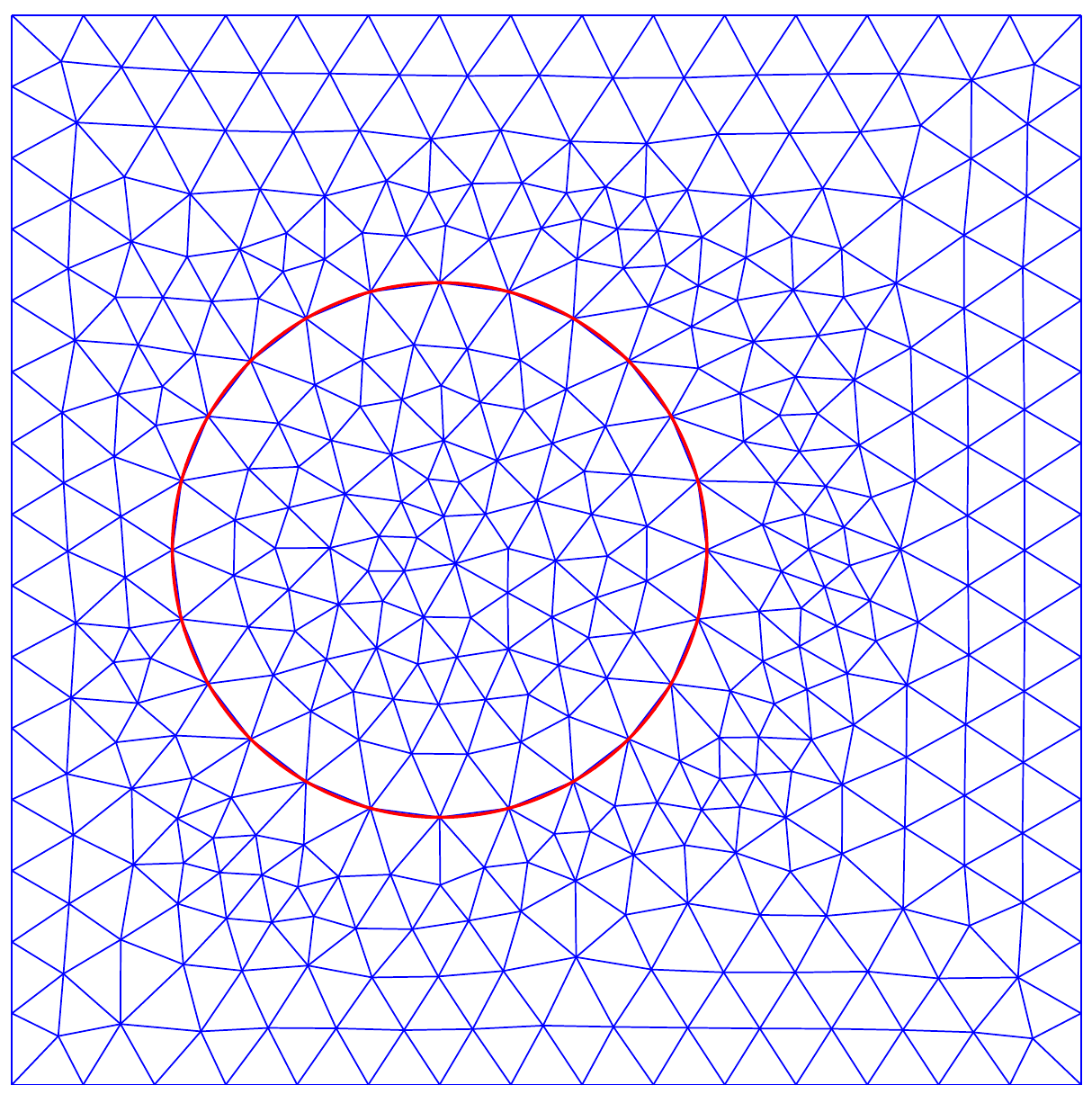}}
    \subfigure[The body fitting mesh after movement]{
    \label{body_fitting_mesh_1} 
    \includegraphics[width=1.2in]{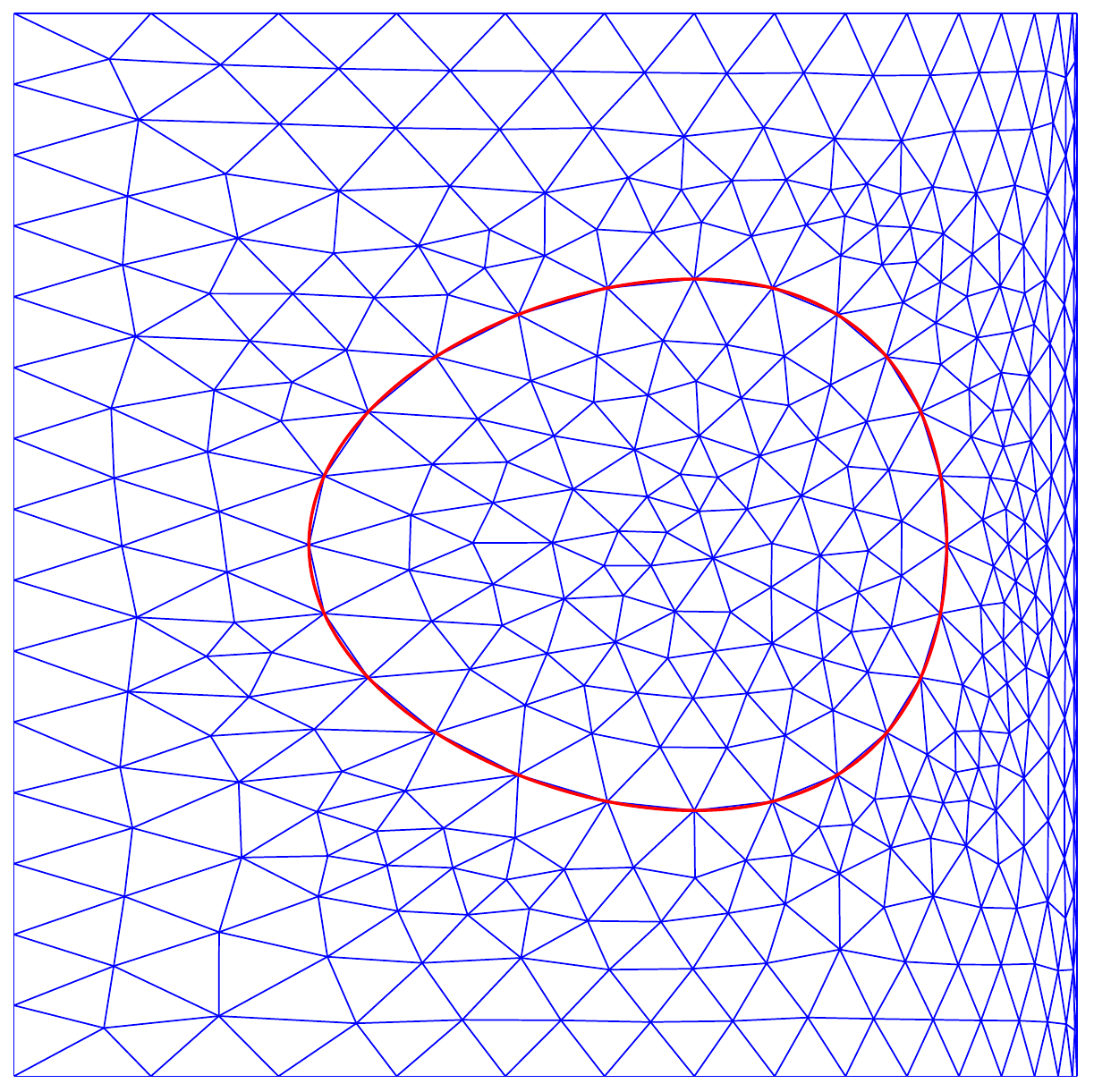}}
    \subfigure[The interface independent mesh]{
    \label{fig_bodyfitting_mesh} 
    \includegraphics[width=1.2in]{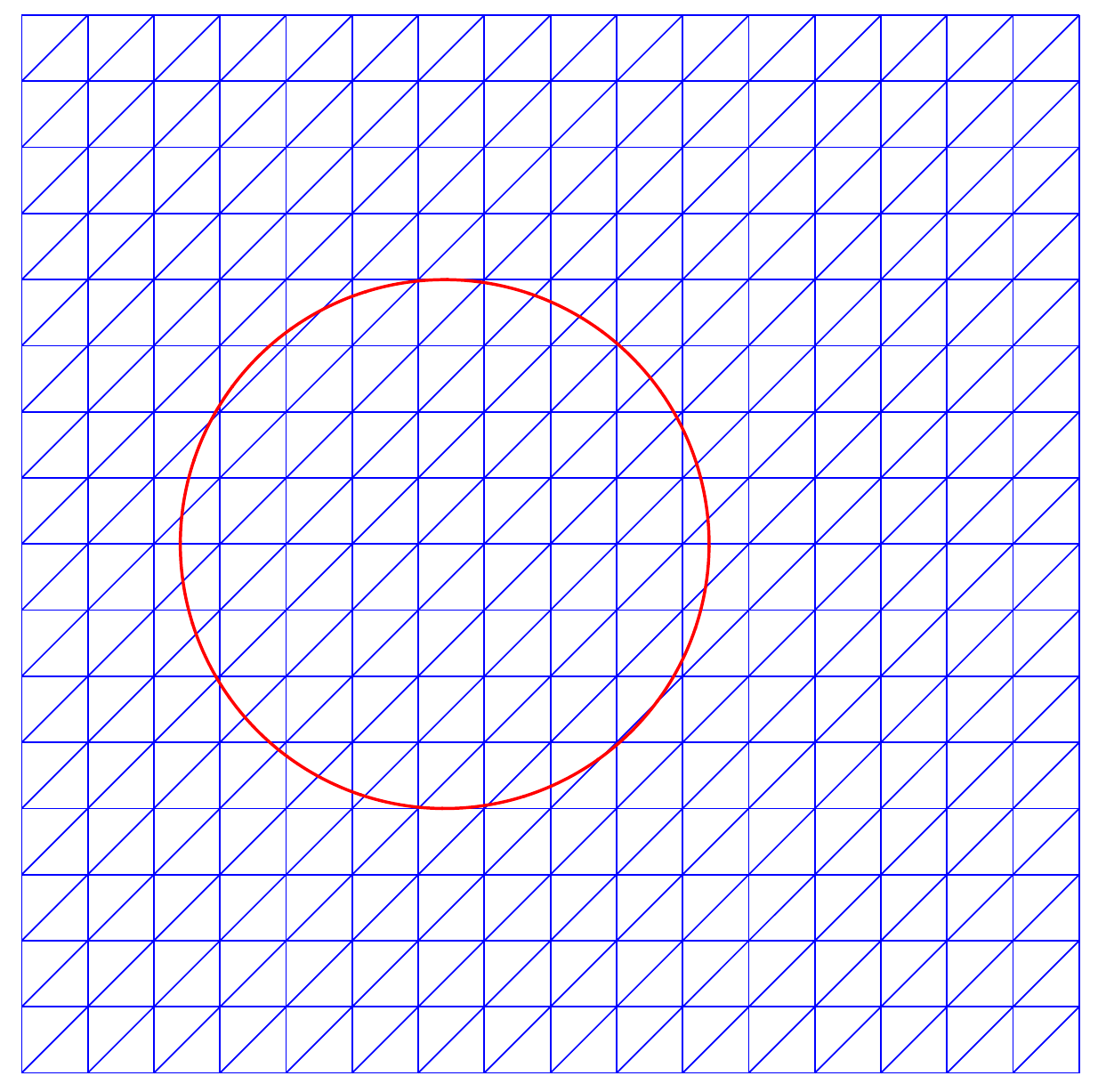}}
  \subfigure[The interface independent mesh]{
    \label{fig_unbodyfitting_mesh} 
    \includegraphics[width=1.2in]{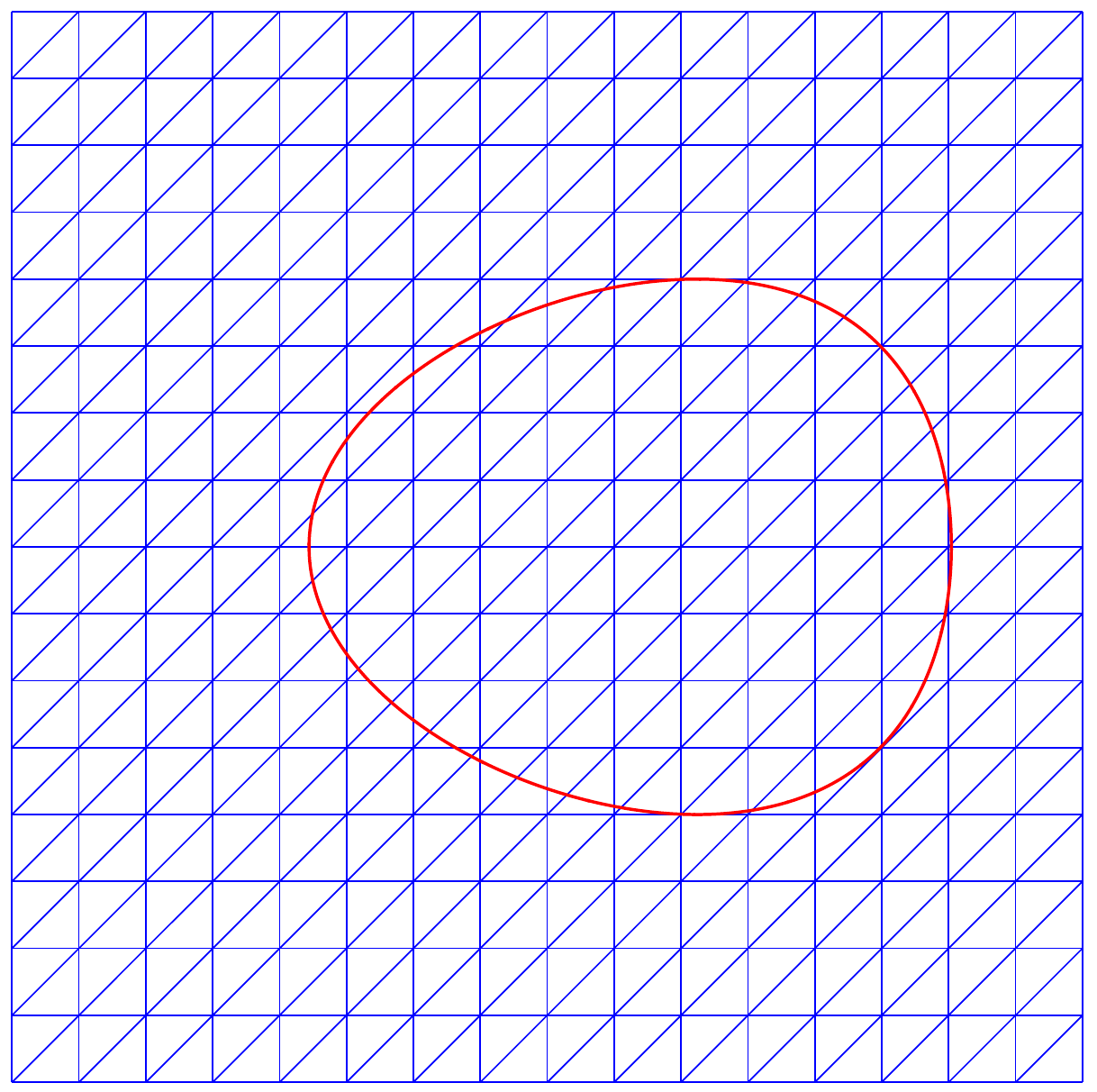}}
  \caption{The body-fitting and interface independent mesh }
  \label{fig:interface_independent_mesh} 
\end{figure}


The sensitivity analysis in a shape optimization is about the derivatives of an objective function with respect to the design variables, i.e., parameters describing the geometry of the domain \cite{1995Bendose,2003HaslingerMkinen}, and this leads to the gradient of the objective function which is a necessary ingredient in common numerical optimization algorithms such as descent direction methods and trust region methods \cite{J.E.Dennis_Robert.B.Schnabel_1996,J.Nocedal_S.Wright_2006}. A velocity field defined as derivatives of node coordinates with respect to the design variables \cite{1994ChoiChang,2004RdenasFuenmayor} is usually employed in the sensitivity analysis. Several approaches for computing the velocity in a Lagrangian framework are summarized in \cite{1994ChoiChang}, which either require computations to be carried out over the whole domain or need some special numerical methods for generating the velocity approximately.

\commentout{
First, the nodes in a body fitting mesh depend on the interface/shape parameters, but it is difficult, if not impossible, to find or construct a formula to describe such a dependence because the mesh is usually produced by an automatic mesh generator, and most, if not all, general mesh generators are not based on any explicit mathematical rules \cite{1994ChoiChang}, requiring such a formula to be differentiable is an even more excessive demand. Furthermore, under the assumption that each node of a body fitting mesh has to move in general as the shape/interface changes in the Lagrangian framework, the sensitivity has to be computed for the pertinent quantities, such as the local stiffness matrix and local load vector, on each elements in the mesh, and this inevitably leads to a global velocity field in the Lagrange framework and a high computational cost especially when the number of interface/design variables is large.

In the Lagrangian framework, the velocity field is highly related to mesh updating procedure in the shape optimization process. According to \cite{1994ChoiChang}, the velocity field needed in the sensitivity analysis should be employed to move the finite element mesh and it should have a few
desirable features for maintaining the topology and quality of the mesh. Indeed, {\color{red}it was pointed out in \cite{1994ChangChoi,1988HouSheen} that an inappropriate choice of velocity field for mesh update will result in a mesh that cannot guarantee the accuracy of the finite element solution}, see the illustration in Figure \ref{fig:body_fitting_mesh} which demonstrates that, with an improper choice of the velocity field, the elements near the right boundary have questionable quality. The authors in \cite{1994ChoiChang} summarized three basic approaches to compute a suitable velocity field: (1) finite difference methods, (2) isoparametric mapping methods and (3) boundary displacement and fictitious load methods. Except for the isoparametric mapping method, the other two lead to
an approximate velocity field to be globally computed over the whole solution domain. The isoparametric mapping method needs spacial decomposition of the computational domain which results in extra difficulties for complicated geometry, and it can not always guarantee mesh quality {\cite{1994ChoiChang}.}
}

Alternatively, the Eulerian approaches based on fixed meshes have been widely used in the shape optimization algorithms \cite{2002KimQuerin,2005KimChang}, which allow the material interface/boundary to cut elements as illustrated in the two plots on the right in Figure \ref{fig:interface_independent_mesh}.
For example, the extended finite element methods (XFEM) based on a fixed mesh are used to solve optimal design problems problems \cite{2012ZhangZhangZhu,2012SoghratiAragn,2008LuoWangWangWei,2015NajafiSafdari} and the inverse geometric problems related to crack detection
\cite{2013NanthakumarLahmer,2007RabinovichGivoli,2009WaismanChatzi}; the immersed interface methods (IIM) based on
a Cartesian mesh is employed to solve an cavity (rigid inclusion) detection problem in \cite{2001ItoKunischLi}.
Various techniques for improving the accuracy of the evaluation of either the stiffness matrix or sensitivity on those boundary/interface elements in Eulerian methods
have been discussed in \cite{2004AllaireJouveToader,2011DunningPeterKim,2005JangWonYoon}.

\commentout{
Solvers based on fixed meshes, such as the extended finite element methods (XFEM)
\cite{2013NanthakumarLahmer,2007RabinovichGivoli,2009WaismanChatzi,2012ZhangZhangZhu,2012SoghratiAragn,2008LuoWangWangWei,2015NajafiSafdari} and the immersed interface methods (IIM) \cite{2001ItoKunischLi}, have been employed to deal with the related forward problems in Eulerian methods.

based on fixed meshes are involved such as the extended finite element methods (XFEM) or the immersed interface methods (IIM). The application of XFEM can be found in \cite{2012ZhangZhangZhu,2012SoghratiAragn,2008LuoWangWangWei,2015NajafiSafdari} for optimal design problems and in \cite{2013NanthakumarLahmer,2007RabinovichGivoli,2009WaismanChatzi} for inverse geometric problems concentrated on crack detection. Besides, IIM was applied to an cavity (rigid inclusion) detection problem in \cite{2001ItoKunischLi}.
}

\commentout{
{\color{blue}In these methods, an optimization model is placed in a fictitious domain with a fixed mesh
whose edges are allowed to be cut by the boundary of this optimization model.} Various techniques are suggested to improve the accuracy of the evaluation of either the stiffness matrix or sensitivity on those boundary elements in \cite{2004AllaireJouveToader,2011DunningPeterKim,2005JangWonYoon}. In \cite{2010PengWangXing,2012ZhangZhangZhu}, a fixed mesh approach based on the extended finite element method(XFEM) \cite{2010FriesBelytschko,1999MoesDolbowBelytschko,2012SoghratiAragn} was proposed to handle either the boundary elements or interface elements. This idea was generalized in \cite{2015NajafiSafdari} to optimize the interface between multi-material using an interface-enriched finite element method (IGFEM) \cite{2012SoghratiAragn}, where the extra degrees of freedoms are added at the intersection points of the interface with mesh edges.
}

The goal of this article is to develop a fixed mesh method based on the partially penalized immersed finite element (PPIFE) method \cite{2015LinLinZhang} for solving the inverse geometric problems/optimal design problems described by \eqref{inter_prob_0}-\eqref{ObjFun_4_interf_optimiz}. A key motivation is that
IFE methods can solve interface forward problems with interface independent meshes optimally with respect to the degree of the involved polynomial spaces. In an IFE method, the interface shape/location and the jump conditions across the interface are utilized in the local IFE shape functions on interface elements, while the standard local finite element spaces are used on non-interface elements. Consequently, the convergence rates, optimal in the sense of the polynomials employed in the involved IFE spaces, have been established for IFE methods on interface independent meshes. We refer readers to IFE spaces constructed with linear polynomials \cite{2004LiLinLinRogers,2003LiLinWu}, with bilinear polynomial \cite{2008HeLinLin,2001LinLinRogersRyan}, and with rotated-$Q_1$ polynomials \cite{2016GuoLinZhang,2013ZhangTHESIS}. Applications of IFE methods to other types of equations or jump conditions can be found in \cite{2015AdjeridChaabaneLin,2011HeLinLin,2012LinZhang,2015ChenwuXiao}.

The advantages of the proposed IFE method for the inverse geometric problem are multifold. Because it is based on an interface independent fixed mesh in the shape optimization process, the issues caused by the mesh regeneration/movement, mesh distortion caused by large geometry changes, as well as some practical and theoretical issues for the construction of the velocity field \cite{1994ChoiChang} are circumvented, see the plots (c) and (d) in Figure \ref{fig:interface_independent_mesh} for an illustration. When the numerical interface curve $\Gamma$ is expressed as a parametric curve whose control points are the design variables for the shape optimization,
the fixed mesh used in the proposed IFE method allows us to develop velocity fields and shape derivatives of IFE shape functions that are advantageous for efficient implementation because they all vanish outside interface elements whose total number is of the order $O(h^{-1})$ on a shape regular mesh compared to the total number of elements in the order of $O(h^{-2})$. In addition, the formulas for the velocity fields and shape derivatives in the propose IFE method can be implemented
precisely without any further approximation procedures. Also, the IFE discretization for the related forward problem leads to an objective function that is optimal with respect to the polynomials employed in the underline finite element space regardless of
the location of the interface to be optimized. Furthermore, using the IFE discretization we are able to derive formulas for the gradient with respect to the design variables for the objective function and these formulas can be efficiently executed within the IFE framework. These benefits together with the fact that no need to remesh again and again in the shape optimization demonstrate a very strong potential of the proposed IFE method.

\commentout{
The accuracy of this objective function in this algorithm
is maintained optimal in this fixed mesh regardless of the location of the interface. We will derive formulas for
computing the gradient of this objective function and show that all of these formulas can be efficiently and accurately computed with the IFE
discretization on a chosen fixed mesh of $\Omega$.
}

\commentout{
{\color{blue}As for details in later sections}, using an IFE method in the shape optimization for solving the inverse geometric problem with a fixed mesh
is advantageous, below are some of its benefits.

\begin{itemize}
\item
Because the IFE method allows us to use an interface independent fixed mesh in the shape optimization process, the issues caused by the mesh regeneration and mesh distortion from large geometry changes are circumvented, see illustrations in Figure \ref{fig:interface_independent_mesh} compared with Figure \ref{fig:body_fitting_mesh}. With a fixed mesh, the velocity field does not need to guide the mesh movement; hence, some practical and theoretical issues for the construction of the velocity field \cite{1994ChoiChang} are simply avoided.

\begin{figure}[H]
\label{fig_mesh}
  \centering
  \subfigure{
    \label{fig_bodyfitting_mesh} 
    \includegraphics[width=2.5in]{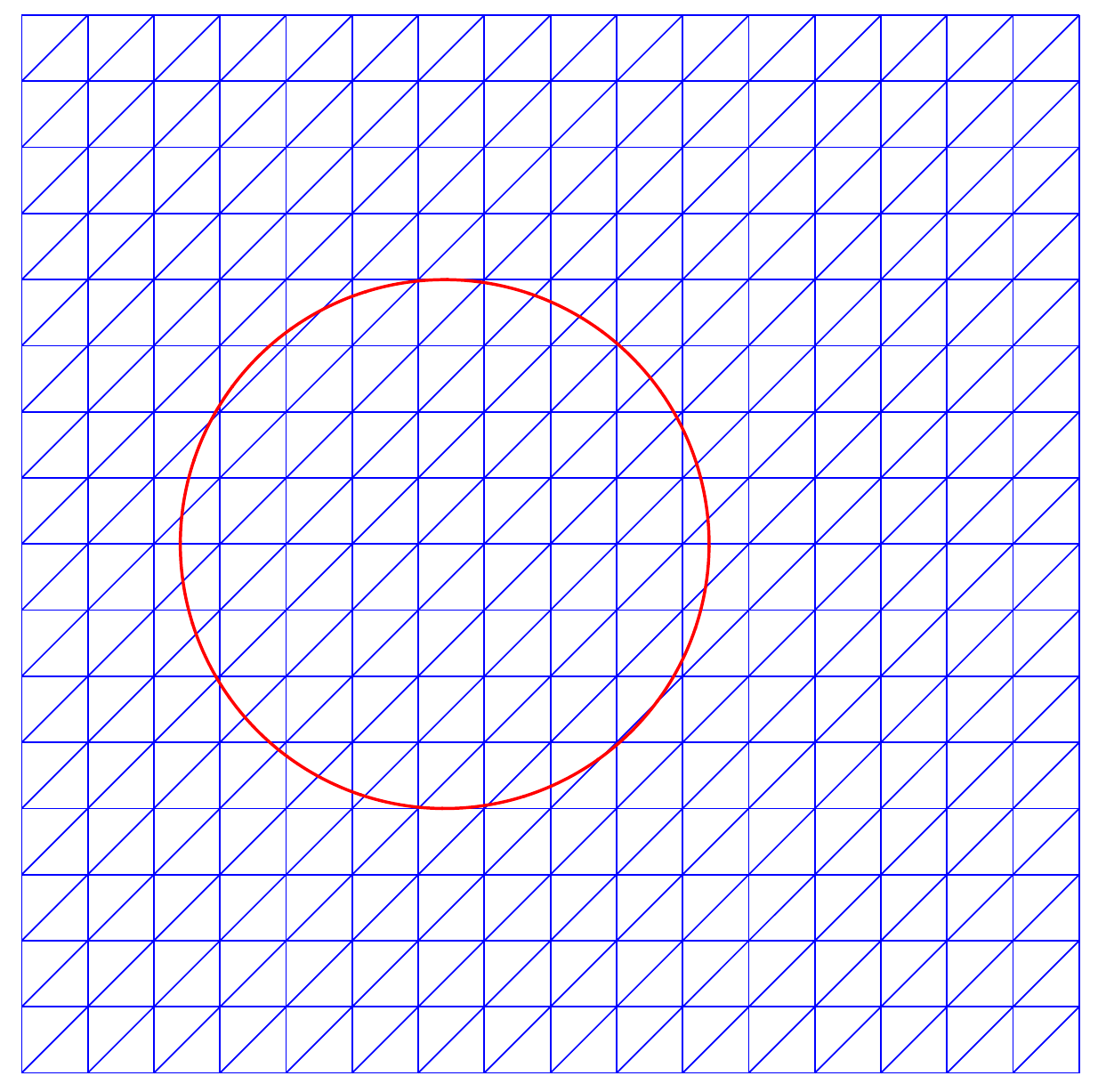}}
  \hspace{0.01in}
  \subfigure{
    \label{fig_unbodyfitting_mesh} 
    \includegraphics[width=2.5in]{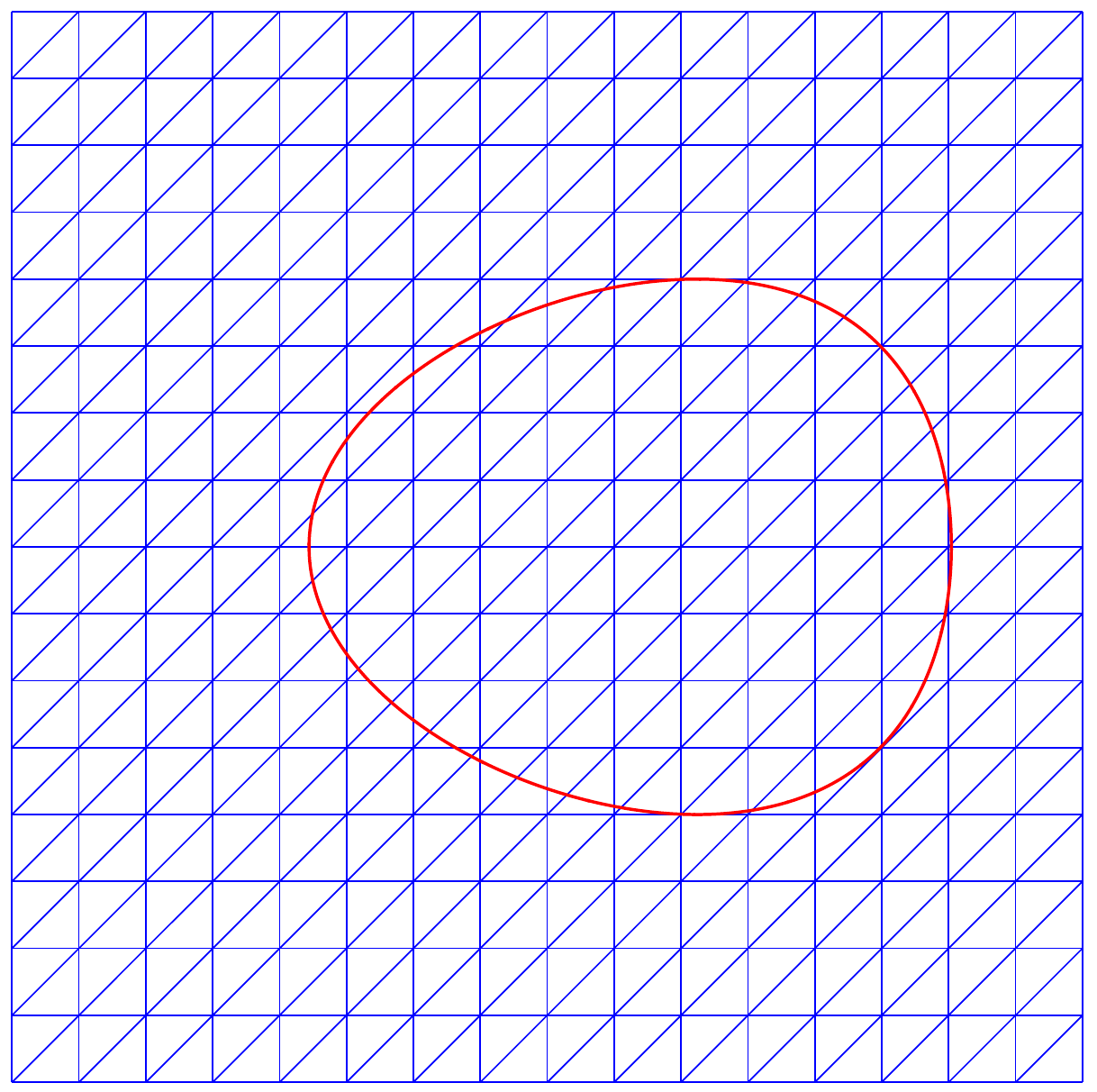}}
  \caption{The interface independent mesh }
  \label{fig:interface_independent_mesh} 
\end{figure}

\item
Using a fixed mesh in the shape optimization process admits a velocity field for the sensitivity that can be computed efficiently and accurately. With a fixed mesh, as the interface updates in the optimization, the interface-mesh intersection points are changing but all the mesh nodes are fixed. Consequently only the points in interface elements should be considered to move according to the interface, and this allows a velocity field that is zero outside interface elements. Moreover, we can derive a function that explicitly relates the points inside interface elements and the design variables of the shape optimization. Thus, the velocity field defined as the derivatives of this function can be computed efficiently by explicit formula rather than any numerical approximation to be globally carried out over the
whole domain. Also, the proposed velocity field has the $H^1$ regularity.

\item
Since the mesh is the same in one step and the next of the iterative optimization process, the IFE discretization of the governing equations \eqref{inter_prob_0} maintains the same algebraic structure because the number of the degrees of freedom and their locations in the domain do not change, and this feature
facilitates the derivation of computation procedures for sensitivity.
Moreover, assuming the interface does not evolve too drastically, the accuracy of an IFE discretization on a fixed mesh can always remain optimal regardless of
the location of the interface. Of course, no need to regenerate the mesh over and over again in the iterative optimization process helps to reduce the computational cost.

\item
Using an IFE method in the discretization of the shape optimization for solving the inverse geometric problem enables efficient and accurate computations for the
sensitivity. {\color{blue}First, almost all of the computations for the sensitivity of either the objective functional or the IFE discretization of the governing
interface forward problems are carried out only over interface elements because both the velocity field constructed in this framework and the shape derivatives of the IFE basis functions vanish over non-interface elements}. This feature drastically decreases the computational cost for sensitivity since the number of the interface elements is only $\mathcal{O}(h^{-1})$ compared with $\mathcal{O}(h^{-2})$ for the total number of elements. Furthermore, the proposed framework with an IFE method allows us to derive formulas for the calculations of all the derivatives with respect to the design variables in the sensitivity analysis either by the standard adjoint method and/or the shape derivatives of IFE basis functions.
\commentout{
Importantly, in this fixed mesh framework based on an IFE method, we have the formula for the shape derivatives of the IFE basis functions because they are locally decoupled on each interface element into the velocity field at interface-mesh intersection points and the derivatives of the IFE shape functions with respect to the coordinates of the interface-mesh intersection points for which the explicit formulas are readily available.
}
\end{itemize}

We note that the first two groups of advantages listed above are from the fact that an IFE method works with an interface independent fixed mesh and other fixed mesh methods may also have these advantages. Nevertheless, features in the last two groups are the consequences of the IFE formulation which allow us not only to
derive a discrete objective function in terms of the design variables that is an accurate approximation on a fixed mesh to the objective functional of the inverse geometric problem, but also to obtain formulas that can be efficiently executed in the IFE framework for accurately computing the gradient of this discrete objective function. As demonstrated by numerical examples for a set of representative inverse geometric problems in Section 4, these features can greatly benefit a successful
application of a typical numerical methods, especially those based on the quasi-Newton descent direction, to the minimization of the objective function in the proposed
algorithm for solving the inverse geometric problems.

}

This article is organized as follows. The next section recalls the linear IFE space and the related PPIFE scheme for the interface forward problems.
Section 3 presents the shape optimization algorithm based on the IFE discretization on a fixed mesh of $\Omega$ for the inverse geometric problem described by \eqref{inter_prob_0}-\eqref{ObjFun_4_interf_optimiz} and the computation procedure for its sensitivity. In Section 4, we demonstrate the strength and versatility of the proposed IFE method by applying it to three
representative interface inverse problems.

\commentout{
The third section is for the sensitivity computation. We start from derivation of the formulas for the derivatives of interface-mesh intersection points with respect to the design variables, which are used for either computing the shape derivatives of the IFE basis functions and constructing the non-zero \textit{velocity field} on interface elements. And to avoid theoretical difficulties, we discretize the equation \eqref{inter_prob_0} by the IFE spaces and then differentiate the system with respect to the design variables based on a discrete adjoint method \cite{2000GilesPierce,2007JamesonJameson}. Then we use the standard material derivative formula and the designed \textit{velocity field} to compute the derivatives of the local stiffness matrices and local edge penalty matrices with respect to the design variables. In the last section, we provide a bunch of numerical examples to test the efficiency and accuracy of the proposed algorithm on the different inverse and design problems, including the output-least-squares problems, the EIT problems and heat dissipation problems using different target curves.
}

\commentout{
{\color{red}
We can try to use these sentences later: \\

In this article, only recovering the location and geometry of the interface $\Gamma$ is of our interest; hence the values of $\beta^{i}$, $i=1,2$, are assumed to be apriori information. And Therefore here we shall assume

is a positive piecewise constant. And we utilize the shape optimization technique to reconstruct the interface where the discontinuities between coefficient $\beta$ occurs. For simplicity, in the discussion of the section \ref{review_ife} and \ref{sensitivity}, we only consider the case that $\Omega$ is partitioned by one interface into two sub-domains $\Omega^-$ and $\Omega^+$, but the proposed algorithm can be applied to any number of interfaces in implementation.

When the jump conditions \eqref{jump_cond_0} are explicitly imposed in related forward problem, the inverse geometric problem poses an additional challenging issue for traditional discretization methods of the involved PDEs because they usually require the mesh to be generated compatible to the discontinuities of the coefficients so that the mesh has be generated over an over again
in the optimization.

of the traditional discretization methods. In \cite{2001ItoKunischLi}, the interface was represented by a level set method

 and used the immersed interface method(IIM) on a Cartersian mesh to reconstruct the interface by the data available in a vicinity to the boundary. And a more challenging case is that the data is only available on whole or even part of the boundary. In \cite{1997HettlichRundell,1998HettlichRundell}, by only using the measurements around the boundary, the authors explicitly represented and solved the Fr\'echet derivatives through a transmission problem with specific interface jump conditions, which are used to for gradient computation in optimization process to reconstruct the interface and the support sets of the source term.

The parameterization of the interface and computation of the derivative of the objective functional $\mathcal{J}$ with respect to the geometry of the interface such as the shape and topology are two
main ingredients in the optimization of $\mathcal{J}$.

The IFE method can enable us to have these advantages in all the three steps for any shape optimization problems. Furthermore we shall emphasize that the proposed framework is a uniform algorithm for different objective functions in any application.
}
}

\section{An IFE Method for the Interface Forward Problems}
\label{review_ife}

In this section we recall the linear IFE method \cite{2016GuoLin,2004LiLinLinRogers} for the discretizatoin of the interface forward problem described by \eqref{inter_prob_0} and \eqref{jump_cond_0} with an interface independent mesh. The following notations will be used throughout this article. We let $\Gamma(t, \bfalpha), t \in [0,1]$ be a parametrization of the interface $\Gamma$ with design variables as entries in the vector $\bfalpha=(\alpha_i)_{i\in\mathcal{D}}$ where $\mathcal{D}$ is the index set of the chosen design variables. For example, when $\Gamma(t, \bfalpha)$ is a cubic spline, $\bfalpha$ is the vector of all the coordinates of control points \cite{2012Walter}. Let $\mathcal{T}_h$ be an interface independent triangular mesh of the domain $\Omega$. An element $T\in\mathcal{T}_h$ will be called an interface element if its interior intersects the interface $\Gamma(t, \bfalpha)$; otherwise we call it a non-interface element. Let $\mathcal{T}^i_h$ ($\mathcal{E}^i_h$) and $\mathcal{T}^n_h$ ($\mathcal{E}^n_h$) be the sets of interface and non-interface elements (edges), respectively. Denote the set of the interior interface edges by $\mathring{\mathcal{E}}^i_h$. And let $\mathcal{N}_h=\{ X_1, X_2,\cdots, X_{|\mathcal{N}_h|} \}$ and $\mathring{\mathcal{N}}_h$ be the sets of all the nodes in the mesh and the interior nodes, respectively.

For each element $T = \triangle A_1A_2A_3 \in\mathcal{T}_h$, we let $\mathcal{I}=\{1,2,3\}$ and let $\psi^{non}_{i,T}, i = 1, 2, 3$ be the standard linear shape functions \cite{2008BrennerScott} such that $\psi^{non}_{i,T}(A_j)=\delta_{ij}, ~i,j\in\mathcal{I}$. The local IFE space on each $T \in \mathcal{T}_h^n$ is
\begin{equation}
\label{non_inter_loc_space}
S_h(T)=\textrm{Span}\{\psi^{non}_{i,T},~i\in\mathcal{I} \}=\mathbb{P}_1.
\end{equation}
On an interface element $T = \triangle A_1A_2A_3\in\mathcal{T}^i_h$, we let $P=(x_P,y_P)^T$ and $Q=(x_Q,y_Q)^T$ be the two interface-mesh intersection points, and let $l$ be the line connecting $P$ and $Q$. The normal vector for the line $l$ is
$
\bar{\mathbf{ n}}=\frac{1}{||P-Q||}(y_P-y_Q,-(x_P-x_Q))^T
$
and the equation for the line $l$ is $L(X) = 0$ with
$
L(X)=\bar{\mathbf{ n}}\cdot(X-P).
$
The line $l$ cuts the element into two sub-elements $\overline{T}^+$ and $\overline{T}^-$, see the sketch on the left in Figure \ref{interface_element}, and we use them to introduce another two index sets $\mathcal{I}^-=\{i:A_i\in \overline{T}^-\}$ and $\mathcal{I}^+=\{i:A_i\in \overline{T}^+ \}$.
\begin{figure}[H]
\centering
\includegraphics[width=1.6in]{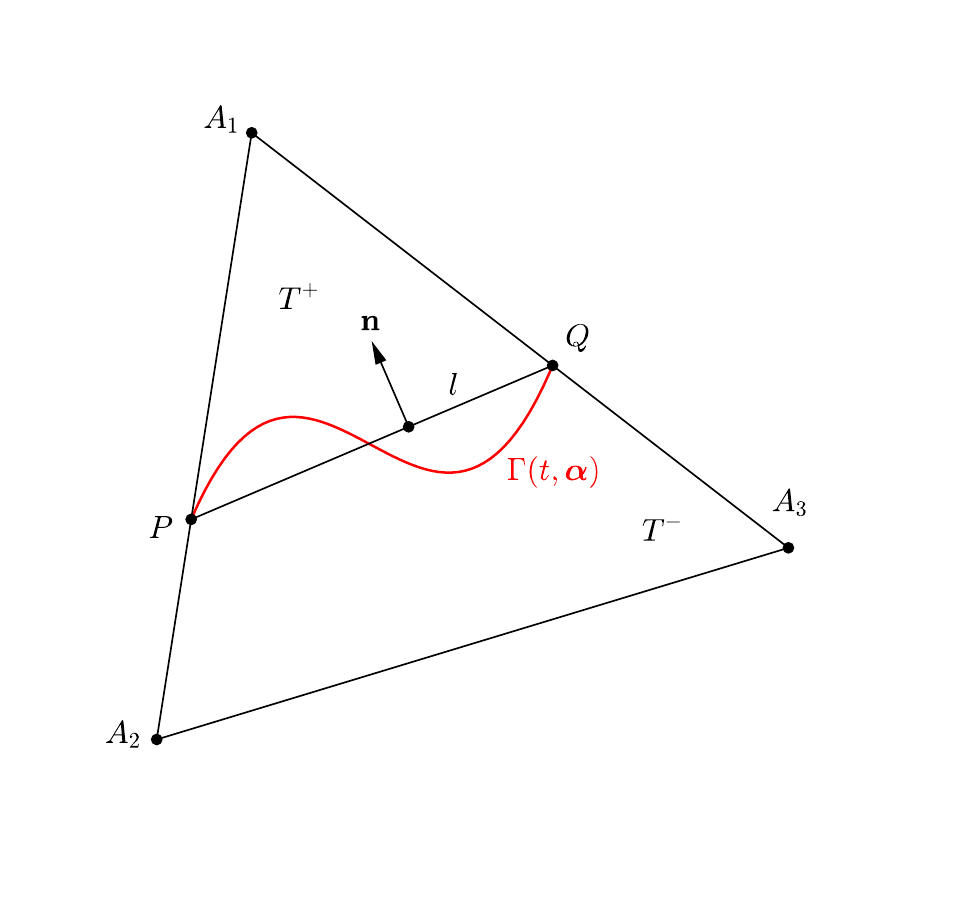}~~~~~~
\includegraphics[width=4.2in]{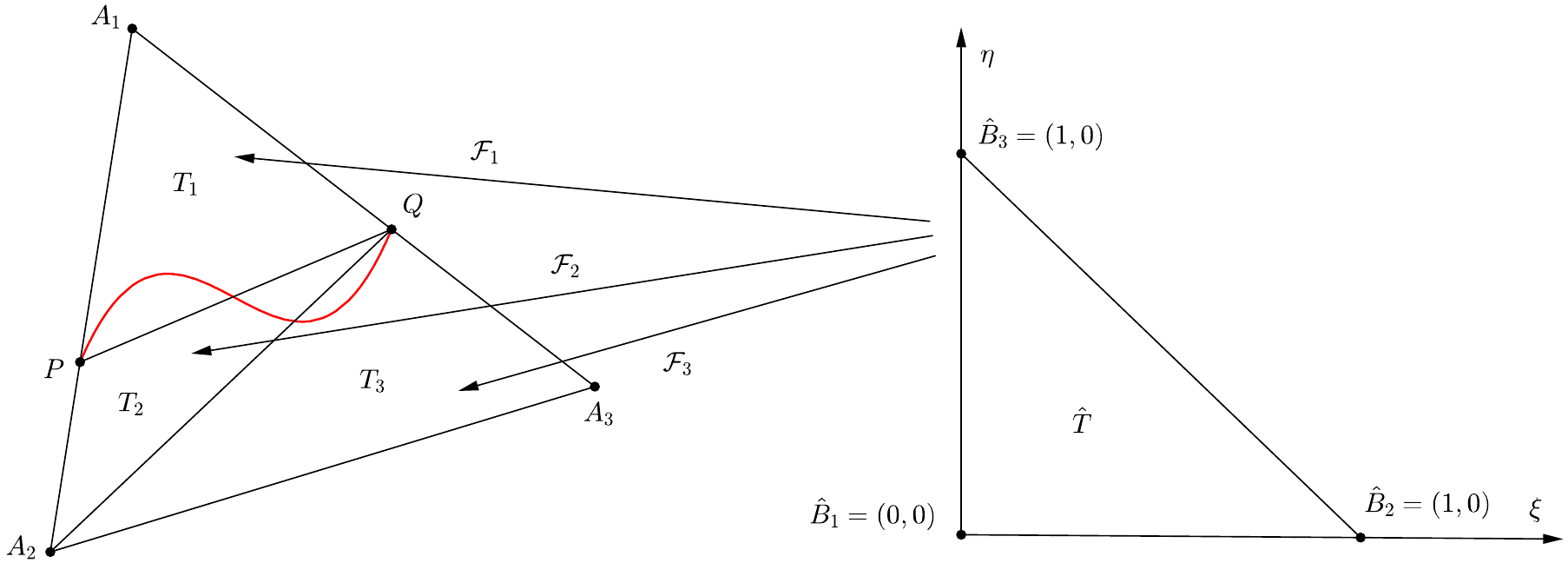}
\caption{ An interface element and its partitions.}
\label{interface_element}
\end{figure}
According to \cite{2016GuoLin}, the linear IFE function constructed according to the interface jump condition \eqref{jump_cond_0} and nodal values
$\mathbf{ v}=(v_1,v_2,v_3)$ has the following formula
\begin{align}
&\psi^{int}_T(X) =
\begin{cases}
 \psi^{int,-}_T(X)  = \psi^{int,+}_T(X)+c_0L(X) & \text{if} \;\; X\in \overline{T}^-, \\
 \psi^{int,+}_T(X)  = \sum_{i\in\mathcal{I}^-}c_i\psi^{non}_{i,T}(X)+\sum_{i\in\mathcal{I}^+}v_i\psi^{non}_{i,T}(X)& \text{if} \;\; X\in \overline{T}^+,
\end{cases} \label{ife_shapefun} \\
\text{with~~~}&c_0=
\left( \frac{\beta^+}{\beta^-}-1 \right) \left( \sum_{i\in\mathcal{I}^-}c_i\nabla\psi^{non}_{i,T}\cdot \bar{\mathbf{ n}}+\sum_{i\in\mathcal{I}^+}v_i\nabla\psi^{non}_{i,T}\cdot \bar{\mathbf{ n}} \right),
~~ ~~ \bfc = \bfb - \mu \frac{(\bfgamma^T \bfb)\bfdelta}{1 + \mu \bfgamma^T\bfdelta},    \label{c0} \\
\text{in which~~~} & \bfgamma = \left( \nabla\psi^{non}_{i,T}\cdot \bar{\mathbf{ n}} \right)_{i\in\mathcal{I}^-}, \bfdelta=\left( L(A_i) \right)_{i\in\mathcal{I}^-}
, \;\;\;
\bfb =
\left(
v_i- \mu L(A_i) \sum_{j\in\mathcal{I}^+} \nabla\psi^{non}_{j,T} \cdot \bar{\mathbf{ n}}~ v_j
 \right)_{i\in\mathcal{I}^-}. \label{gam_del}
\end{align}
Then, using $\mathbf{ v} = \mathbf{e}_i, i \in \mathcal{I}$, the standard basis vectors of $\mathbb{R}^3$ in \eqref{ife_shapefun}-\eqref{gam_del}, we obtain the IFE shape functions $\psi^{int}_{i,T}(X), i\in\mathcal{I}$ satisfying $\psi^{int}_{i,T}(A_j)=\delta_{ij}$ for $i,j\in\mathcal{I}$. The local IFE space on each $T \in \mathcal{T}_h^i$ is defined as
\begin{equation}
\label{ife_space_inter_element}
S_h(T)=\textrm{Span}\{ \psi^{int}_{i,T}\,:\, i\in\mathcal{I} \}.
\end{equation}

\commentout{
since the function $L$ and the normal vector $\bar{\mathbf{ n}}$ depend on the interface-mesh intersection points $P$, $Q$ and thus depend on the interface $\Gamma$, it is clear that the coefficients $(c_i)$ and $c_0$ depends on the design variable vectro $\bfalpha$. Moreover according to \eqref{ife_shapefun}, the dependence of the IFE shape functions on design variables is only through the coefficients $(c_i)$, $c_0$ and the function $L(X)$.
}

Local IFE spaces defined on all elements $\mathcal{T}_h$ are then used to define the IFE space globally as follows
\begin{align}
&S_h(\Omega)=\big\{ v\in L^2(\Omega):v|_{T}\in S_h(T); v|_{T_1}(A)=v|_{T_2}(A),~\forall A \in \mathcal{N}_h, \forall\,T_1, T_2 \in \mathcal{T}_h \text{~with~} A \in T_1\cap T_2 \big\}. \label{global_ife_space}
\end{align}
With this IFE space and its associated space $S^0_h(\Omega) = \left\{ v\in S_h(\Omega) ~:~ v(X)=0,~\forall X\in \mathcal{N}_h\cap\partial\Omega_D \right\}$, the interface forward problem \eqref{inter_prob_0} and \eqref{jump_cond_0} can be disretized by
the symmetric PPIFE (SPPIFE) method \cite{2015LinLinZhang} as follows:
find $u_h^k \in S_h(\Omega), k = 1, 2, \cdots, K$ such that
\commentout{
Because the $K$ interface problems in \eqref{inter_prob_0} and \eqref{jump_cond_0}
are essentially the same type, we consider a typical interface problem
described by \eqref{inter_prob_0} and \eqref{jump_cond_0} without the $k$ superscript. Then, its SPPIFE
discretization is to find $u_h\in S_h(\Omega)$ such that
}
\begin{equation}
\label{ppife_1}
\begin{split}
&a_h(u_h^k,v_h) = L_f^k(v_h),  ~~~~ \forall v_h \in S^0_h(\Omega), ~~u_h^k(X) =g_D^k(X), ~~~ \forall X\in \mathcal{N}_h\cap\partial\Omega_D^k,
\end{split}
\end{equation}
where the bilinear form $a_h$ and linear functional $L_f^k$ are given by
\begin{equation}
\begin{split}
\label{ppife_2}
a_h(u_h,v_h) =& \sum_{T\in\mathcal{T}_h} \int_T \beta \nabla u_h\cdot \nabla v_h dX - \sum_{e\in\mathcal{E}^{i}_h\backslash \partial \Omega_N^k} \int_e \{ \beta \nabla u_h \}_e \cdot [v_h]_e ds \\
& - \sum_{e\in\mathcal{E}^{i}_h\backslash \partial \Omega_N^k} \int_e \{ \beta \nabla v_h \}_e \cdot [u_h]_e ds + \sum_{e\in\mathcal{E}^{i}_h} \frac{\sigma^0_e}{|e|} \int_e [u_h]_e \cdot [v_h]_e ds,~~\forall u_h, v_h \in S_h(\Omega),
\end{split}
\end{equation}
\begin{equation}
\begin{split}
\label{ppife_3}
L_f^k(v_h)=& \int_{\Omega} f^k v_h dX +\int_{\partial \Omega_N^k}g_N^k v_h ds + \epsilon \sum_{e\in\mathcal{E}^{i}_h\cap\partial\Omega_D^k} \int_e \beta g_D^k \nabla v_h\cdot\mathbf{ n}_e  ds \\
& + \sum_{e\in\mathcal{E}^{i}_h\cap\partial\Omega_D^k} \frac{\sigma^0_e}{|e|} \int_e g_D^k v_h ds, ~~\forall v_h \in S_h(\Omega).
\end{split}
\end{equation}
In the bilinear form $a_h(\cdot, \cdot)$, the operators $[\cdot]_e$ and $\{ \cdot \}_e$ on each interior interface edge $e\in\mathring{\mathcal{E}}^i_h$ shared by
$T_1$ and $T_2$ are such that $[v]_e=(v|_{T_1}\mathbf{ n}^1_e + v|_{T_2}\mathbf{ n}^2_e)$, and $\{\beta\nabla v\}_e=\frac{1}{2}(\beta\nabla v|_{T_1} + \beta\nabla v|_{T_2}), ~\forall v\in S_h(\Omega)$,
where the normal vector $\mathbf{ n}^1_e=-\mathbf{ n}^2_e$ is from $T^1$ to $T^2$. For
$e\in \mathcal{E}^i_h\cap\partial\Omega$, we define the operators $[\cdot]_e$ and $\{ \cdot \}_e$ as
$[v]_e=v|_{T}\mathbf{ n}_e, ~\{\beta\nabla v\}_e=\beta\nabla v|_{T}, ~\forall v\in S_h(\Omega)$,
where $T$ is the element that contains $e$ and $\mathbf{ n}_e$ is the outward normal vector to $\partial\Omega$. In our applications, we choose $\sigma^0_E=10 \max\{ \beta^-, \beta^+\}$. It has been proven \cite{2015LinLinZhang} that the PPIFE solutions $u^k_h$ from \eqref{ppife_1} approximate the true solutions $u^k$, $1\leqslant k \leqslant K$, with an optimal accuracy with respect to the involved polynomials regardless of the interface location and shape, i.e.,
\begin{equation}
\label{optimal_accuracy}
\| u^k_h - u^k \|_{L^2(\Omega)} + h | u^k_h - u^k |_{H^1(\Omega)} \leqslant Ch^2 \| u \|_{H^2(\Omega)}.
\end{equation}

We now put the SPPIFE method described by \eqref{ppife_1}-\eqref{ppife_3} in the matrix form. We assume that $S_h(\Omega) = \textrm{Span}\{\phi_i(X) ~|~ X_i \in \mathcal{N}_h\}$
in which $\phi_i(X)$ is the global IFE basis function associated with the node $X_i \in \mathcal{N}_h$.
When the $k$-th ($1 \leq k \leq K$) interface forward problem has a mixed boundary condition, we let
$
\mathcal{N}^m_h=\{X_i ~|~ X_i\in\mathring{\mathcal{N}}_h\cup\partial\Omega_N \}
$
such that we can denote the SPPIFE solution $u_h^k(X)\in S_h(\Omega)$ determined by \eqref{ppife_1}-\eqref{ppife_3} as follows:
\begin{equation}
\label{ife_solu}
u_h^k(X)=\sum^{|\mathcal{N}^m_h|}_{i=1} u_i^k \phi_{i}(X) +\sum^{|\mathcal{N}_h|}_{i=|\mathcal{N}^m_h|+1}g_{D}^k(X_i)\phi_i(X),
\end{equation}
where, without loss of generality, we have assumed that nodes in $\mathcal{N}^m_h$ are ordered first.
The stiffness matrix
$\tilde{\mathbf{ A}} = (a_{i,j})_{i, j = 1}^{\abs{\mathcal{N}_h}}$ associated with the bilinear form defined in \eqref{ppife_2} can be assembled from the following local matrices on elements and edges of $\mathcal{T}_h$:
\begin{subequations}
\label{ppife_local_mat}
\begin{align}
\label{ppife_local_mat_1}
&\mathbf{ K}_T=\left( \int_T \beta \nabla \psi_{p,T} \cdot \nabla \psi_{q,T} dX \right)_{p, q\in\mathcal{I}},  ~~~~~~~~~~~~~~~~~~~ \forall T\in\mathcal{T}_h,   \\
\label{ppife_local_mat_2}
&\mathbf{ E}^{r_1r_2}_e=  \left(\int_e  \beta \nabla \psi_{p,T^{r_1}} \cdot  (\psi_{q,T^{r_2}} \mathbf{ n}^{r_2}_e) ds \right)_{p, q\in\mathcal{I} }, ~~~~~~~~~~ \forall e\in\mathcal{E}^i_h, \\
\label{ppife_local_mat_3}
&\mathbf{ G}^{r_1r_2}_e=  \left( \frac{\sigma^0_e}{|e|} \int_e (\psi_{p,T^{r_1}}\mathbf{ n}^{r_1}_e) \cdot (\psi_{q,T^{r_2}}\mathbf{ n}^{r_2}_e) ds \right)_{p, q\in\mathcal{I}}, ~~~~ \forall e\in\mathcal{E}^i_h,
\end{align}
\end{subequations}
where the index $r_1,r_2=1,2$ and the edge $e\in\mathring{\mathcal{E}}^i_h$ shared by the elements $T^1$ and $T^2$. But in the case $e\in\mathcal{E}^i_h\cap\partial\Omega$, we let $r_1=r_2=0$, $\mathbf{ n}^0_e=\mathbf{ n}_e$ is the outward normal vector and $T^0=T$ is the element that contains $e$.
Let
$
\tilde{\mathbf{ A}}_b^{m, k} = (a_{b, i}^k)_{i = 1}^{\abs{\mathcal{N}_h}} = \tilde{\mathbf{ A}} \begin{bmatrix}
\bf{0} &
\mathbf{ g}_D^k
\end{bmatrix}^T,
$
where $\bf{0}$ is the $\abs{\mathcal{N}^m_h}$-dimensional zero vector and
$\mathbf{ g}_D^k=(g_{D}^k(X_{|\mathcal{N}^m_h|+1}),\cdots,g_{D}^k(X_{|\mathcal{N}_h|}))^T$.
Similarly, the load vector $\tilde{\mathbf{ F}}^k = (f_i^k)_{i=1}^{\abs{\mathcal{N}_h}}$ associated with the linear form defined in \eqref{ppife_3}
can be assembled from the following vectors:
\begin{subequations}
\label{ppife_local_vec}
\begin{align}
\label{ppife_local_vec_1}
& \mathbf{ F}_T^k=\left( \int_T  f^k \psi_{p,T} dX \right)_{p\in\mathcal{I}},    &\forall T\in\mathcal{T}_h,  \\
\label{ppife_local_vec_2}
& \mathbf{ B}_e^k=  \left( \int_e \beta g_D^k \nabla \psi_{p,T} \cdot \mathbf{ n}_e ds \right)_{p\in\mathcal{I}},  ~~~ \mathbf{ C}_e^k=  \frac{\sigma^0_e}{|e|}\left( \int_e \beta g_D^k \psi_{p,T} ds \right)_{p \in\mathcal{I}},    &\forall e\in\mathcal{E}^i_h\cap\partial\Omega_D, \\
& \label{ppife_local_vec_Neu}
\mathbf{ N}_e^k=\left( \int_e  g_N^k \psi_{p,T} ds \right)_{p\in\mathcal{I}}, & \forall e\in \mathcal{E}^i_h\cap\partial\Omega_N.
\end{align}
\end{subequations}
Letting $\mathbf{ u}_h^{m, k} =(u_1^k,u_2^k,\cdots,u_{|\mathcal{N}^m_h|}^k)^T$, we can see that
the unknown coefficient vector $\mathbf{ u}_h^{m, k}$ of the SPPIFE solution $u_h^k(X)$ described by \eqref{ife_solu} is determined by the following linear system:
\begin{equation}
\label{ppife_mat_1}
\mathbf{ A}^{m, k}\mathbf{ u}_h^{m, k}=\mathbf{ F}^{m, k},
\end{equation}
where $\mathbf{ A}^{m, k} = (a_{i,j})_{i, j = 1}^{\abs{\mathcal{N}^m_h}}$,
$\mathbf{ F}^{m, k} = (f_i^k)_{i=1}^{\abs{\mathcal{N}^m_h}} - (a_{b,i}^k)_{i=1}^{\abs{\mathcal{N}^m_h}}$, and
the superscript $m$ in \eqref{ppife_mat_1} means that the boundary condition in the
$k$-th interface forward problem is of a mixed type.
\commentout{
Note that functions in the global space $S_h(\Omega)$ may not be continuous along the interface edges. So by adding penalties terms on the interface edges, the authors in \cite{2015LinLinZhang} proposed the partial penalty immersed finite element(PPIFE) method and proved the optimal convergence rate for the IFE solutions. Also the optimal convergence rate obtained is independent with the relative location of the interface, which is a critical feature for moving interface problems. Without loss of generality, here we only give the details for the mixed boundary value problem which also covers the pure Dirichlet boundary value problem and the discussion for the pure Neumann boundary problem is similar. To this end, we consider the sub set of the nodes set $\mathcal{N}_h$ given by $\mathcal{N}^m_h=\{X_i\in: X_i\in\mathring{\mathcal{N}}_h\cup\partial\Omega_N \}$. The mixed boundary problem based on the IFE discretization could be formulated as finding

For the convenience of presentation and without cause of any confusion, in the following discussion we only use $\psi_{j,T}$, $j\in\mathcal{I}$, to denote the local basis functions, i.e., $\psi_{j,T}$ can be $\psi^{int}_{j,T}$ or $\psi^{non}_{j,T}$ according to whether $T$ is an interface element or not. And let $\phi_i$ be the global shape functions corresponding to each node in the mesh, $i=1,2,\cdots,|\mathcal{N}_h|$. Then the finite element solution to \eqref{ppife_1} can be expressed as

where $\mathbf{ u}^m_h=(u_1,u_2,\cdots,u_{|\mathcal{N}^m_h|})^T$ and $\mathbf{ g}_D=(g_{D}(X_{|\mathcal{N}^m_h|+1}),\cdots,g_{D}(X_{|\mathcal{N}_h|}))^T$ satisfy the linear system

Here the matrices $\mathbf{ A}^m$ and $\mathbf{ A}^m_b$ are assembled by the following local matrices

Besides, the right hand side is assembled by the local vectors
}

When the $k$-th ($1 \leq k \leq K$) interface forward problem has Neumann boundary condition such that $\partial \Omega_N^k = \partial \Omega$,
we know that $\abs{\mathcal{N}_h^m} = \abs{\mathcal{N}_h}$, $u_h^k(X)$ given in \eqref{ife_solu} does not have the second term and the related load vector $\tilde{\mathbf{ F}}^k = (f_i^k)_{i=1}^{\abs{\mathcal{N}_h}}$ is assembled by the local vectors only in \eqref{ppife_local_vec_1} and \eqref{ppife_local_vec_Neu}. Since the solution to the interface problem is not unique, as a common practice, the normalization condition $\int_{\Omega} u^k dX = u_0^k$ is imposed such that the SPPIFE solution $u_h^k(X)$ described by \eqref{ife_solu} is determined by the following linear system:
\begin{align}
&\mathbf{ A}^{n}\mathbf{ u}_h^{n, k} = \mathbf{ F}^{n, k}, ~~~ \text{with}~~
\mathbf{ A}^{n} = \begin{bmatrix}
\tilde{\mathbf{A}} & \mathbf{ R} \\
\mathbf{R}^T & 0
\end{bmatrix},
~~\begin{cases}
\mathbf{ u}_h^{n, k} = [u_1^k, u_2^k, \cdots, u_{\abs{\mathcal{N}_h}}^k, \lambda]^T, \\
\mathbf{ F}^{n, k} =  [f_1^k, f_2^k, \cdots, f_{\abs{\mathcal{N}_h}}^k, u_0^k]^T,
\end{cases}  \label{ppife_mat_Ne}
\end{align}
where the superscript $n$ refers to a pure Neuman boundary condition, $\lambda$ is the Lagrange multiplier,
and $\mathbf{R}$ is the vector assembled with the following local vector constructed on each element:
\begin{equation}
\label{ppife_local_mat_Neu_1}
\mathbf{ R}_T=\left( \int_T \psi_{p,T} dX \right)_{p\in\mathcal{I}},~~\forall T \in \mathcal{T}_h.
\end{equation}

\commentout{
Need to be cleaned:

The matrix form of the PPIFE discretization for the interface forward problem is slightly different when its boundary is such that $\partial \Omega_N = \partial \Omega$, i.e., when it has a pure Neumann boundary condition. In this case, $u_h(X)$ given in \eqref{ife_solu} does not have the second term and
the related load vector $\tilde{F} = (f_i)_{i=1}^{\abs{\mathcal{N}_h}}$ is assembled by the local vectors defined in \eqref{ppife_local_vec_1} and \eqref{ppife_local_vec_Neu}, but local vectors given in
\eqref{ppife_local_vec_2} are not needed. In

Putting the coefficients of the IFE solution $u_h(X)$
in $\mathbf{ u}^n_h=(u_1,u_2,\cdots,u_{|\mathcal{N}_h|})^T$, we can see that the PPIFE discretization for
the interface forward problem has the following matrix form:

$\mathbf{ u}^n_h=(u_1,u_2,\cdots,u_{|\mathcal{N}_h|})^T$, $\mathbf{ F}^n$ is assembled by the local vectors \eqref{ppife_local_vec_1} and $\mathbf{ N}^n$ is assembled by the local vectors \eqref{ppife_local_vec_Neu} but on all $\partial\Omega_N$. For pure Neumann boundary value problem, the normalization condition
\begin{equation}
\label{normlization}
\int_{\Omega} u dX = u_0
\end{equation}
is needed to guarantee the uniqueness. So the last row of $\mathbf{ A}^n$ is generated by the local vectors

and correspondingly the last component of the whole right hand side is $u_0$.
}

In summary, according to \eqref{ppife_mat_1} and \eqref{ppife_mat_Ne}, the SPPIFE discretization for
the $K$ interface forward problems described in \eqref{inter_prob_0} and \eqref{jump_cond_0} can be written in the following unified matrix form:
\begin{align}
&\mathbf{ A}^k\mathbf{ u}_h^k = \mathbf{ F}^k, ~~\mathbf{ u}_h^k = \begin{cases}
\mathbf{ u}_h^{m, k} \\
\mathbf{ u}_h^{n, k}
\end{cases} \mathbf{ A}^k = \begin{cases}
\mathbf{ A}^m \\
\mathbf{ A}^n
\end{cases}\mathbf{ F}^k = \begin{cases}
\mathbf{ F}^{m, k} & \text{for a mixed boundary condition},  \\
\mathbf{ F}^{n, k} & \text{for a Neumann boundary condition}.
\end{cases} \label{ppife_mat_unified}
\end{align}
\commentout{
in which, depending on the boundary condition in the interface forward problem,
\begin{eqnarray*}
\mathbf{ u}_h^k = \begin{cases}
\mathbf{ u}_h^{m, k}, \\
\mathbf{ u}_h^{n, k}
\end{cases}~~~~\mathbf{ A}^k = \begin{cases}
\mathbf{ A}^m, \\
\mathbf{ A}^n,
\end{cases}~~~~\mathbf{ F}^k = \begin{cases}
\mathbf{ F}^{m, k}, & \text{for a mixed boundary condition},  \\
\mathbf{ F}^{n, k}, & \text{for a pure Neumann boundary condition}.
\end{cases}
\end{eqnarray*}
}
We note that the matrices $\mathbf{ A}^k$s in \eqref{ppife_mat_unified}
are symmetric positive definite and their size and algebraic structure remain the same as the interface $\Gamma(t, \bfalpha), t \in [0,1]$ evolves in a fixed mesh when the design variable $\bfalpha$ varies.

\section{An IFE Method for the Interface Inverse Problem}
\label{sensitivity}

We now discuss the discretization of the inverse geometric problem \eqref{interf_optimiz} subject to the governing equations \eqref{inter_prob_0} and \eqref{jump_cond_0} by the SPPIFE method on a fixed mesh. When the design variable $\bfalpha$ varies, the parametric interface $\Gamma = \Gamma(t, \bfalpha), t\in [0, 1]$ moves, and the two sub-domains $\Omega^-$ and $\Omega^+$ have to change their shapes correspondingly. Consequently, according to \cite{1994ChoiChang,2004RdenasFuenmayor}, the spacial variables $X \in \Omega$ is considered as a mapping from the design variables $\bfalpha$ to $\Omega$, i.e., $X=X(\bfalpha)$, of which the derivative $\frac{\partial X}{\partial \bfalpha}$ is the so called velocity field. This consideration implies that the local matrices \eqref{ppife_local_mat_1}-\eqref{ppife_local_mat_3}, local vectors
\eqref{ppife_local_vec_1}-\eqref{ppife_local_vec_Neu} and \eqref{ppife_local_mat_Neu_1} should be influenced
by this shape variation; hence, we write the matrix $\mathbf{A}^k$ and vector $\mathbf{F}^k$ in the SPPIFE equation \eqref{ppife_mat_unified} as
$\mathbf{A}^k = \mathbf{A}^k(X(\bfalpha),\bfalpha), ~~\mathbf{F}^k = \mathbf{F}^k(X(\bfalpha),\bfalpha),
~~k = 1, 2, \cdots, K$,
which further imply the solution $\bfu_h^k$ to the IFE equation \eqref{ppife_mat_unified} depends on $\alpha$ so we will denote it as $\bfu_h^k(\bfalpha)$ from now on.
Therefore, the IFE solution $u_h^k(X)$ to the $k$-th ($1 \leq k \leq K$) governing interface forward problem in the form of \eqref{ife_solu} depends on $\bfalpha$ through the IFE solution vector $\mathbf{ u}_h^k(\bfalpha)$, the spacial variable $X(\bfalpha)$, and the IFE basis functions as follows
\begin{equation}
\begin{split}
\label{ife_solu_alpha}
u_h^k(X) = u_h^k(\bfalpha) = u_h^k(\mathbf{ u}_h^k(\bfalpha),X(\bfalpha),\bfalpha) =& \sum^{|\mathcal{N}^m_h|}_{i=1} u_i^k(\bfalpha) \phi_{i}(X(\bfalpha), \bfalpha) +\sum^{|\mathcal{N}_h|}_{i=|\mathcal{N}^m_h|+1}g_{D}^k(X_i)\phi_i(X(\bfalpha), \bfalpha)
\end{split}
\end{equation}
where the second variable in $\phi_i(X(\bfalpha), \bfalpha)$ emphasizes the fact that $\bfalpha$ also effects the IFE solution
$u_h^k(X)$ through the coefficients of the IFE shape functions by the formulas \eqref{c0}-\eqref{gam_del}.

\commentout{
It has been proven in \cite{2015LinLinZhang} that the PPIFE solution $u_h^k \in S_h(\Omega)$ approximates the exact solution $u^k$ to the interface forward problem \eqref{inter_prob_0}-\eqref{jump_cond_0} with an optimal convergence rate regardless of the interface location in optimization iterations, and this is a critical feature the IFE method for the inverse geometric problem in which the interface is to be optimized with a fixed mesh.
}

The IFE solutions $u_h^k(X)\approx u^k(X),~1 \leq k \leq K$ naturally suggest the following discretization of the integrand in the objective functional defined by  \eqref{ObjFun_4_interf_optimiz}:
$$
J(u_h^1(\bfalpha), u_h^2(\bfalpha), \cdots, u_h^K(\bfalpha); X, \Gamma(\cdot, \bfalpha)) \approx
J(u^1(\bfalpha), u^2(\bfalpha), \cdots, u^K(\bfalpha); X, \Gamma(\cdot, \bfalpha)).
$$
Following explanations similar to those in the previous
paragraph, the design variable $\bfalpha$ can influence the approximate integrand $J(u_h^1(\bfalpha), u_h^2(\bfalpha), \cdots, u_h^K(\bfalpha); X, \Gamma(\cdot, \bfalpha))$ through
$\mathbf{ u}_h^k(\bfalpha), 1 \leq k \leq K$, $X(\bfalpha)$, and $\bfalpha$ itself; hence, we can denote these dependencies as
\begin{eqnarray}\label{discrt_obj_1}
J_h(\mathbf{ u}_h^1(\bfalpha), \mathbf{ u}_h^2(\bfalpha), \cdots, \mathbf{ u}_h^K(\bfalpha), X(\bfalpha),\bfalpha) \coloneqq J(u_h^1(\bfalpha), u_h^2(\bfalpha), \cdots, u_h^K(\bfalpha); X, \Gamma(\cdot, \bfalpha)).
\end{eqnarray}
Therefore, we propose an IFE method for solving the inverse geometric problem described in \eqref{inter_prob_0}-\eqref{ObjFun_4_interf_optimiz} on a fixed mesh of $\Omega$ by carrying out a shape optimization as follows: look for the design variable $\bfalpha^*$ that can minimize the following objective function
\begin{equation}
\begin{split}
\label{discrt_obj_2}
&\mathcal{J}_h (\mathbf{ u}_h^1(\bfalpha), \mathbf{ u}_h^2(\bfalpha), \cdots, \mathbf{ u}_h^K(\bfalpha), \bfalpha) := \int_{\Omega_0} J_h(\mathbf{ u}_h^1(\bfalpha), \mathbf{ u}_h^2(\bfalpha), \cdots, \mathbf{ u}_h^K(\bfalpha), X(\bfalpha),\bfalpha) dX, \\
\textrm{subject to}~~~~& \mathbf{ A}^k(X(\bfalpha), \bfalpha)\mathbf{ u}_h^k(\bfalpha) - \mathbf{ F}^k(X(\bfalpha), \bfalpha)=\mathbf{ 0},~~k = 1, 2, \cdots, K.
\end{split}
\end{equation}

The fact that the IFE solution $u_h^k(X),~1 \leq k \leq K$ is an optimal approximation to $u^k(X),~1 \leq k \leq K$ regardless of the location of the
interface $\Gamma(t, \bfalpha),~t \in [0, 1]$ to be optimized in a chosen fixed mesh \cite{2015LinLinZhang}
implies that the objective function $\mathcal{J}_h (\mathbf{ u}_h^1(\bfalpha), \mathbf{ u}_h^2(\bfalpha), \cdots, \mathbf{ u}_h^K(\bfalpha), \bfalpha)$
in this IFE method is an optimal approximation to the objective functional given in \eqref{ObjFun_4_interf_optimiz} regardless
of the interface location in a chosen fixed mesh for common inverse geometric problems such as those to be presented in Section \ref{sec:Applications}.
For example, when the shape functional $J(u^{1}, u^{2}, \cdots, u^{K}; X, \Gamma)$ is in the popular output-least-squares form such that:
\begin{equation}
\label{optimal_shape_fun_eq_1}
J(u^{1}, u^{2}, \cdots, u^{K};X,\Gamma) = \sum_{k=1}^K  | u^k - \bar{u}^k |^2,
\end{equation}
where $u^{1}, u^{2}, \cdots, u^{K}$ are the true solutions and $\bar{u}^{1}, \bar{u}^{2}, \cdots, \bar{u}^{K}$ are the given data functions in $L^2(\Omega)$, then, according to \eqref{discrt_obj_1} and \eqref{discrt_obj_2}, we have
\begin{equation}
\begin{split}
\label{optimal_shape_fun_eq_3}
& \left|\mathcal{J}(u^{1}, \cdots, u^{K}, \Gamma) - \mathcal{J}_h(\bfu_h^{1}, \cdots, \bfu_h^{K}, \Gamma)\right| \leqslant  \int_{\Omega} | J(u^{1}, \cdots, u^{K};X,\Gamma) - J(u^{1}_h, \cdots, u^{K}_h;X,\Gamma) | dX \\
 = &  \sum_{k=1}^K \int_{\Omega} | u^{k} - u^{k}_h | ~ | u^{k} + u^{k}_h -2\bar{u}^k | dX \\
 \leqslant &  \sum_{k=1}^K ( \| u^{k}\|_{L^2(\Omega)} + \|u^{k}_h\|_{L^2(\Omega)} + 2 \|\bar{u}^k \| _{L^2(\Omega)} ) \|  u^{k} - u^{k}_h \|_{L^2(\Omega)} \leqslant C ( \| u^k \|^2_{H^2(\Omega)} + \|\bar{u}^k \|^2_{L^2(\Omega)} ) h^2,
\end{split}
\end{equation}
where in the last step, we have applied the optimal estimation in the $L^2$ norm for the PPIFE solutions \cite{2015LinLinZhang}: $\|  u^{k} - u^{k}_h \|_{L^2(\Omega)} \leqslant Ch^2 \| u^k \|_{H^2(\Omega)}$ which also implies $\| u^k_h \|\leqslant C\| u^k \|_{H^2(\Omega)} $. Therefore, it holds that
\begin{equation}
\label{optimal_shape_fun_eq_2}
|\mathcal{J}(u^{1}, u^{2}, \cdots, u^{K}, \Gamma) - \mathcal{J}_h(\bfu_h^{1}, \bfu_h^{2}, \cdots, \bfu_h^{K}, \Gamma)| \leqslant C h^2
\end{equation}
in which the constant $C$ depend on the solutions $u^{k}$ and data $\bar{u}^{k}$, $1\leqslant k \leqslant K$.


 Furthermore, as discussed in the following subsections, we have formulas that can be efficiently executed in the IFE framework for accurately computing the gradient of this discrete objective function. These features are advantageous for implementing this IFE method with a typical numerical optimization algorithm, such as those based on the descent direction, for efficiently and accurately solving an
inverse geometric problem described by \eqref{inter_prob_0}-\eqref{ObjFun_4_interf_optimiz}.

\commentout{
Hence, by the IFE discretization, a discretized objective function $\mathcal{J}_h$ should encompass the influence of $\bfalpha$ also through the spacial variables $X$, the solution vector $\mathbf{ u}_h^k$ and all the rest including the coefficients \eqref{coef} in IFE shape functions. This leads to the following discretized form of
the objective functional given in \eqref{ObjFun_4_interf_optimiz}:

{\color{red}Not needed here: In this section, we derive a discretization for optimization problem \eqref{ObjFun_4_interf_optimiz} subject to the governing equations described by the interface forward problems \eqref{inter_prob_0} and \eqref{jump_cond_0} and then carry out the related sensitivity analysis. As we can see later in this section, the sensitivity analysis
for the discrete optimization problem critically depends on how the intersection points $P$ and $Q$ of the
interface $\Gamma$ and an interface element (see Figure \ref{interface_element}) change when the design variable $\bfalpha=(\alpha_i)_{i\in\mathcal{D}}$ of the parametrization of the interface $\Gamma$ varies.
With a fixed mesh, we show that the formulas for the derivatives of $P$ and $Q$ with respect to
$\bfalpha=(\alpha_i)_{i\in\mathcal{D}}$ can be derived. These formulas are then applied to
derive formulas for the velocity field and the shape derivatives of IFE shape functions. Ultimately,
all of these formulas enables us to efficiently and accurately compute the sensitivities on the fixed mesh through the standard discretized adjoint procedure {\color{red}[references ????]}.
}
}

\commentout{
where the constraint equation is described by \eqref{ppife_mat_unified} with its stiffness matrix
and load vector written as $\mathbf{ A}(\bfalpha)$ and $\mathbf{ F}(\bfalpha)$, respectively, to emphasize their dependence on the design variable $\bfalpha$.
}

\subsection{Velocity at Intersection Points} \label{sec:DerivativeIntersectionPoints}

By \eqref{ife_shapefun}, an IFE function on an interface element $T = \bigtriangleup A_1A_2A_3\in \mathcal{T}_h$ depends on the interface-mesh intersection points $P$ and $Q$, see the first sketch in Figure \ref{interface_element}. Obviously, points $P$ and $Q$
change their locations
when interface $\Gamma(t, \bfalpha), ~t \in [0, 1]$ evolves due to the change in the design variables $\bfalpha=(\alpha_i)_{i\in\mathcal{D}}$. Hence, the objective function in \eqref{discrt_obj_2} essentially depends on how the interface-mesh intersection points $P$ and $Q$ change when the design variable $\bfalpha$ varies such that the derivatives of $P$ and $Q$ with respect to $\bfalpha$ are critical ingredients for the sensitivity analysis
of the proposed IFE method, and this motivates us to derive their formulas in this subsection. According to \cite{2015NajafiSafdari}, these derivatives are the velocity defined at those intersection points and they will be used to develop the velocity field on the whole domain $\Omega$.

\commentout{
In general the velocity field is not unique because the functions relating the design variables $\bfalpha$ and spacial variables $X$ might be different according to different automatic mesh generator or different algorithms. Under these automatic mesh generators, usually every nodes in the mesh should be assumed to move as the interface changes and the dependence of the mesh nodes on the interface parameters is not based on any explicit mathematic rules. Thus it is even difficult to construct a function to relate the mesh nodes and design variables, let alone requiring the differentiability of the function.
}

\commentout{To be precise, we let $\Gamma(t,\bfalpha)=(x(t,\bfalpha),y(t,\bfalpha)),~t \in [0, 1]$ be the standard cubic spline parametrization of the interface curve and let
$T = \bigtriangleup A_1A_2A_3\in \mathcal{T}_h$ be an interface element such that
$A_i = (x_i, y_i)^T, i = 1, 2, 3$. }

Assume that $\Gamma(t,\bfalpha)$ intersects with the edge of $T$ at points $P = (x_P, y_P)$ and $Q = (x_Q, y_Q)$ corresponding to certain parameters $\hat{t}_P, \hat{t}_Q \in [0, 1]$, see the illustration in Figure \ref{interface_element}. Obviously, these two interface-mesh intersection points and their corresponding parameters $\hat{t}_P$ and $\hat{t}_Q$ all vary with respect to the design variable $\bfalpha$; hence, we can
express them as functions of $\bfalpha$ as follows:
\begin{align*}
&P = P(\bfalpha) = (x_P, y_P) = (x(\hat{t}_P), y(\hat{t}_P)) = (x(\hat{t}_P(\bfalpha),\bfalpha), y(\hat{t}_P(\bfalpha),\bfalpha)), \\
&Q = Q(\bfalpha) = (x_Q, y_Q) = (x(\hat{t}_Q), y(\hat{t}_Q)) =  (x(\hat{t}_Q(\bfalpha),\bfalpha), y(\hat{t}_Q(\bfalpha),\bfalpha)).
\end{align*}
In the following discussions, we use $D_{\alpha_j}$ to denote the total derivative operator with respect to the $j$-th design variable $\alpha_j$, $j\in\mathcal{D}$, and $D_{\bfalpha}$ is the corresponding gradient operator. But we use $\frac{\partial}{\partial \alpha_j}$ and $\frac{\partial}{\partial \bfalpha}$
to denote the standard partial differential operators and the gradient operator with respect to $\alpha_j$ and $\bfalpha$.

Without loss of generality, we assume that the interface-mesh intersection points
are such that $P \in \overline{A_1A_2}$ and $Q \in \overline{A_1A_3}$ as illustrated in Figure \ref{interface_element}. Then the following lemma establishes explicit formulas for computing the total derivatives of interface-mesh intersection points with respect to $\bfalpha$.

\commentout{
{\color{red}write the following into two lemmas??? A lemma for P, and use some statements to describe results for Q plus some comments on why the results depend only on the edge ...}
}
\begin{lemma}
\label{inter_pt_deri}
Assume $\Gamma(t, \bfalpha)$ is not tangent to $A_1A_2$ at $P$. Then the function $P = P(\hat{t}_P(\bfalpha), \bfalpha)$ is differentiable and
its velocity defined as the total derivatives $D_{\alpha_j}P$ with respect to $\alpha_j$, $j\in\mathcal{D}$ are determined by the following linear system:
\begin{align}
&M_P(\hat{t}_P)  ~ D_{\alpha_j}P =  b_{P, j}(\hat{t}_P),~~\forall ~j \in \mathcal{D}, \label{P_deri_4} \\
\text{with~~~~} & M_P(\hat{t}_P)= \begin{bmatrix}
y_2-y_1 & -(x_2-x_1) \\
 \frac{\partial y}{\partial t}(\hat{t}_P) & -\frac{\partial x}{\partial t} (\hat{t}_P)
\end{bmatrix} ~~\textrm{and}~~
b_{P,j}(\hat{t}_P)=
\begin{bmatrix}
0 \\
\frac{\partial y}{\partial t}(\hat{t}_P)\frac{\partial x}{\partial \alpha_j}(\hat{t}_P) -
 \frac{\partial x}{\partial t}(\hat{t}_P)\frac{\partial y}{\partial \alpha_j}(\hat{t}_P)
\end{bmatrix}. \nonumber
\end{align}
\end{lemma}
\begin{proof}
First, differentiating $x_P=x(\hat{t}_P(\bfalpha),\bfalpha)$ and $y_P=y(\hat{t}_P(\bfalpha),\bfalpha)$ with respect to $\alpha_j$, we have
$
D_{\alpha_j}x_P= \frac{\partial x}{\partial t} \frac{\partial \hat{t}_P}{ \partial \alpha_j } +\frac{\partial x}{\partial \alpha_j},~
D_{\alpha_j}y_P= \frac{\partial y}{\partial t} \frac{\partial \hat{t}_P}{ \partial \alpha_j } +\frac{\partial y}{\partial \alpha_j},
$
which leads to
\begin{equation}
\label{P_deri_3}
 \frac{\partial y}{\partial t} D_{\alpha_j}x_P - \frac{\partial x}{\partial t} D_{\alpha_j}y_P = \frac{\partial y}{\partial t}\frac{\partial x}{\partial \alpha_j} -
 \frac{\partial x}{\partial t}\frac{\partial y}{\partial \alpha_j}.
\end{equation}
On the other hand, since $P$ is on the edge $A_1A_2$, we have the equation
$
(y_2-y_1)x_P-(x_2-x_1)y_P=x_2y_1-x_1y_2.
$
Differentiating it with respect to $\alpha_j$ yields
\begin{equation}
\label{P_deri_2}
(y_2-y_1)D_{\alpha_j}x_P -(x_2-x_1)D_{\alpha_j}y_P=0.
\end{equation}
Combining \eqref{P_deri_2} and \eqref{P_deri_3} yields the linear system for $D_{\alpha_j}P$ in
\eqref{P_deri_4}.
Let $\mathbf{ n}_e$ be the normal vector to the edge $A_1A_2$. Then we have
$
\textrm{det}(M_P(\hat{t}_P))=\mathbf{ n}_e\cdot \nabla\Gamma(\hat{t}_P(\bfalpha), \bfalpha)
$
which is non zero by the assumption that $A_1A_2$ is not tangent to $\Gamma(t, \bfalpha)$ at $P$.
\end{proof}

Similar results hold for the interface-mesh intersection point $Q$. Assume that
$\Gamma(t, \bfalpha)$ is not tangent to $A_1A_3$ at $Q$, then the function $Q = Q(\hat{t}_Q(\bfalpha), \bfalpha)$ is differentiable and for $j\in\mathcal{D}$, $D_{\alpha_j}Q$ is determined by
\begin{align}
&M_Q(\hat{t}_Q)  ~ D_{\alpha_j}Q =  b_{Q, j}(\hat{t}_Q),~~\forall ~j \in \mathcal{D}, \label{Q_deri_4}\\
\text{with~~} &M_Q(\hat{t}_Q)= \begin{bmatrix}
y_3-y_1 & -(x_3-x_1) \\
 \frac{\partial y}{\partial t}(\hat{t}_Q) & -\frac{\partial x}{\partial t} (\hat{t}_Q)
\end{bmatrix} ~~~ \textrm{and} ~~~
b_{Q,j}(\hat{t}_Q)= \begin{bmatrix}
0 \\
\frac{\partial y}{\partial t}(\hat{t}_Q)\frac{\partial x}{\partial \alpha_j}(\hat{t}_Q) -
\frac{\partial x}{\partial t}(\hat{t}_Q)\frac{\partial y}{\partial \alpha_j}(\hat{t}_Q)
\end{bmatrix}. \nonumber
\end{align}


\commentout{
Formulas \eqref{P_deri_4} and \eqref{Q_deri_4} allow us to compute material derivatives of the interface-mesh intersection points since the edge tangent to the interface is not considered as the interface edge. In addition, the case in which the interface is tangent to the edge of an element actually rarely happens in computation.
}

Note that $\partial x/\partial t, \partial y/\partial t, \partial x/\partial \alpha_j$ and $\partial y/\partial \alpha_j$ required in
formulas \eqref{P_deri_4} and \eqref{Q_deri_4} depend on the chosen parametrization for the interface $\Gamma(t, \bfalpha)$, but they are usually
easy to derive by standard calculus procedures.

\commentout{
Note that formulas \eqref{P_deri_4} and \eqref{Q_deri_4} allow us to compute
$D_{\alpha_j}P$ and $D_{\alpha_j}Q$ without using $\frac{\partial \hat{t}}{ \partial \alpha_j }$
which is usually cumbersome because the dependence of $\hat{t}$ on $\bfalpha$ involves the inverse function
of the parametrization $\Gamma(t, \bfalpha)$ for the interface.

depend on the specific parameterization of $\Gamma$. For example,
let $\bfalpha = (\bfalpha_x, \bfalpha_y)$ be such that $\bfalpha_x$ and $\bfalpha_y$ are the vectors for the $x$ and $y$ coordinates of control points,
and let $\Gamma(t, \bfalpha) = (x(t, \bfc_x(\bfalpha_x)), y(t, \bfc_y(\bfalpha_y))), ~t = [0, 1]$ be a parametric cubic spline whose
coefficient vectors are determined by two linear systems $\mathbf{ W}_s\bfc_s=\mathbf{ r}_s(\bfalpha_s)$, $s=x,y$ \cite{2012Walter}. We note that matrices $\mathbf{ W}_s, s = x, y$ are independent of $\bfalpha$ and vectors $\mathbf{ r}_s(\bfalpha_s)$ are simple polynomials of $\bfalpha_s, s = x, y$. Then we can solve for $\frac{\partial\bfc_x}{\partial\alpha_j}$ from $\mathbf{ W}_x\frac{\partial\bfc_x}{\partial\alpha_j}=\frac{\partial\mathbf{ r}_x(\bfalpha_x)}{\partial \alpha_j}$,
with $1 \leq j \leq dim(\bfalpha_x)$ and solve for $\frac{\partial\bfc_y}{\partial\alpha_j}$ from
$
\mathbf{ W}_y\frac{\partial\bfc_y}{\partial\alpha_j}=\frac{\partial\mathbf{ r}_y(\bfalpha_y)}{\partial\alpha_j}
$
with $dim(\bfalpha_x) + 1 \leq j \leq dim(\bfalpha_x) + dim(\bfalpha_y)$. Then, we obtain $\partial x/\partial \alpha_j = x(t, \frac{\partial\bfc_x}{\partial\alpha_j})$ and $\partial y/\partial \alpha_j = y(t, \frac{\partial\bfc_y}{\partial\alpha_j})$.
}

\commentout{
Let $x = x(t, \bfc_x)$ and $y = y(t, \bfc_y)$ be the cubic splines
interpolating vector $\bfalpha_x$ and $\bfalpha_y$, respectively, at the sampling points with
$\bfc_x$ and $\bfc_y$ being the coefficients of these two cubic splines. As usual, $\bfalpha_x$ and $\bfalpha_y$ are the $x$ and $y$ coordinates of the control points for this cubic spline parameterization. According to {\color{red}\cite{2012Walter}},
vectors $\bfc_x$ and $\bfc_y$ are determined by two linear systems $\mathbf{ W}_x\bfc_x=\mathbf{ r}_x(\bfalpha_x)$ and $\mathbf{ W}_y\bfc_y=\mathbf{ r}_y(\bfalpha_y)$
in which the matrices $\mathbf{ W}_x$ and $\mathbf{ W}_y$ are independent of the
design variable $\bfalpha = (\bfalpha_x, \bfalpha_y)^T$. Differentiating these systems with respect to $\alpha_j$ leads to
\begin{subequations}
\label{spline_coe_deri}
\begin{align}
&\mathbf{ W}_x\frac{\partial\bfc_x}{\partial\alpha_j}=\frac{\partial\mathbf{ r}_x(\bfalpha_x)}{\alpha_j}, ~~\text{if~~} 1 \leq j \leq dim(\bfalpha_x) \label{spline_coe_1} \\
&\mathbf{ W}_y\frac{\partial\bfc_y}{\partial\alpha_j}=\frac{\partial\mathbf{ r}_y(\bfalpha_y)}{\partial\alpha_j}, ~~
\text{if~~} dim(\bfalpha_x) + 1 \leq j \leq dim(\bfalpha_x) + dim(\bfalpha_y), \label{spline_coe_2}
\end{align}
\end{subequations}
from which we can solve \eqref{spline_coe_deri} for $\frac{\partial\bfc_x}{\partial\alpha_j}$ and
$\frac{\partial\bfc_y}{\partial\alpha_j}$ to obtain
$\partial x/\partial \alpha_j = x(t, \frac{\partial\bfc_x}{\partial\alpha_j})$ and $\partial y/\partial \alpha_j = y(t, \frac{\partial\bfc_y}{\partial\alpha_j})$.
}
\commentout{
\begin{rem}
\label{solu_of_Mp}
Actually Lemma \ref{inter_pt_deri} indicates that the formula for the
\end{rem}
}

\commentout{
{\color{red}Add some comments stating that the results above are valid for general parametrizations. Given specific descriptions about how to compute
$$\frac{\partial x}{\partial t}, \frac{\partial y}{\partial t}, \frac{\partial x}{\partial \alpha_j}, \frac{\partial y}{\partial \alpha_j}$$
for using a cubic spline
}
}


\begin{rem}
\label{P_A1A2}
Let $\mathbf{ n}_e$ be the normal of $\overline{A_1A_2}$, then we can directly verify that
$
D_{\alpha_j}P\cdot \mathbf{ n}_e =\mathbf{ n}_e M_P^{-1}b_P =0
$
which means that $D_{\alpha_j}P$ is parallel to the edge $A_1A_2$ for any $\alpha_j$, $j\in\mathcal{D}$. Geometrically, this property implies that every intersection point can only move along the corresponding interface edge.
\end{rem}

\subsection{A Velocity Field for Sensitive Computations}\label{sec:VelocityField}

In the inverse geometric problem, the two sub-domains $\Omega^-$ and $\Omega^+$ separated from each other by the interface $\Gamma = \Gamma(t, \bfalpha), t\in [0, 1]$ change their shapes when the parametric interface moves because of a variation in the design variable $\bfalpha$. Hence, $\Omega^-$ and $\Omega^+$ can be considered as functions of $\bfalpha$. Consequently, since $\Omega^- \cup \Omega^+ = \Omega \backslash \Gamma$, we can consider the spacial variable $X \in \Omega$ as a mapping from the design variables $\bfalpha$ to $\Omega$, i.e., $X=X(\bfalpha)$, and its derivative $\frac{\partial X}{\partial \bfalpha}$ is the so called velocity field \cite{2004RdenasFuenmayor}, a key ingredient in the sensitivity analysis in shape optimizations. Therefore, in this subsection, we develop and analyze a velocity field for the IFE-based shape optimization to solve the inverse geometric problem.
\commentout{
{\color{blue}and naturally, on each interface element, this velocity field depends on the velocity at the interface-mesh intersection points derived in the last sub-subsection.}
}

Since the IFE method proposed in \eqref{discrt_obj_2} is based on a fixed interface independent mesh, all the points located in non-interface elements can be considered as constant functions of the design variable $\bfalpha$. Therefore, on such a fixed mesh used by the proposed IFE method, the velocity field vanish on all non-interface elements because $\frac{\partial X}{\partial \bfalpha}=0$, and this suggests we need to discuss the velocity field only on interface elements.

\commentout{
{\color{red}Moved here from 3.1:
Naturally, our intention is to solve the discrete optimization problems \eqref{discrt_obj_2} numerically
by gradient-based methods such as the {\color{red}XXX method and the XXX [???]}, and a key ingredient in computing the gradient of its objective function with respect to the design variable $\bfalpha$ is the velocity field. Since the IFE method provides an accurate discretization for the interface forward problems on a fixed interface independent mesh, all the points located in non-interface elements can be considered as constant functions of the design $\bfalpha$, and thus, do not move as the interface changes {\color{red}[an references???]}. Therefore, on such a fixed mesh used by the IFE discretization, the velocity field vanish on all non-interface elements because $\frac{\partial X}{\partial \bfalpha}=0$, and this implies that non-zero velocity field only needs to be considered on interface elements. Furthermore, with the IFE discretization on a fixed mesh, the dependence of $X$ in every interface element on the design variables $\bfalpha$ is through the interface-mesh intersection points, explicit formulas of $X$ in terms of $\bfalpha$ is readily available and these formulas are indeed differentiable. Hence, to construct the velocity field needed for
solving the discrete optimization problems \eqref{discrt_obj_2} numerically, we derive the formulas for the derivatives of the coordinates of the interface-mesh intersection points with respect to design $\bfalpha$.
Furthermore, we note that this fundamental result enables us to derive the formula for the shape derivatives of the IFE shape functions.

The basic results in the previous subsection enables us to derive formulas that define a velocity field whose support is the union of all the interface elements, i.e., this velocity field is zero outside interface elements. This velocity field is defined piecewise on elements of an interface independent mesh $\mathcal{T}_h$ and it has an $H^1$ regularity.
}
}
As before, we consider a typical interface element $T = \bigtriangleup A_1A_2A_3 \in \mathcal{T}_h^i$, without loss of generality, we assume that the parameterized interface $\Gamma(t,\bfalpha)=(x(t,\bfalpha),y(t,\bfalpha)),~t \in [0, 1]$ intersects with $T$ at $P(\bfalpha) \in \overline{A_1A_2}$ and $Q(\bfalpha) \in \overline{A_1A_3}$, see the first sketch in Figure \ref{interface_element}, but neither $P$ nor $Q$ coincides with vertices of $T$. All results derived from now on are readily extended to the case in which one of the interface-mesh intersection points $P$ and $Q$ is a vertex of $T$.

Inspired by the ideas from \cite{2003HaslingerMkinen,2015NajafiSafdari}, we partition $T$ into three sub-elements as follows:
$
T_1 = \bigtriangleup A_1PQ, ~~T_2 = \bigtriangleup A_2QP, ~~T_3 = \bigtriangleup A_3QA_2,
$
and let $\hat{T} = \bigtriangleup \hat{B}_1\hat{B}_2\hat{B}_3$ be the usual reference element with vertices
$
\hat{B}_1 = (0, 0)^T, ~~\hat{B}_2 = (1, 0)^T, ~~\hat{B}_1 = (0, 1)^T,
$
see the 2nd and the 3rd sketches in Figure \ref{interface_element}. Then, the standard affine mappings from the reference element $\hat{T} = \bigtriangleup \hat{B}_1\hat{B}_2\hat{B}_3$ to $T_m, m = 1, 2, 3$ provide a relation between the points in
$T$ and the design variable $\bfalpha$ as follows:
\begin{equation}
\label{aff_map}
X(\bfalpha)=\mathcal{F}_m(\bfalpha,\xi,\eta)=\mathbf{ J}_m(\bfalpha)\left(\begin{array}{c} \xi \\ \eta \end{array}\right)+A_m, ~~\text{for~~} \begin{pmatrix}
\xi \\
\eta
\end{pmatrix} \in \hat{T}, ~~m = 1, 2, 3,
\end{equation}
where the matrix $\mathbf{ J}_m(\bfalpha)$ is the Jacobian matrix of $\mathcal{F}_m$ such that
$
\mathbf{ J}_1(\bfalpha)=\left( P(\bfalpha)-A_1, Q(\bfalpha)-A_1 \right), ~\mathbf{ J}_2(\bfalpha)=\left( Q(\bfalpha)-A_2, P(\bfalpha)-A_2 \right), ~\mathbf{ J}_3(\bfalpha)=\left( Q(\bfalpha)-A_3, A_2-A_3 \right).
$
For every $X\in T$, the function $X(\bfalpha)$ given in \eqref{aff_map} is a piecewise
differentiable function such that for every $j\in\mathcal{D}$
\begin{align}
&D_{\alpha_j} X(\bfalpha)= (D_{\alpha_j} \mathbf{ J}_m(\bfalpha)) \mathbf{ J}^{-1}_m(\bfalpha)(X(\bfalpha)-A_m) ~~~\text{for~} X(\bfalpha) \in T_m \subseteq T, ~m=1, 2, 3, \label{DX} \\
\text{with} ~~&D_{\alpha_j}\mathbf{ J}_1(\bfalpha)=\left(D_{\alpha_j}P, D_{\alpha_j}Q\right), ~~ D_{\alpha_j}\mathbf{ J}_2(\bfalpha)=\left( D_{\alpha_j}Q, D_{\alpha_j}P\right), ~~ D_{\alpha_j}\mathbf{ J}_3(\bfalpha)=\left( D_{\alpha_j}Q, \mathbf{0} \right), \label{J_123_deriv}
\end{align}
in which $D_{\alpha_j}P$ and $D_{\alpha_j}Q$ are provided by the formulas \eqref{P_deri_4} and \eqref{Q_deri_4}.
Therefore, by the formula for $X(\bfalpha) \in T\in \mathcal{T}_h^i$ given in \eqref{aff_map} and its derivatives given
in \eqref{DX}, we introduce a piecewise velocity field $\mathbf{ V}^j$ with respect to the $j$-th design variable $\alpha_j$, $j\in\mathcal{D}$ as follows:
\begin{eqnarray}
\label{v_filed}
\mathbf{ V}^j(X) = \begin{cases}
\mathbf{ V}^j_T(X)=\mathbf{0}, & \text{if $T\notin \mathcal{T}^i_h $ }, \\
\mathbf{ V}^j_T(X)=(D_{\alpha_j} \mathbf{ J}_m(\bfalpha)) \mathbf{ J}^{-1}_m(\bfalpha)(X(\bfalpha)-A_m), & \text{if $T\in\mathcal{T}^i_h$ and $X\in T_m$}, m = 1, 2, 3.
\end{cases} \label{eq4_1}
\end{eqnarray}

We now present the properties of the velocity field by the following theorem.
\commentout{
\begin{lemma}
\label{v_property}
On each interface element $T = \bigtriangleup A_1A_2A_3 \in \mathcal{T}_h^i$, the following formulas hold for the velocity field $\mathbf{V}^j(X), ~j \in \mathcal{D}$ defined in \eqref{v_filed}:
\begin{equation}
\label{property_1}
\mathbf{ V}^j_T|_{A_iP}=\frac{\norm{X-A_i}}{\norm{P-A_i}}D_{\alpha_j}P, ~ i=1,2,~~~and~~~ \mathbf{ V}^j_T|_{A_iQ}=\frac{\norm{X-A_i}}{\norm{Q-A_i}}D_{\alpha_j}Q, ~ i=1,3.
\end{equation}
\begin{align}
 \mathbf{ V}^j_T|_{PQ}=\frac{\norm{X-Q}}{\norm{P-Q}}D_{\alpha_j}P+\frac{\norm{X-P}}{\norm{P-Q}}D_{\alpha_j}Q,& ~~ \mathbf{ V}^j_T|_{A_2Q}=\frac{\norm{X-A_2}}{\norm{Q-A_2}}D_{\alpha_j}Q, ~~\mathbf{ V}^j_T|_{A_2A_3}=\mathbf{ 0}. \label{property_2}
\end{align}
and
\begin{equation}
\label{div}
\emph{div}(\mathbf{ V}^j_{T_m})=\emph{tr}\left( (D_{\alpha_j} \mathbf{ J}_m) \mathbf{ J}^{-1}_m \right), ~~m = 1, 2, 3,
\end{equation}
\end{lemma}
\begin{proof}
Formula \eqref{div} follows directly from \eqref{v_filed}. Formulas in \eqref{property_1} and \eqref{property_2} can be verified by direct calculations.
Take the first identity in \eqref{property_2} for example, we can use \eqref{v_filed}
and \eqref{J_123_deriv} to calculate $\mathbf{ V}^j|_{T_i}, i = 1, 2$ and then show that, when restricted on $PQ$, they lead to the same formula given in the first identity in \eqref{property_2}.

\commentout{
Identities in \eqref{property_1} and \eqref{property_2} can be verified by direct calculations, and to reduce the presentation page space, we only provide a proof for the first identity in \eqref{property_1}, the proofs for the other follow similarly. Note that $\mathbf{ J}^{-1}_2(P-A_2)=(0,1)^T$. Then for any $X$ on the line $A_2P$, by \eqref{J_123}, we have
\begin{equation*}
\begin{split}
(D_{\alpha_j} \mathbf{ J}_2) \mathbf{ J}^{-1}_2(X-A_2)&= \frac{\norm{X-A_2}}{\norm{P-A_2}}(D_{\alpha_j} \mathbf{ J}_2) \mathbf{ J}^{-1}_2(P-A_2) \\
&=  \frac{\norm{X-A_2}}{\norm{P-A_2}} (D_{\alpha_j} \mathbf{ J}_2) \left(\begin{array}{c}0 \\1 \end{array}\right) \\
&= \frac{\norm{X-A_2}}{\norm{P-A_2}} D_{\alpha_j}P,
\end{split}
\end{equation*}
which establishes the first identity in \eqref{property_1}.
}
\end{proof}
}
\begin{thm}
\label{v_H1}
For any $j\in\mathcal{D}$, the velocity $\mathbf{ V}^j(X)$ defined in \eqref{v_filed} has the properties:
\begin{itemize}
  \item[\textbf{P1}:] on each interface element $T = \bigtriangleup A_1A_2A_3 \in \mathcal{T}_h^i$, there holds
 \begin{subequations}
 \label{v_iden}
 \begin{align}
\mathbf{ V}^j_T|_{A_iP}=\frac{\norm{X-A_i}}{\norm{P-A_i}}D_{\alpha_j}P, ~ i=1,2,~~~&~~~ \mathbf{ V}^j_T|_{A_iQ}=\frac{\norm{X-A_i}}{\norm{Q-A_i}}D_{\alpha_j}Q, ~ i=1,3, \label{property_1} \\
\mathbf{ V}^j_T|_{PQ}=\frac{\norm{X-Q}}{\norm{P-Q}}D_{\alpha_j}P+\frac{\norm{X-P}}{\norm{P-Q}}D_{\alpha_j}Q,& ~~ \mathbf{ V}^j_T|_{A_2Q}=\frac{\norm{X-A_2}}{\norm{Q-A_2}}D_{\alpha_j}Q, ~~\mathbf{ V}^j_T|_{A_2A_3}=\mathbf{ 0}, \label{property_2} \\
\emph{div}(\mathbf{ V}^j_{T_m})=\emph{tr}\left( (D_{\alpha_j} \mathbf{ J}_m) \mathbf{ J}^{-1}_m \right)&, ~~m = 1, 2, 3; \label{div}
\end{align}
\end{subequations}
  \item[\textbf{P2}:] $\mathbf{ V}^j \in H^1(\Omega)$ and $supp(\mathbf{ V}^j) \subseteq \bigcup_{T \in \mathcal{T}_h^i}T$;
  \item[\textbf{P3}:] when restricted on each interface edge e, $\mathbf{ V}^j(X)$ has the same direction as the edge e.
\end{itemize}
\end{thm}
\begin{proof}
\textbf{P1} can be verified by calculation and the definition \eqref{v_filed}. \textbf{P2} is the consequence of \textbf{P1} and the definition \eqref{v_filed}. \textbf{P3} is based on \eqref{property_1} and Remark \ref{P_A1A2}.
\end{proof}

\commentout{
We note that there are other ways to express a point $X$ in an interface element $T$ as a function
of the design variable $\bfalpha$ and they can lead to different velocity fields. The advantage of the
proposed formula \eqref{aff_map} for $X$ is that its derivatives can be calculated by the
explicit formula given in \eqref{v_filed} instead of any approximation techniques; hence, formula
\eqref{aff_map} can accurately represent how the quantities depend on spacial variable change according to the movement of the interface. Furthermore, this formulation of $X$ in terms of $\bfalpha$ naturally enables us to
propose a velocity field that can be used to compute the sensitivity efficiently since the support of this velocity field is inside the union of interface elements whose number is proportional only to $O(h^{-1})$ while
the number of all elements in the mesh is in the order of $O(h^{-2})$.
}

\subsection{Shape Derivatives of IFE Shape Functions}

\commentout{
{\color{red}These sentences are from Section 2: By the formula \eqref{c0}, we note each IFE shape function on an interface element $T \in \mathcal{T}_h^i$ is determined by $L(X)$, $c_0$ and $\mathbf{ c}=(c_i)_{i\in\mathcal{I}^-}$ which are functions of
the interface-mesh intersection points $P$, $Q$ described explicitly by \eqref{c0}-\eqref{gam_del}. Furthermore,
the rate of change of the interface-mesh intersection points $P$, $Q$ with respect to the interface $\Gamma$ in terms of
the design variable vector $\bfalpha$ can be readily established, see more details in Section {\color{red}???}. Therefore, \eqref{c0} is actually a formula
explicitly describing how each IFE shape function depends on the design variable vector $\bfalpha$ and this feature
greatly benefits our sensitivity computations in the shape optimization.}
}

In the proposed IFE method described by \eqref{discrt_obj_2}, the IFE basis functions $\phi_i, 1 \leq i \leq
\abs{\mathcal{N}_h}$ on the chosen fixed interface independent mesh are directly employed in the objective function $\mathcal{J}_h$ according to \eqref{ife_solu_alpha} and \eqref{discrt_obj_1}. By their construction described in
\eqref{ife_shapefun}-\eqref{gam_del}, the IFE basis functions change when the interface
$\Gamma(t, \bfalpha), t \in [0, 1]$ moves because of the variations in the design variable $\bfalpha$. Hence, the gradient of the objective function $\mathcal{J}_h$ in this IFE method inevitably involves the derivatives of
the IFE basis functions with respect to $\bfalpha$. By definition, each IFE basis function is a piecewise polynomial that
is a linear combination of the IFE shape functions on each element described by \eqref{ife_space_inter_element} or \eqref{non_inter_loc_space} depending on whether the element is an interface element or not. Consequently, the derivative of an IFE basis function $\phi_i$ with respect to $\bfalpha$ is zero on each non-interface element where all the shape functions are independent of $\bfalpha$, and our focus in this subsection will be the derivative of IFE shape functions with respect to $\bfalpha$ on interface elements. We note that \cite{2015NajafiSafdari,2012ZhangZhangZhu} presented similar approaches to calculate the shape derivative for special finite element shape functions.

Consider a typical interface element $T = \bigtriangleup A_1A_2A_3$ configured as in Figure \ref{interface_element}. By \eqref{ife_shapefun} and the discussions
at the beginning of this section and Section \ref{sec:DerivativeIntersectionPoints}, we express an IFE shape function $\psi_T^{int}(X)$ on $T$ as
$\psi_T^{int}(X) = \psi_T^{int}(X(\bfalpha), \bfalpha)$ to emphasize that the design variable $\bfalpha$ influences the value of $\psi_T^{int}$ not only
through the spatial variable $X$ which is a function of $\bfalpha$ according to \eqref{aff_map}, but also directly through its coefficients $c_0, \bfc$ and the coefficients of $L(X)$. However, the rate of change for an IFE shape function $\psi_T^{int}$ with respect to $\alpha_j, j \in \mathcal{D}$ through $X(\bfalpha)$
is readily known by the simple chain rule for differentiation because $\psi_T^{int}(X(\bfalpha), \bfalpha)$ depends on $X$ linearly and $\frac{\partial X}{\partial \alpha_j}$ is a velocity field already discussed in Section \ref{sec:VelocityField}. Therefore, we only need to discuss the rate of change for an IFE shape function $\psi_T^{int}$ with respect to $\alpha_j, j \in \mathcal{D}$ not through $X(\bfalpha)$, and this rate of change is referred as a shape derivative in the shape optimization literature
\cite{2003HaslingerMkinen}.

\commentout{
According to the formula for $\psi_T^{int}(X(\bfalpha), \bfalpha)$ given in \eqref{ife_shapefun}, we need to prepare
formulas for $\frac{\partial \mathbf{ c}}{\partial \alpha_j}$, $\frac{\partial  c_0}{\partial \alpha_j}$ and $\frac{\partial L}{\partial \alpha_j}$ which is the shape derivative of
$L(X(\bfalpha), \bfalpha)$. For simplicity, we adopt the abbreviation $X$ for $X(\bfalpha)$ in this subsection for the discussion of shape derivatives of IFE functions.

We now consider the shape derivatives of IFE shape functions locally on each interface element which are needed for computing the shape derivatives of IFE basis functions globally defined over the whole solution domain $\Omega$. Let $T = \bigtriangleup A_1A_2A_3$ be a typical interface element
configured as in Figure \ref{interface_element}. By \eqref{ife_shapefun}, an IFE function
$\psi_T^{int}(X)$ on $T$ is a functional of $\bfalpha$ because it is constructed with
$c_0, L(X)$, and $\bfc$ which depend on $\bfalpha$ through interface-mesh intersection points
$P$ and $Q$. Hence, we can write it as {\color{red}$\psi_T^{int}(X) = \psi_T^{int}(X(\bfalpha), \bfalpha)$}, by which, we note that the design variable $\bfalpha$ can change the value of $\psi_T^{int}$ not only
directly, but also through the spatial variable $X$ which is a function of $\bfalpha$ according to \eqref{aff_map}. We note that the enriched basis functions in \cite{2015NajafiSafdari} corresponding to the interface-mesh intersection points also have such a feature. This explicit dependence
of IFE shape functions on the design variable enables use to compute the shape derivatives of
IFE basis functions which is an important ingredient in the proposed IFE method for shape optimizations, and this makes our proposed IFE method very different from solving shape optimization problems by the standard element method in which the dependence of the shape functions on $\bfalpha$ is only through the spacial variables $X$ because the standard finite element shape functions on reference elements are independent of the interface change after the affine mapping and only the Jacobian matrix of the affine mapping depend on $\bfalpha$.

{\color{red}By definition \cite{J.Sokolowski_J.-P.Zolesio_1992}}, a shape derivative of a function is about how
$\bfalpha$ affects this function directly, not how $\bfalpha$ affects it via the spatial variable $X(\bfalpha)$. Hence,
according to its definition \eqref{ife_shapefun}, to derive the formula for the shape
derivative of an IFE function $\psi^{int}_T(X(\bfalpha), \bfalpha)$ with respect to $\alpha_j, j \in \mathcal{I}$ we need to derive formulas for computing $\frac{\partial \mathbf{ c}}{\partial \alpha_j}$, $\frac{\partial  c_0}{\partial \alpha_j}$ and $\frac{\partial L}{\partial \alpha_j}$ which is the shape derivative of
$L(X(\bfalpha), \bfalpha)$. For simplicity, we adopt the abbreviation $X$ for $X(\bfalpha)$ in this subsection for the discussion of shape derivatives of IFE functions.
}
First, by their formulas given in Section \ref{review_ife}, both $L(X)$ and $\bar{\mathbf{ n}}$ depend on the design variable $\bfalpha$ because of their dependence on the interface-mesh intersection points
$P = (x_P, y_P)$ and $Q = (x_Q, y_Q)$ that are functions of $\bfalpha$. By direct calculations, we have
\begin{align}
&\frac{\partial L}{\partial P}=\frac{(X-P)^T\bar{\mathbf{ t}}\bar{\mathbf{ n}}^T}{\norm{P-Q}}+\bar{\mathbf{ n}}, ~~\frac{\partial L}{\partial Q}=-\frac{(X-Q)^T\bar{\mathbf{ t}}\bar{\mathbf{ n}}^T}{\norm{P-Q}}, ~~\frac{\partial \bar{\mathbf{ n}}}{\partial P}=\frac{\bar{\mathbf{ t}} \mathbf{ \bar{n}}^T }{\norm{P-Q}}, ~~\frac{\partial \bar{\mathbf{ n}}}{\partial Q}=-\frac{\bar{\mathbf{ t}} \mathbf{ \bar{n}}^T }{\norm{P-Q}}, \label{Ln_PQ}
\end{align}
where $\frac{\partial L}{\partial P}=(\frac{\partial L}{\partial x_P}, \frac{\partial L}{\partial y_P})$, $\frac{\partial L}{\partial Q}=(\frac{\partial L}{\partial x_Q}, \frac{\partial L}{\partial y_Q})$ are $1$-by-$2$ matrices, and $\bar{\mathbf{ t}}=\frac{1}{\norm{P-Q}}(x_P-x_Q, y_P-y_Q)^T$ is the tangential vector of $l$,
$\frac{\partial \bar{\mathbf{ n}}}{\partial P}=(\frac{\partial \bar{\mathbf{ n}}}{\partial x_P},\frac{\partial \bar{\mathbf{ n}}}{\partial y_P})$, $\frac{\partial \bar{\mathbf{ n}}}{\partial Q}=(\frac{\partial \bar{\mathbf{ n}}}{\partial x_Q},\frac{\partial \bar{\mathbf{ n}}}{\partial y_Q})$ are $2$-by-$2$ matrices. Then, by the chain rule,
we can use \eqref{Ln_PQ} to calculate $\frac{\partial L(X, \bfalpha)}{\partial \alpha_j}$ and $\frac{\partial\bar{\mathbf{ n}}}{\partial \alpha_j}$ as follows:
\begin{align}
&\frac{\partial L(X, \bfalpha)}{\partial \alpha_j}=\frac{\partial L}{\partial P} ~D_{\alpha_j}P
                                                  +\frac{\partial L}{\partial Q} ~D_{\alpha_j}Q, ~~\frac{\partial \bar{\mathbf{ n}}}{\partial \alpha_j}=\frac{\partial \bar{\mathbf{ n}}}{\partial P} ~
D_{\alpha_j}P + \frac{\partial \bar{\mathbf{ n}}}{\partial Q} ~D_{\alpha_j}Q, \label{L_deri} 
\end{align}
in which $D_{\alpha_j}P$ and $D_{\alpha_j}Q$ are given by formulas in \eqref{P_deri_4}, and \eqref{Q_deri_4}.
\commentout{According to \eqref{L_deri}, we note that the shape derivatives of $L$ and $\bar{\mathbf{ n}}$ are decomposed into two components: (1) their derivatives with respect to interface-mesh intersection points $P$ and $Q$, and (2) the total derivatives of $P$ and $Q$ with respect to the design variables.
The derivatives of $L$ and $\bar{\mathbf{ n}}$ with respect to interface-mesh intersection points $P$ and $Q$ given in \eqref{Ln_PQ} are independent of the parametrization of the interface.}

Then, by \eqref{c0} and \eqref{gam_del}, we have
\begin{subequations}
\label{cgdb_deri}
\begin{align}
&\frac{\partial c_0}{\partial \alpha_j}  =  \mu \left( \sum_{i\in\mathcal{I}^-} \frac{\partial c_i}{\partial \alpha_j} \nabla\psi^{non}_{i,T}\cdot \bar{\mathbf{ n}}+ c_i \nabla\psi^{non}_{i,T} \cdot \frac{\partial\bar{\mathbf{ n}}}{\partial \alpha_j}+\sum_{i\in\mathcal{I}^+}v_i\nabla\psi^{non}_{i,T}\cdot  \frac{\partial\bar{\mathbf{ n}}}{\partial \alpha_j} \right), \label{c0_deri}\\
&\frac{\partial\bfgamma}{\partial \alpha_j} = \left( \nabla\psi^{non}_{i,T}\cdot \frac{\partial\bar{\mathbf{ n}}}{\partial \alpha_j} \right)_{i\in\mathcal{I}^-}, \;\;\;\;\;\;\;\frac{\partial\bfdelta}{\partial \alpha_j}=\left( \frac{\partial L(A_i)}{\partial \alpha_j} \right)_{i\in\mathcal{I}^-}, \label{gd_deri} \\
& \frac{\partial\bfb}{\partial \alpha_j} =
\left(
- \mu\frac{ \partial L(A_i)}{\partial \alpha_j} \sum_{j\in\mathcal{I}^+} \nabla\psi^{non}_{j,T} \cdot \bar{\mathbf{ n}}~ v_j - \mu L(A_i)\sum_{j\in\mathcal{I}^+} \nabla\psi^{non}_{j,T} \cdot \frac{\partial\bar{\mathbf{ n}}}{\partial \alpha_j}~ v_j
 \right)_{i\in\mathcal{I}^-}. \label{b_deri}
\end{align}
\end{subequations}
\commentout{
In addition, by \eqref{coef}, \eqref{gd_deri}, and \eqref{b_deri}, we obtain
\begin{align}
{\color{red}\frac{\partial \mathbf{c}}{\partial \alpha_j} = ??????} \label{c_deri}
\end{align}
}
Furthermore, by \eqref{c0} again, we can compute $\frac{\partial \mathbf{ c}}{\partial\alpha_j}, j \in \mathcal{D}$ from \eqref{gd_deri}, \eqref{b_deri} as follows:
\begin{equation}
\label{c_alpha_j}
\frac{\partial \mathbf{ c}}{\partial \alpha_j}=\frac{\partial \mathbf{ b}}{\partial \alpha_j}-\mu\frac{ \left[\left(\frac{\partial \bfgamma}{\partial \alpha_j}\right)^T\mathbf{ b}\bfdelta + \bfgamma^T\frac{\partial\mathbf{ b}}{\partial\alpha_j}\bfdelta + \bfgamma^T\mathbf{ b}\frac{\partial \bfdelta}{\partial \alpha_j} \right](1+\mu\bfgamma^T\bfdelta) -
\mu\bfgamma^T\mathbf{ b}\bfdelta\left[ \left( \frac{\partial \bfgamma}{\partial \alpha_j} \right)^T\bfdelta + \bfgamma^T \frac{\partial \bfdelta}{\partial \alpha_j} \right]
 }{(1+\mu\bfgamma^T\bfdelta)^2}.
\end{equation}
Finally, we use \eqref{L_deri}, \eqref{c0_deri}, and \eqref{c_alpha_j} to obtain the formula for the shape derivatives of an IFE shape function defined by \eqref{ife_shapefun} by the following formula: for every $j \in \mathcal{D}$,
\begin{equation}
\label{ife_shapefun_shape_deriv}
\frac{\partial \psi^{int}_T(X, \bfalpha)}{\partial \alpha_j} =
\begin{cases}
\Frac{\partial \psi^{int,-}_T(X, \bfalpha)}{\partial \alpha_j}  = \Frac{\partial \psi^{int,+}_T(X, \bfalpha)}{\partial \alpha_j} + \Frac{\partial c_0}{\partial \alpha_j}L(X, \bfalpha) + c_0 \Frac{L(X, \bfalpha)}{\partial \alpha_j}& \text{if} \;\; X\in \overline{T}^-, \\
\\
\Frac{\partial \psi^{int,+}_T(X, \bfalpha)}{\partial \alpha_j}  = \sum_{i\in\mathcal{I}^-}\Frac{\partial c_i}{\partial \alpha_j}\psi^{non}_{i,T}(X)& \text{if} \;\; X\in \overline{T}^+.
\end{cases}
\end{equation}

\commentout{
We note the proposed velocity field can be utilized by any fixed mesh methods for sensitivities. The particular advantage for the IFE methods on shape optimization problems is that the shape derivatives of the IFE shape functions can be calculated by explicit formulas. So next, based on Lemma \ref{inter_pt_deri}, on any interface element $T\in\mathcal{T}^i_h$, we derive those explicit formula for the derivatives of the coefficients \eqref{c0} and \eqref{coef} with respect to $\bfalpha$. Since the dependence of the coefficients are only through the mesh-interface intersection points, the total derivatives can be decoupled into two independent parts by the chain rule: the derivatives of the interface-mesh intersection points with respect to $\bfalpha$ and the derivatives of the coefficients in the IFE shape functions with respect to the intersection points. We shall see the latter one only needs to be derived once and remains the same in the whole optimization process while the former one is given by Lemma \ref{inter_pt_deri}. We emphasize that it is a difference from solving shape optimization problems using the standard finite element method in which the dependence of the shape functions on $\bfalpha$ is only through the spacial variables $X$, since the standard finite element shape functions on reference elements are independent with the interface change after the affine mapping and only the Jacobian matrix of the affine mapping depend on $\bfalpha$. We note that the enriched basis functions in \cite{2015NajafiSafdari} corresponding to the interface-mesh intersection points also have such a feature. In the formulation \eqref{ife_shapefun}, only the constructed coefficients $\mathbf{c}$, $c_0$ and the function $L(X)$ depend on $\bfalpha$ while the standard basis functions $\psi^{non}_{i,T}$ are fixed. Hence for any $\alpha_j$, $j\in\mathcal{D}$, we only need to compute $\frac{\partial \mathbf{ c}}{\partial \alpha_j}$, $\frac{\partial  c_0}{\partial \alpha_j}$ and $\frac{\partial L}{\partial \alpha_j}$. And according to the formula \eqref{c0} and \eqref{coef}, we only need:

Note that $\frac{\partial \mathbf{ c}}{\partial\alpha_j}$ can be computed from \eqref{gd_deri}, \eqref{b_deri} and the formula \eqref{coef}. Then we only need to calculate $\frac{\partial\bar{\mathbf{ n}}}{\partial \alpha_j}$. By the chain rule, for any $X$, we have

Let $\bar{\mathbf{ t}}=\frac{1}{|P-Q|}(x_P-x_Q, y_P-y_Q)^T$ be the tangential vector of $l$. Then by direct calculation, we have
\begin{equation}
\label{n_PQ}
\frac{\partial \bar{\mathbf{ n}}}{\partial P}=\frac{\bar{\mathbf{ t}} \mathbf{ \bar{n}}^T }{|P-Q|}, ~~~~~ \frac{\partial \bar{\mathbf{ n}}}{\partial Q}=-\frac{\bar{\mathbf{ t}} \mathbf{ \bar{n}}^T }{|P-Q|},
\end{equation}

Finally the analytical shape derivatives of IFE shape functions are based on \eqref{coef} and \eqref{cgdb_deri}-\eqref{L_PQ}. Here \eqref{nL_deri}-\eqref{L_PQ} indicate that the calculation of the $\frac{\partial \bar{\mathbf{ n}}}{\partial \alpha_j}$ and $ \frac{\partial L(X)}{\partial \alpha_j}$ can be decoupled into the the derivatives of intersection points with respect to $\alpha_j$ and the derivatives with respect to the intersection points \eqref{n_PQ} and \eqref{L_PQ}. Indeed the formulas \eqref{n_PQ} and \eqref{n_PQ} are essentially applicable for any interface shapes and locations. We shall see in the following of this section the advantages that the shape derivatives of the IFE shape functions can be calculated by explicit formulas and the shape functions on non-interface elements are independent with the interface change make the sensitivity calculation efficiently and accurately.
}


\subsection{The Gradient of the Discretized Objective Function}

The gradient of the objective function $\mathcal{J}_h$ is necessary for implementing the proposed IFE method with a common
minimization algorithm based on a decent direction or trust region. We now put all the preparations in the previous subsections together to derive
the formula for the gradient of the objective function $\mathcal{J}_h$ that can be executed efficiently within the IFE framework. This formula involves the total derivatives of $\mathcal{J}_h$ with respect to $\bfalpha_j,~j \in \mathcal{D}$ depending on the velocity field and these total derivatives are also referred as the material derivatives of $\mathcal{J}_h$ in the shape optimization literature \cite{2003HaslingerMkinen}. For the simplicity of presentation, we assume that the boundary condition functions $g^k_N$, $g^k_D$ and the force term $f^k$ are fixed and independent with interface change, $1\leqslant k \leqslant K$. In the following this discussion, we use $\nabla$ to denote the standard gradient operator with respect to $X$. We start from the material derivatives with respect to $\alpha_j, j \in \mathcal{D}$ of the local matrices and vectors which are used to construct $\mathbf{ A}^k(X(\bfalpha), \bfalpha)$ and $\mathbf{ F}^k(X(\bfalpha), \bfalpha)$. Their formulas are presented in the two theorems below, of which the derivation is based on Lemma 3.3 of \cite{2003HaslingerMkinen} in the direction of
the velocity field developed in Section \ref{sec:VelocityField} together with the properties in Theorem \ref{v_H1} and the shape derivatives of the IFE shape functions given by the formula \eqref{ife_shapefun_shape_deriv}

\begin{thm}\label{D_alpha_local_matrices}
On each interface element $T\in \mathcal{T}_h^i$ and each interface edge $e \in \mathcal{E}_h^i$,
we have the following formulas for the material derivatives of $\mathbf{K}_T, \mathbf{E}_e^{r_1, r_2}, \mathbf{G}_e^{r_1, r_2}$ and $\mathbf{R}_T$ with respect to $\alpha_j, j\in \mathcal{D}$:
\begin{subequations}
\begin{align}
\label{KT_deri}
D_{\alpha_j}\mathbf{ K}_T &= \left( \int_T \beta \nabla \frac{\partial\psi_{p,T}}{\partial \alpha_j} \cdot \nabla \psi_{q,T} dX \right)_{p, q\in\mathcal{I}} + \left( \int_T \beta \nabla \frac{\partial\psi_{p,T}}{\partial \alpha_j} \cdot \nabla \psi_{q,T} dX \right)_{p, q\in\mathcal{I}}^T \nonumber \\
&~~~~~~ +\left( \sum_{i=1}^3 \int_{T_i} \beta \nabla \psi_{p,T} \cdot \nabla \psi_{q,T}~ dX ~ \emph{tr}\left( (D_{\alpha_j} \mathbf{ J}_i) \mathbf{ J}^{-1}_i \right) \right)_{p, q\in\mathcal{I}}, \\
D_{\alpha_j} \mathbf{ E}^{r_1r_2}_e & =
\left(\int_e \beta \nabla \frac{\partial\psi_{p,T^{r_1}}}{\partial \alpha_j}\cdot (\psi_{q,T^{r_2}}\mathbf{ n}^{r_2}_e)  ds \right)_{p, q\in\mathcal{I}} + \left(\int_e  \beta \nabla \psi_{p,T^{r_1}}\cdot  (\frac{\partial\psi_{q,T^{r_2}}}{\partial \alpha_j}\mathbf{ n}^{r_2}_e)  ds \right)_{p, q\in\mathcal{I}} \nonumber \\
& ~~+  \Bigg( \beta^- \nabla \psi^-_{p,T^{r_1}}\cdot (\psi^-_{q,T^{r_2}}\mathbf{ n}^{r_2}_e)|_{P} - \beta^+ \nabla \psi^+_{p,T^{r_1}}\cdot (\psi^+_{q,T^{r_2}}\mathbf{ n}^{r_2}_e)|_{P}  \Bigg)_{p, q\in\mathcal{I}} \frac{ D_{\alpha_j}P\cdot(A_2-A_1)}{\norm{A_2-A_1}}, \label{DE_e} \\
D_{\alpha_j} \mathbf{ G}^{r_1r_2}_e &= \frac{\sigma^0_e}{|e|} \left(\int_e (\frac{\partial\psi_{p,T^{r_1}}}{\partial \alpha_j}\mathbf{ n}^{r_1}_e)\cdot (\psi_{q,T^{r_2}}\mathbf{ n}^{r_2}_e)  ds + \int_e  (\psi_{p,T^{r_1}}\mathbf{ n}^{r_1}_e)\cdot (\frac{\partial\psi_{q,T^{r_2}}}{\partial \alpha_j}\mathbf{ n}^{r_2}_e) ds \right)_{p, q\in\mathcal{I}}, \label{DG_e} \\
D_{\alpha_j} \mathbf{R}_T &= \left( \int_T \frac{\partial\psi_{p,T}}{\partial \alpha_j} dX + \int_T\nabla\psi_{p,T}\cdot\mathbf{ V}^j dX \right)_{p\in\mathcal{I}} +\left( \sum_{i=1}^3 \int_{T_i} \psi_{p,T}~ dX ~ \emph{tr}\left( (D_{\alpha_j} \mathbf{ J}_i) \mathbf{ J}^{-1}_i \right) \right)_{p \in\mathcal{I}}. \label{DR_T}
\end{align}
\end{subequations}
\end{thm}
\commentout{
\begin{proof}
These formulas follow directly from the standard material derivative formulas given in Lemma 3.3 of \cite{2003HaslingerMkinen} in the direction of
the velocity field developed in Section \ref{sec:VelocityField} plus the fact that each IFE shape function is a piecewise linear polynomial. For example,
we can apply the first material derivative formula in Lemma 3.3 of \cite{2003HaslingerMkinen} to $\mathbf{ K}_T$ in \eqref{ppife_local_mat_1}, then we
obtain \eqref{KT_deri} by simplifying the result with \eqref{div} and $\frac{\partial}{\partial X} \left( \beta \nabla \psi_{k,T} \cdot \nabla \psi_{l,T} \right)=0$.
According to Theorem \ref{v_H1}, formula \eqref{DE_e} follows from applying the second formula in Lemma 3.3 of \cite{2003HaslingerMkinen} to
$\mathbf{ E}_e^{r_1r_2}$ in \eqref{ppife_local_mat_2} and then simplifying the result by \eqref{property_1} and the fact $\mathbf{ V}^j_T(A_i)=\mathbf{ 0}$, $i=1,2$. The other two formulas \eqref{DG_e} and \eqref{DR_T} follow from similar arguments.
\commentout{
For each $T\in\mathcal{T}^i_h$, according to the material derivative formula given in {\color{red}Lemma 3.3} of \cite{2003HaslingerMkinen}, we have
\begin{equation}
\begin{split}
\label{KT_deri_1}
D_{\alpha_j}\mathbf{ K}_T 
&= \left( \int_T \beta \nabla \frac{\partial}{\partial \alpha_j}\psi_{p,T} \cdot \nabla \psi_{q,T} dX + \int_T \beta \nabla \psi_{p,T} \cdot \nabla\frac{\partial}{\partial \alpha_j} \psi_{q,T} dX \right)_{p, q\in\mathcal{I}} \\
& +\left( \int_T \frac{\partial}{\partial X} \left( \beta \nabla \psi_{p,T} \cdot \nabla \psi_{q,T} \right) \cdot \mathbf{ V}^j_T dX \right)_{p, q\in\mathcal{I}} +\left( \int_T \beta \nabla \psi_{p,T} \cdot \nabla \psi_{q,T}~ \textrm{div}(\mathbf{ V}^j_T)~ dX \right)_{p, q\in\mathcal{I}},
\end{split}
\end{equation}
where $\frac{\partial\psi_{p,T}}{\partial \alpha_j}, p \in \mathcal{I}$ are the shape derivatives of the IFE shape functions with respect to $\alpha_j$ given by the formula in \eqref{ife_shapefun_shape_deriv}. Since
$\nabla \psi_{p,T} \cdot \nabla \psi_{q,T}$ are constants for any $p, q\in\mathcal{I}$, we have
$\frac{\partial}{\partial X} \left( \beta \nabla \psi_{k,T} \cdot \nabla \psi_{l,T} \right)=0$;
hence, \eqref{KT_deri} follows from \eqref{KT_deri_1} together with \eqref{div}. {\color{blue}And the formula \eqref{DR_T} for the material derivatives of $\mathbf{ R}_T$ can be derived similarly.}

For the material derivatives of $\mathbf{ E}^{r_1r_2}_e$, let $e = \overline{A_1A_2}$ in an interface
element $T$ as configured in Figure \ref{interface_element} without loss of generality. According to Remark \ref{v_e}, the geometry change restricted on interface edges can be viewed as the 1-D shape change with the two ending points fixed. Hence, applying the material derivative formula given in {\color{red}Lemma 3.3} of \cite{2003HaslingerMkinen} to $E^{r_1r_2}_e$ given in \eqref{ppife_local_mat_2}, we have
\begin{align}
D_{\partial_j} \mathbf{ E}^{r_1r_2}_e &=  D_{\partial_j} \left(\int_e  \beta \nabla \psi_{p,T^{r_1}}\cdot (\psi_{q,T^{r_2}}\mathbf{ n}^{r_2}_e) ds \right)_{p, q\in\mathcal{I}}  \nonumber \\
&= \left(\int_e  \frac{\partial}{\partial \alpha_j}  \left( \beta \nabla \psi_{p,T^{r_1}}\cdot (\psi_{q,T^{r_2}}\mathbf{ n}^{r_2}_e) \right) ds \right)_{p, q\in\mathcal{I}} +  \Bigg( \beta^- \nabla \psi^-_{p,T^{r_1}}\cdot (\psi^-_{q,T^{r_2}}\mathbf{ n}^{r_2}_e)|_{P}\mathbf{ V}^j_T(P)\cdot\mathbf{ t}_e  \Bigg)_{p, q\in\mathcal{I}} \nonumber \\
&-  \Bigg( \beta^+ \nabla \psi^+_{p,T^{r_1}}\cdot  (\psi^+_{q,T^{r_2}}\mathbf{ n}^{r_2}_e)|_{P}\mathbf{ V}^j_T(P)\cdot\mathbf{ t}_e  \Bigg)_{p, q\in\mathcal{I}}  - \Bigg( \beta \nabla \psi_{p,T^{r_1}}\cdot (\psi_{q,T^{r_2}}\mathbf{ n}^{r_2}_e)|_{A_1}\mathbf{ V}^j_T(A_1)\cdot\mathbf{ t}_e  \Bigg)_{p, q\in\mathcal{I}} \nonumber \\
&+ \Bigg( \beta \nabla \psi_{p,T^{r_1}}\cdot  (\psi_{q,T^{r_2}}\mathbf{ n}^{r_2}_e)|_{A_2}\mathbf{ V}^j_T(A_2)\cdot\mathbf{ t}_e  \Bigg)_{p, q\in\mathcal{I}}, \label{DE_e1}
\end{align}
where $\mathbf{ t}_e$ is the unit direction vector from $A_1$ to $A_2$. Then,
\eqref{DE_e} follows from applying \eqref{property_1} and the fact $\mathbf{ V}^j_T(A_1)=\mathbf{ V}^j_T(A_2)=\mathbf{ 0}$ to \eqref{DE_e1}. The formula \eqref{DG_e} for the material derivative
of $\mathbf{G}_e^{r_1, r_2}$ can be derived similarly.
}
\end{proof}
}

\commentout{
Suppose $\frac{\partial \mathcal{J}_h}{\partial \mathbf{ u}_h}$ is known, in order to compute \eqref{adj_1}, we only need to evaluate $ D_{\alpha_j}\mathbf{ A}$ and  $D_{\alpha_j}\mathbf{ \widetilde{F}}$. Since the global matrix $\mathbf{ A}$ and vector $\widetilde{\mathbf{ F}}$ are assembled by the local matrices and vectors on elements or edges, the computation for $ D_{\alpha_j}\mathbf{ A}$ and  $D_{\alpha_j}\widetilde{\mathbf{ F}}$ reduces to the derivatives of those local matrices or vectors with respect to $\alpha_j$. Again, based on the IFE framework, the non-interface elements or edges and the associated standard finite element shape functions are all independent with the interface change, of which then the derivatives with respect to $\bfalpha$ vanish. Therefore to finish the sensitivity calculation, we derive the formula for the derivatives of $\mathbf{ K}_T$, $\mathbf{ E}^{r_1r_2}_e$, $\mathbf{ G}^{r_1r_2}_e$ and $\mathbf{ F}_T$, $\mathbf{ B}_e$, $\mathbf{ C}_e$, $\mathbf{ N}_e$ given by \eqref{ppife_local_mat_1}-\eqref{ppife_local_vec_Neu}, with respect to $\alpha_j$ on the interface elements $T\in\mathcal{T}^i_h$ and edges $e\in\mathcal{E}^i_h$.

\begin{equation}
\begin{split}
\label{KT_deri_2}
D_{\alpha_j}\mathbf{ K}_T&= \left( \int_T \beta \nabla \frac{\partial}{\partial \alpha_j}\psi_{k,T} \cdot \nabla \psi_{l,T} dX \right)_{k,l\in\mathcal{I}} + \left( \int_T \beta \nabla \frac{\partial}{\partial \alpha_j}\psi_{k,T} \cdot \nabla \psi_{l,T} dX \right)_{k,l\in\mathcal{I}}^T \\
&+\left( \sum_{i=1}^3 \int_{T_i} \beta \nabla \psi_{k,T} \cdot \nabla \psi_{l,T}~ dX ~ \textrm{tr}\left( (D_{\alpha_j} \mathbf{ J}_i) \mathbf{ J}^{-1}_i \right) \right)_{k,l\in\mathcal{I}}.
\end{split}
\end{equation}
}


\begin{thm} \label{D_alpha_local_vectors}
On each interface element $T\in \mathcal{T}_h^i$ and each interface edge $e \in \mathcal{E}_h^i$,
we have the following formulas for the material derivatives of
$\mathbf{F}_T, \mathbf{B}_e, \mathbf{C}_e$ and $\mathbf{N}_e$ with respect to $\alpha_j, j\in \mathcal{D}$:
\begin{subequations}
\begin{align}
D_{\alpha_j}\mathbf{ F}_T^k&= \left( \int_T f^k \frac{\partial\psi_{p,T}}{\partial \alpha_j} dX \right)_{p\in\mathcal{I}}  + \left( \int_T \nabla(f^k\psi_{p,T})\cdot\mathbf{ V}^j_T dX \right)_{p\in\mathcal{I}} \nonumber \\
 &+\left( \sum_{i=1}^3 \int_{T_i} f^k  \psi_{p,T}~ dX ~ \emph{tr}\left( (D_{\alpha_j} \mathbf{ J}_i) \mathbf{ J}^{-1}_i \right) \right)_{p \in\mathcal{I}}, \label{FT_deri} \\
D_{\partial_j} \mathbf{ B}_e^k &= \left(\int_e \beta g_D^k \nabla \frac{\partial\psi_{p,T}}{\partial \alpha_j}\cdot \mathbf{ n}_e  ds \right)_{p\in\mathcal{I}}
\nonumber \\
&~~~~~~~~ +  \Bigg( \beta^- g_D^k \nabla \psi^-_{p,T}\cdot \mathbf{ n}_e|_{P} - \beta^+ g_D^k \nabla \psi^+_{p,T}\cdot \mathbf{ n}_e|_{P}  \Bigg)_{p\in\mathcal{I}} \frac{ D_{\alpha_j}P\cdot(A_2-A_1)}{\norm{A_2-A_1}}, \label{DBe_deri} \\
D_{\partial_j} \mathbf{ C}_e^k &= \frac{\sigma^0_e}{|e|}\left(\int_e \beta g_D^k \frac{\partial\psi_{p,T}}{\partial \alpha_j} ds \right)_{p\in\mathcal{I}} +  \frac{\sigma^0_e}{|e|}\Bigg( \beta^- g_D^k  \psi^-_{p,T}|_{P} - \beta^+ g_D^k  \psi^+_{p,T}|_{P}  \Bigg)_{p\in\mathcal{I}} \frac{ D_{\alpha_j}P\cdot(A_2-A_1)}{\norm{A_2-A_1}}, \label{DCe_deri} \\
D_{\partial_j} \mathbf{ N}_e^k &= \left(\int_e g_N^k \frac{\partial\psi_{p,T}}{\partial \alpha_j}  ds \right)_{p\in\mathcal{I}} +  \Bigg( g_N^k  \psi^-_{p,T}|_{P} - g_N^k  \psi^+_{p,T}|_{P}  \Bigg)_{p\in\mathcal{I}} \frac{ D_{\alpha_j}P\cdot(A_2-A_1)}{\norm{A_2-A_1}}. \label{DNe_deri}
\end{align}
\end{subequations}
\end{thm}

Now, by Lemma 3.3 in \cite{2003HaslingerMkinen} again, we have the following
standard formula for the material derivative associated to the $j$-th design variable $\alpha_j$:
\begin{align}
\label{material_deri}
&D_{\alpha_j} \mathcal{J}_h = \sum_{k=1}^K \left( \frac{\partial \mathcal{J}_h}{\partial \mathbf{ u}_h^k}\cdot D_{\alpha_j}\mathbf{ u}_h^k \right)
 +\int_{\Omega_0} \frac{\partial J_h}{\partial \alpha_j} dX + \int_{\Omega_0} \nabla  J_h\cdot \mathbf{ V}^j dX +  \int_{\Omega_0} J_h ~ \textrm{div}\left( \mathbf{ V}^j \right)dX
\end{align}
in which we have used the fact that $\frac{\partial \mathcal{J}_h}{\partial \mathbf{ u}_h^k} =
\int_{\Omega_0} \frac{\partial J_h}{\partial \mathbf{ u}_h^k} dX$,
and, as demonstrated by examples presented in the next section,
$\frac{\partial \mathcal{J}_h}{\partial \mathbf{ u}_h^k}, \nabla J_h,
\frac{\partial J_h}{\partial \alpha_j}$ and $J_h$ itself are problem dependent, but they
are usually easy to calculate for many applications. Also, we note that $\mathbf{V}^j$ is given in \eqref{v_filed} and $\textrm{div}\left( \mathbf{ V}^j \right)$ is given in \eqref{div}; hence, we
proceed to derive formula for $\left( \frac{\partial \mathcal{J}_h}{\partial \mathbf{ u}_h^k} \right)\cdot D_{\alpha_j}\mathbf{ u}_h, ~j \in \mathcal{D}$ which can be directly used in \eqref{material_deri}.

\commentout{
First, by the definition of $\mathbf{V}^j$ given in \eqref{v_filed}, $\textrm{div}\left(\mathbf{ V}^j\right)$ is zero on every non-interface element. On each interface element $T$, configured as in Figure \ref{interface_element_map} without loss of generality, by \eqref{v_filed}, we have
\begin{equation}
\label{div}
\textrm{div}(\mathbf{ V}^j_{T_i})= \textrm{tr}\left( \nabla_X \mathbf{ V}^j_{T_i} \right)=\textrm{tr}\left( (D_{\alpha_j} \mathbf{ J}_i) \mathbf{ J}^{-1}_i \right), ~~i = 1, 2, 3,
\end{equation}
where the formula for $D_{\alpha_j} \mathbf{ J}_i, i, j \in \mathcal{D}$ is given in \eqref{J_123_deriv}.
}

For $D_{\alpha_j}\mathbf{ u}_h^k, 1 \leq k \leq K, j \in \mathcal{D}$, by differentiating the IFE system in \eqref{ppife_mat_unified} with respect to $\alpha_j$, we have
the following linear system for $D_{\alpha_j}\mathbf{ u}_h^k$:
$
\mathbf{ A}^k(X(\bfalpha), \bfalpha) ~ D_{\alpha_j} \mathbf{ u}_h^k  =  D_{\alpha_j}\mathbf{ F}^k(X(\bfalpha), \bfalpha)- D_{\alpha_j} \mathbf{ A}^k(X(\bfalpha), \bfalpha) ~\mathbf{ u}_h^k(\bfalpha)
$, $1 \leq k \leq K.$
Then, by the standard process in the discretized adjoint method \cite{2000GilesPierce}, we
can compute $\left(\frac{\partial \mathcal{J}_h}{\partial \mathbf{ u}_h^k}\right)\cdot D_{\alpha_j}\mathbf{ u}_h^k$ efficiently
(especially when $\abs{\mathcal{D}}$ is large) by solving for $\mathbf{Y^k}$ from
$\big(\mathbf{ A}^k\big)^T\mathbf{ Y}^k = \frac{\partial \mathcal{J}_h}{\partial \mathbf{ u}_h^k}$, and then
\begin{equation}
\label{adj_1}
\begin{split}
\left( \frac{\partial \mathcal{J}_h}{\partial \mathbf{ u}_h^k} \right)\cdot D_{\alpha_j}\mathbf{ u}_h &=
\mathbf{Y}^k\cdot \left( D_{\alpha_j}\mathbf{ F}^k(X(\bfalpha), \bfalpha)- D_{\alpha_j} \mathbf{ A}^k(X(\bfalpha), \bfalpha) ~\mathbf{ u}_h^k(\bfalpha) \right), ~~1 \leq k \leq K,
\end{split}
\end{equation}
where $\mathbf{ A}^k$ is the matrix for the $k$-th IFE equation described in \eqref{ppife_mat_unified}.
As summarized in the next section, one advantage of the proposed IFE method is that computations for the material derivatives \eqref{material_deri}
can be very efficiently implemented in the IFE framework.

\commentout{
{\color{red}
We note that the matrices $D_{\alpha_j} \mathbf{ A}^k(X(\bfalpha), \bfalpha)$ and the vectors $D_{\alpha_j}\mathbf{ F}^k(X(\bfalpha), \bfalpha)$
are assembled from the gradient formulas of local matrices and vectors prepared in Theorem \ref{D_alpha_local_matrices} and Theorem \ref{D_alpha_local_vectors},
and their assemblages can be done very efficiently because the related element-by-element assembling procedures are performed only over interface elements whose number is in the order of $O(h^{-1})$ versus the number of all elements in the order of $O(h^{-2})$ for a Cartesian mesh $\mathcal{T}_h$. Similarly,
the computations for the last two terms in \eqref{material_deri} also need to be carried out only on interface elements intersecting with $\Omega_0$. In contrast, preparing $D_{\alpha_j}\mathbf{ F}^k(X(\bfalpha), \bfalpha)$ and $D_{\alpha_j} \mathbf{ A}^k(X(\bfalpha), \bfalpha)$ is usually expensive within the Lagrange framework where a global velocity field requires to carry out the assemblages over all elements in a mesh \cite{1994ChoiChang}, and $D_{\alpha_j}\mathbf{ F}^k(X(\bfalpha), \bfalpha)$ and $D_{\alpha_j} \mathbf{ A}^k(X(\bfalpha), \bfalpha)$ are usually prepared approximately in methods in the Eullerian framework \cite{2011DunningPeterKim,2010PengWangXing,2012ZhangZhangZhu}.
}
}

\commentout{
All in all, in contrast to the Lagrange framework, the velocity can be calculated accurately and efficiently within the Eulerian framework. In order to compute the sensitivities on the elements which contain discontinuous material properties, many approximation approaches are proposed in literature
\cite{2011DunningPeterKim,2010PengWangXing,2012ZhangZhangZhu,2015NajafiSafdari}, the IFE method based method enables us to calculate the material derivatives of the local stiffness matrix and load vector on interface elements by analytical formulas

local matrices and vectors only on interface elements. Indeed, they are assembled from those matrices and vectors given in Theorem \ref{D_alpha_local_matrices} and Theorem \ref{D_alpha_local_vectors}, respectively, element-by-element only over interface elements instead of over all elements of $\mathcal{T}_h$. {\color{red}Therefore, the IFE method proposed here is advantageous compared with those in the literature because preparing $D_{\alpha_j}\mathbf{ F}^k(X(\bfalpha), \bfalpha)$ and $D_{\alpha_j} \mathbf{ A}^k(X(\bfalpha), \bfalpha)$
are expensive within the Lagrange framework [????] with a global velocity field, and $D_{\alpha_j}\mathbf{ F}^k(X(\bfalpha), \bfalpha)$ and $D_{\alpha_j} \mathbf{ A}^k(X(\bfalpha), \bfalpha)$ are usually prepared approximately within the Eullerian framework [????]}

The matrix $D_{\alpha_j} \mathbf{ A}^k(X(\bfalpha), \bfalpha)$ and vector $D_{\alpha_j}\mathbf{ F}^k(X(\bfalpha), \bfalpha)$ can be assembled in the ways similar to those for their
counterparts $\mathbf{ A}^k(X(\bfalpha), \bfalpha)$ and $\mathbf{ F}^k(X(\bfalpha), \bfalpha)$ described
in Section \ref{review_ife}, respectively.

Also, since both the velocity field and the shape derivatives of the IFE basis functions vanish on all the non-interface elements,
by formulas \eqref{v_filed} and \eqref{div}, the last two terms in \eqref{material_deri} need to be computed only on interface elements as follows
\begin{align}
\int_{\Omega_0} \left( \frac{\partial  J_h}{\partial X}\right)^T \mathbf{ V}^j dX &= \sum_{T\in\mathcal{T}^i_h} \sum_{i=1}^3 \int_{\Omega_0\cap T_i} \left( \frac{\partial J_h}{\partial X}\right)^T (D_{\alpha_j} \mathbf{ J}_i) \mathbf{ J}^{-1}_i(X-A_i)dX, \label{J_x_v}\\
\int_{\Omega_0} J_h ~ \textrm{div}\left( \mathbf{ V}^j \right)dX &= \sum_{T\in\mathcal{T}^i_h} \sum_{i=1}^3 \int_{\Omega_0\cap T_i} J_h ~ \textrm{tr}\left( (D_{\alpha_j} \mathbf{ J}_i) \mathbf{ J}^{-1}_i \right)dX. \label{J_div}
\end{align}
}

\commentout{
In summary, the gradient formula $D_{\bfalpha}\mathcal{J}_h = (D_{\alpha_j}\mathcal{J}_h)_{j \in \mathcal{D}}$ given in \eqref{material_deri} can be implemented accurately and efficiently in the procedure described below. \\
\\
{\bf Procedure to Compute the Gradient $D_{\bfalpha}\mathcal{J}_h$}: Compute $\frac{\partial  J_h}{\partial X}$
according to the specific form of $J_h$ pertinent to the application. Then, for each $j \in \mathcal{D}$,
\begin{itemize}
\item
Compute $\int_{\Omega_0} \frac{\partial J_h}{\partial \alpha_j} dX$ according to the specific form of $J_h$ and compute
$\int_{\Omega_0}  \left( \frac{\partial  J_h}{\partial X}\right)^T \mathbf{ V}^j dX$ and
$\int_{\Omega_0} J_h ~ \textrm{div}\left( \mathbf{ V}^j \right)dX$ by \eqref{J_x_v} and \eqref{J_div}, respectively.

\item
For $k = 1, 2, \cdots, K$, compute $\frac{\partial \mathcal{J}_h}{\partial \mathbf{ u}_h^k}$ according to the specific form of
$J_h$; solve \eqref{eq_for_Y} for $\mathbf{Y}^k$; assemble $D_{\alpha_j}\mathbf{ F}^k(X(\bfalpha), \bfalpha)$ and $D_{\alpha_j} \mathbf{ A}^k(X(\bfalpha), \bfalpha)$;
compute $\left( \frac{\partial \mathcal{J}_h}{\partial \mathbf{ u}_h^k} \right)\cdot D_{\alpha_j}\mathbf{ u}_h$ by \eqref{adj_1}.

\item
Then, compute $D_{\alpha_j} \mathcal{J}_h$ by \eqref{material_deri}.

\end{itemize}
}


\subsection{Implementation}
\label{Implementation}

In this subsection, we discuss the implementation of the proposed IFE-based shape optimization method. First, we summarize the discretization of forward/inverse problems and the sensitivity computation discussed in the previous subsections into the following algorithm.

\begin{algorithm}[H]\label{alg:IFE_shape_opt_alg}
\renewcommand{\thealgorithm}{}
\caption{The IFE Shape Optimization Algorithm}
\label{algo1}
\begin{algorithmic}[1]
\State Generate a fixed mesh and choose an initial design variable $\bfalpha$.
\State Loop until convergence.
\State Prepare data:
\begin{algsubstates}
\State use the design variables to generate the parametric curve as the numerical interface;
\State find the interface-mesh intersection points, interface edges and interface elements.
\end{algsubstates}
\State Prepare matrices and vectors for the IFE systems and compute the cost function:
\begin{algsubstates}
\State use \eqref{ppife_local_mat} and \eqref{ppife_local_vec} and the IFE shape functions given in \eqref{non_inter_loc_space} and \eqref{ife_space_inter_element} to assemble matrices and vectors $\mathbf{ A}^k, \mathbf{ F}^k, 1 \leq k \leq K$ for the IFE systems \eqref{ppife_mat_unified};
\State compute the PPIFE solutions $\mathbf{u^k}, 1 \leq k \leq K$ by \eqref{ppife_mat_unified} and compute the objective function $\mathcal{J}_h (\mathbf{ u}_h^1(\bfalpha), \mathbf{ u}_h^2(\bfalpha), \cdots, \mathbf{ u}_h^K(\bfalpha), \bfalpha)$
in \eqref{discrt_obj_2}.
\end{algsubstates}
\State Compute the shape sensitivities:
\begin{algsubstates}
\State prepare the velocity fields $\mathbf{ V}^j$, $j\in\mathcal{D}$, and shape derivatives of IFE shape functions according to \eqref{v_filed} and \eqref{ife_shapefun_shape_deriv}, respectively;
\State form the material derivatives of local matrices and vectors according to Theorem \ref{D_alpha_local_matrices} and Theorem \ref{D_alpha_local_vectors}, and use them to assemble the global matrices $D_{\alpha_j}\mathbf{ A}^k(X(\bfalpha), \bfalpha)$ and vectors $D_{\alpha_j} \mathbf{ F}^k(X(\bfalpha), \bfalpha)$;
\State compute $\frac{\partial \mathcal{J}_h}{\partial \mathbf{ u}_h^k}\cdot D_{\alpha_j}\mathbf{ u}_h^k$ for $k=1,\cdots,K$, according to \eqref{adj_1};
\State compute the terms $\int_{\Omega_0}\frac{\partial J_h}{\partial \alpha_j}dX$, $\int_{\Omega_0}\nabla  J_h\cdot \mathbf{ V}^j dX$ and $\int_{\Omega_0}J_h \text{div}(\mathbf{ V}^j)dX$ according to the given shape functional;
\State compute the material derivatives of the objective function according to \eqref{material_deri}.
\end{algsubstates}
\State Update the design variable $\bfalpha$ by a chosen gradient-based optimization algorithm.
\State End loop
\end{algorithmic}
\end{algorithm}

In this proposed IFE Shape Optimization Algorithm, we note that the mesh is fixed during the optimization process, and the only mesh information needed to be updated are those interface-mesh intersection points and interface elements/edges. Consequently, the global matrices $\mathbf{ A}^k$ and vectors $\mathbf{ F}^k$ in step 4 remain the same size and algebraic structure on this fixed mesh, which is beneficial for implementation. Also, they do not need to be completely re-assembled in each iteration,
because only those global basis functions whose supports overlap with the interface elements/edges in two consecutive iterations are changed. As a result, their assemblage can be done very efficiently by just updating those entries corresponding to the global basis functions whose supports overlap with the interface elements/edges in the previous and the current iteration.

In step 5 above (computing the shape sensitivities), we emphasize that the velocity fields and the shape derivatives of IFE shape functions are only needed on interface elements, which can be implemented according to the analytical formulas \eqref{v_filed} and \eqref{ife_shapefun_shape_deriv}. These two quantities vanishing over all the non-interface elements make the whole procedure of shape sensitivity computation remarkably efficient. Firstly the integration of the terms $\int_{\Omega_0}\nabla  J_h\cdot \mathbf{ V}^j dX$ and $\int_{\Omega_0}J_h \text{div}(\mathbf{ V}^j)dX$ in the material derivative of the objective functional \eqref{material_deri} only needs to be done on interface elements intersecting $\Omega_0$ because the involved integrands all vanish on the non-interface elements. Secondly assembling the matrices $D_{\alpha_j}\mathbf{ F}^k(X(\bfalpha), \bfalpha)$ and $D_{\alpha_j} \mathbf{ A}^k(X(\bfalpha), \bfalpha)$, i.e., the material derivatives of global matrices $\mathbf{ A}^k$ and $\mathbf{ F}^k$, is also a very efficient process since it is only performed over the interface elements/edges by the explicit formulas given in Theorems \ref{D_alpha_local_matrices} and \ref{D_alpha_local_vectors}. In summary, the shape sensitivity in this algorithm is done by computations only need to be carried out over
interface elements whose number is in the order of $O(h^{-1})$ versus the number of all elements in the order of $O(h^{-2})$ in the mesh.  In contrast, preparing $D_{\alpha_j}\mathbf{ F}^k(X(\bfalpha), \bfalpha)$ and $D_{\alpha_j} \mathbf{ A}^k(X(\bfalpha), \bfalpha)$ is usually expensive within the Lagrange framework where a global velocity field requires to carry out the assemblages over all elements in a mesh \cite{1994ChoiChang}, and $D_{\alpha_j}\mathbf{ F}^k(X(\bfalpha), \bfalpha)$ and $D_{\alpha_j} \mathbf{ A}^k(X(\bfalpha), \bfalpha)$ are usually prepared approximately in methods in the Eullerian framework, see related discussions in \cite{2011DunningPeterKim,2010PengWangXing,2012ZhangZhangZhu}.

In addition, whenever necessary, one could refine the mesh easily at any point of the optimization process. Because of the Cartesian grid used by the IFE method, information in the previous mesh can be easily transformed to a new mesh through an interpolation operator. As demonstrated by examples presented in the last section, we note that the mesh refinement actually enables us to obtain better reconstruction of interface in some challenging inverse geometric problems.

Finally, we note that the proposed IFE shape optimization algorithm is highly parallelizable because computing the velocity fields $\mathbf{ V}^j$ \eqref{v_filed}, shape derivatives of IFE shape functions $\frac{\partial \phi_T}{\partial \alpha_j}$ \eqref{ife_shapefun_shape_deriv} and the material derivatives of stiffness matrices and vectors $D_{\alpha_j} \mathbf{ A}^k(X(\bfalpha), \bfalpha)$ and $D_{\alpha_j}\mathbf{ F}^k(X(\bfalpha), \bfalpha)$, i.e., the material derivatives of objective functions \eqref{material_deri} with respect to each individual design variable $\alpha_j$, are independent with each other. Hence these computations can be done very efficiently with an easy implementation on modern parallel computers.

Therefore, we believe these properties together with the optimal accuracy of PPIFE solutions \eqref{optimal_accuracy} and the resulted optimal accuracy of discretized objective functions, regardless of the interface location, make the proposed IFE shape optimization algorithm advantageous compared with those in the literature.

\section{Some Applications}
\label{sec:Applications}

In this section, we demonstrate how the general IFE method proposed in the previous section can use a fixed mesh to solve a wide spectrum of interface inverse problems posed in the format of \eqref{inter_prob_0}-\eqref{ObjFun_4_interf_optimiz} by applying this method to, but not limited to, three representative interface inverse\slash design problems: (1). the output-least-squares problem \cite{2013Cantarero,2003ChanTai,2009GockenbachKhan}; (2). the Dirichlet-Neumann problem \cite{2013BelhachmiMeftahi,2005HolderDavid,1991SattingerTracy}; and (3). the heat dissipation minimization problem \cite{2008GaoZhangZhu,2004LiStevenXie,2011ZhangLiuQiao}. The first problem uses the interior data available on the whole or a portion of $\Omega$ to reconstruct\slash design the interface, the second one recovers the interface from the data only available on $\partial \Omega$, and the last one is an application for optimal design of heat conduction fields. These examples also provide additional hints/suggestions about how to implement the proposed IFE method efficiently.

All numerical examples to be presented are posed on the domain $\Omega = (-1,1)\times (-1,1)$ on which Cartesian meshes for these numerical examples are formed by cutting $\Omega$ into $N\times N$ congruent small squares and then cutting each small square into two triangles along a diagonal line of this small square. In the following discussion, we will specify the mesh size $N$ for each example and the numerical interface curve is parameterized by a cubic spline. This choice of parametrization is based on the accuracy, versatility, and popularity of the cubic spline, and we emphasize that the fixed mesh method developed here can be readily extended to other parameterizations.

\commentout{
In the first two cases, we assume the data is available on the whole domain $\Omega$, while in the last one, the data is only assumed to be known on a sub-domain $\Omega_0\subset\Omega$. All the
results reported here are generated by numerically solving the optimization problem \eqref{J_1_h} with the well known BFGS optimization algorithm using the
gradient $D_{\bfalpha}\mathcal{J}_h$, and each sequence of the optimizations is carried out on a fixed Cartesian mesh of $\Omega$ on which the
forward interface problem is discretized by the IFE method. All the Cartesian meshes used are formed by cutting $\Omega$ into
$N\times N$ uniform small squares and then cutting each small square into two triangles along a diagonal line of this small square.
}

\commentout{
{\color{red}In the numerical examples to be presented, we use ... say something about the numerical details or later in each example.} \\

{\color{red}We should not use the following two paragraphs in the article, but Ruchi should use them together with other experiences in his thesis} \\

Before presenting each numerical example, we first describe some implementation issues which may happen in the optimization process and the corresponding techniques to address such issues. The first issue concerns the reconstructed curve intersecting with itself, which is not allowed in our case. Since the reconstructed curve is a cubic spline generated by control points, our solution for this issue is to cancel a few control points which are related to intersection parts such that the cubic spline generated by the rest of the points does not intersect itself. In this procedure, to keep the most information of the curve already constructed, we should cancel as few control points as possible. And we resample the new curve to generate a new group of control points equally spaced on the curve. See Figure \ref{restart}(a) for an illustration. Then we restart the optimization from these new control points. This restart mechanism needs to assume the interface to be recovered is a Jordan curve.

\begin{figure}[H]
  \centering
  \subfigure[The initial body fitting mesh]{
    \label{intersection} 
    \includegraphics[width=2.5in]{graph/intersection_curve-eps-converted-to.pdf}}
  \hspace{0.01in}
  \subfigure[The body fitting mesh after movement]{
    \label{large_cur} 
    \includegraphics[width=2.5in]{graph/large_curva_curve-eps-converted-to.pdf}}
  \caption{An improper velocity field causes mesh distortion for large shape change}
  \label{restart} 
\end{figure}

The second issue concerns a point at which the reconstructed curve has a very large curvature. We have observed that the objective functional is usually insensitive to the variations of the curve around points with a large curvature so that the curve gets stuck there in the optimization process. Our remedy to this issue is to
abandon every control point whose curvature is over a prescribed threshold $\tau$ and generate a new curve using the rest points. We then resample a set of
control on this new curve from which we restart the optimization, see \ref{restart}(b) for illustration. To facilitate such a mechanism, it is better for use know priori an estimate of the maximum curvature of the true interface, denoted by $Cur_{max}(\Gamma)$ so that we can choose a $\tau \geq Cur_{max}(\Gamma)$, and in our numerical experiments, we usually set $\tau$ to be much larger than $Cur_{max}(\Gamma)$.
}
\commentout{
Hence we first set up a threshold value for curvature, denoted by $\tau$. And the solution to this issue is also to cancel the control point whose curvature is over $\tau$ and generate a new curve using the rest points. Still through resampling on the new curve, we generate a new group of control points from which we restart the optimization algorithm, see \ref{restart}(b) for illustration. To equip such a mechanics in our algorithm, we need to assume that the maximum curvature of the true interface, denoted by $Cur_{max}(\Gamma)$, can not go beyond $\tau$. In real implementation, we usually set $\tau$ to be much larger than $Cur_{max}(\Gamma)$. In the following examples, unless specially specified, we always set $\tau=80$.  According to the numerical experiments, it turns out that the proposed algorithm equipped these reinitianization mechanics can recover many interfaces which even have complicated shapes.
}

\subsection{An Output-Least-Squares Problem} \label{sec:Output-Least-Squares}
In this example, we consider an interface inverse/design problem associated to the interface forward problem described by \eqref{inter_prob_0} and \eqref{jump_cond_0} with $K=1$, in which, we assume an observation data $\bar{u}$ for the solution $u^1$ to the forward problem (or a target function in optimal design application) is available on a sub-domain $\Omega_0\subseteq\Omega$, and we need to recover\slash design the location and the shape of the interface from $\bar{u}$ by solving an output-least-squares problem \cite{2003ChanTai,2001ItoKunischLi}, i.e., by optimizing the following shape functional
\begin{equation}
\label{J_1}
\mathcal{J} (u^1(\Gamma), \Gamma)=\int_{\Omega_0}(u^1-\bar{u})^2dx
\end{equation}
where $u^1$ is the solution to the interface forward problem described by \eqref{inter_prob_0}, \eqref{jump_cond_0} and \eqref{beta} with a pure Dirichlet boundary condition $g^1_D$ on the whole $\partial\Omega$. This problem appears in oil/underwater reservoirs \cite{1983Ewing,1986Yeh} and optimal designing of cooling elements in battery systems \cite{2013PengNiakhai}. And a related time dependent problem is discussed in \cite{2013HarbrechtTausch}. Applying the IFE method proposed in
\eqref{discrt_obj_2} to the inverse problem formulated in \eqref{J_1} suggests to seek the design variable
$\bfalpha^*$ that minimizes the following discrete objective functional
\commentout{
The application for this objective functional governing by \eqref{inter_prob_0} can be found in oil/underwater reservoirs \cite{1983Ewing,2001Knowles,2004KnowlesLeYan} for steady state. In this case, $u$ represents the piezometric head, $\beta$ the transmissivity and $f$ the recharge. Recently the authors in \cite{2013PengNiakhai} used the same objective functional to investigate the optimal design of cooling elements in battery systems.

We follow \eqref{discrt_obj_1} to discretize the objective functional in \eqref{J_1} by replacing $u^1$ with its IFE solution $u_h^1$ given in \eqref{ppife_mat_unified}.
Then, we employ the procedure proposed in \eqref{discrt_obj_2} to solve this interface inverse{\color{blue}\slash design} problem by seeking the design variable
$\bfalpha^*$ that minimizes the following discrete objective functional}
\begin{equation}
\begin{split}
\label{J_1_h}
&\mathcal{J}_{h} (\bfu_h^1(\bfalpha), \bfalpha)=\int_{\Omega_0} J_h(\bfu_h^1(\bfalpha), X(\bfalpha), \bfalpha)dX, \\
\text{subject to}~~~~~~ & \mathbf{ A}^1(X(\bfalpha), \bfalpha)\bfu_h^1(\alpha) - \bfF^1(X(\bfalpha), \bfalpha) =\mathbf{ 0},
\end{split}
\end{equation}
where, expressing the IFE solution $u_h^1(\bfalpha) =u_h^1(\mathbf{ u}_h^1(\bfalpha),X(\bfalpha),\bfalpha)$ in the format given in \eqref{ife_solu_alpha}, we have
\begin{eqnarray}
\begin{split}\label{J_1_h_2}
J_h(\bfu_h^1(\bfalpha), X(\bfalpha), \bfalpha) &= \big({\tilde J}_h(\bfu_h^1(\bfalpha), X(\bfalpha), \bfalpha)\big)^2, \\
\text{with~~} {\tilde J}_h(\bfu_h^1(\bfalpha), X(\bfalpha), \bfalpha) &= \sum^{|\mathring{\mathcal{N}}_h|}_{i=1} u_i^1(\bfalpha) \phi_{i}(X(\bfalpha), \bfalpha) +\sum^{|\mathcal{N}_h|}_{i=|\mathring{\mathcal{N}}_h|+1}g^1_{D}(X_i)\phi_i(X(\bfalpha), \bfalpha) -\bar{u}. \\
\end{split}
\end{eqnarray}
We have shown in \eqref{optimal_shape_fun_eq_2} that this discretized objective function has the optimal second order accuracy to approximate the continuous one regardless of the interface location and shape. According to \eqref{J_1_h_2}, the evaluation of $J_h(\bfu_h^1(\bfalpha), X(\bfalpha), \bfalpha)$ is straightforward and it is obvious that
\begin{equation} \label{eq:J_1_DJDX_DJDalpha}
\begin{split}
&\nabla J_h= 2{\tilde J}_h(\bfu_h^1(\bfalpha), X(\bfalpha), \bfalpha)\left( \sum^{|\mathring{\mathcal{N}}_h|}_{i=1} u_i^1(\bfalpha)
\nabla\phi_{i}(X(\bfalpha), \bfalpha) +\sum^{|\mathcal{N}_h|}_{i=|\mathring{\mathcal{N}}_h|+1}g^1_{D}(X_i)
\nabla\phi_i(X(\bfalpha), \bfalpha) - \nabla\bar{u}\right), \\
&\frac{\partial  J_h}{\partial \alpha_j} = 2{\tilde J}_h(\bfu_h^1(\bfalpha), X(\bfalpha), \bfalpha)\left( \sum^{|\mathring{\mathcal{N}}_h|}_{i=1} u_i^1(\bfalpha)
\frac{\partial\phi_{i}(X(\bfalpha), \bfalpha)}{\partial \alpha_j} +\sum^{|\mathcal{N}_h|}_{i=|\mathring{\mathcal{N}}_h|+1}g^1_{D}(X_i)
\frac{\partial\phi_i(X(\bfalpha), \bfalpha)}{\partial \alpha_j}\right),
\end{split}
\end{equation}
where $\bar{u}$ is assumed to be optimization independent and the shape derivatives of the global IFE basis $\phi_{i}(X(\bfalpha), \bfalpha)$ are zero on all the non-interface elements, but on every
interface element $T$, $\nabla\phi_{i}(X(\bfalpha), \bfalpha)$ and $\frac{\phi_{i}(X(\bfalpha), \bfalpha)}{\partial \alpha_j}, j \in \mathcal{D}$ can be computed according to \eqref{ife_shapefun} and \eqref{ife_shapefun_shape_deriv}, respectively.
Furthermore, a direct calculation leads to
\commentout{because $\mathcal{J}_{h} (\bfu_h^1(\bfalpha), \bfalpha)$ is a quadratic function in terms of $\mathbf{ u}_h^1$, we have the following formula for computing $\frac{\partial \mathcal{J}_h}{\partial \mathbf{ u}_h}$ efficiently in the IFE frame work:}
\begin{align}
&\mathcal{J}_h=  \left(\begin{array}{c} \mathbf{ u}_h^1 \\ \mathbf{ g}^1_D \end{array}\right)^T\mathbf{ M} \left(\begin{array}{c} \mathbf{ u}_h^1 \\ \mathbf{ g}^1_D \end{array}\right) - 2 \left(\begin{array}{c} \mathbf{ u}_h^1 \\ \mathbf{ g}^1_D \end{array}\right)^T \bar{\mathbf{ u}}+ \bar{\mathbf{ u}}^T\bar{\mathbf{ u}}, ~~~\frac{\partial \mathcal{J}_h}{\partial \mathbf{ u}_h^1} = \mathbf{ M}_0 \left(\begin{array}{c} \mathbf{ u}_h^1 \\ \mathbf{ g}^1_D \end{array}\right) -\bar{\mathbf{ u}}_0,
\label{J_1_h_grad_1} \\
\text{where}~~&\mathbf{ M}=\left( \int_{\Omega_0} \phi_i\phi_j dX \right)_{i=1,j=1}^{|\mathcal{N}_h|,|\mathcal{N}_h|} \in \mathbb{R}^{|\mathcal{N}_h|\times|\mathcal{N}_h|}, ~~~~ \bar{\mathbf{ u}}=\left( \int_{\Omega_0} \bar{u}\phi_i dX \right)_{i=1}^{|\mathcal{N}_h|} \in \mathbb{R}^{|\mathcal{N}_h|\times1}, \label{J_1_h_grad_2}
\end{align}
and $\mathbf{ M}_0$, $\bar{\mathbf{ u}}_0$ are formed by the first $|\mathring{\mathcal{N}}_h|$ columns of $\mathbf{ M}$ and $\bar{\mathbf{ u}}$, respectively. Formulas above confirm the observation that the computations for $\frac{\partial \mathcal{J}_h}{\partial \mathbf{ u}_h^k}, \nabla J_h,
\frac{\partial J_h}{\partial \alpha_j}$ and $J_h$ itself are problem dependent but they are usually straightforward to calculate within the IFE framework. These preparations can then be utilized in the proposed IFE Shape Optimization Algorithm presented in Section \ref{Implementation}.


\begin{table}[H]
\begin{center}
\footnotesize
\renewcommand{\arraystretch}{1.5}
  \begin{tabular}{ | @{}Sc | @{}Sc | @{}Sc | @{}Sc | }
    \hline
    Cases & $\beta$                                 & Interface $S$ and initial guess & Data $\bar{u}$ \\ \hline
    {\bf Case 1} & $\left.\begin{array}{c} \beta^-=1 \\ \beta^+=20 \end{array}\right.$ &   $\left.\begin{array}{c} S=(x^2+y^2)^2(1+0.8\sin{(6\arctan{(y/x)})})-0.1 \\ S_0=(x+0.6)^2+(y+0.2)^2-(\pi/9)^2 \end{array}\right.$ & $\begin{aligned}\bar{u}=&S/\beta^s~\textrm{in}~\Omega^s \\  &s=\pm  \end{aligned}$ \\ \hline
    {\bf Case 2} & $\left.\begin{array}{c} \beta^1=1~~ \beta^2=10 \\ \beta^3=100 \end{array}\right.$ & $\left.\begin{array}{c} S=4\sin(\pi x)\cos(\pi y +\pi/2)-2 \\ S^1_0=64x^2+144(y+0.5)^2-\pi^2 \\ S^2_0=64x^2+144(y-0.5)^2-\pi^2 \end{array}\right.$ & $\begin{aligned}\bar{u}=&S/\beta^i  ~\textrm{in}~\Omega^i \\  &i=1,2,3  \end{aligned}$ \\ \hline
    {\bf Case 3} & $\left.\begin{array}{c} \beta^-=1 \\ \beta^+=10 \end{array}\right.$ & $\left.\begin{array}{c} S=r-1, ~ \textrm{where} ~ r=(16x^2+64(y-0.4)^2)/\pi^2 \\ S_0=(x-0.4)^2+(y-0.2)^2-(\pi/6.28)^2 \end{array}\right.$  & $\begin{aligned} & \bar{u}  =  \frac{1024}{\pi^4\beta^s}(r^{\frac{5}{2}}-1) \\
    & +\frac{1024}{\pi^4\beta^-}  ~\textrm{in}~ \Omega_0^s, ~ s=\pm.
\end{aligned}$  \\
    \hline
  \end{tabular}
  \caption{Configuration for the Output-Least-Squares Problem}
  \renewcommand{\arraystretch}{1}
\end{center}
\label{cases_for_output_least_squares_problem}
\end{table}
We now present three specific cases for this interface inverse/design problem whose key data are described in Table \ref{cases_for_output_least_squares_problem}. In this table, $S(x,y) = 0$ is the target curve $\Gamma$ to be recovered that is plotted as a dotted curve (in red color) in the related figures. We use the BFGS optimization algorithm \cite{J.Nocedal_S.Wright_2006} in step 6 of the IFE Shape Optimization Algorithm presented in Section \ref{Implementation}, for which, $S_0(x, y) = 0$ is the initial curve that is plotted as a solid curve (in blue color) in the related figures as all other presented approximate curves in the BFGS iterations. \\

\noindent
{\bf Case 1}: The data ${\bar u}(X)$ is given on the whole $\Omega$. The numerical
curve is a parametric cubic spline with 20 control points, the initial curve is a simple circle but the target curve has a star shape representing a certain complexity.
In order to capture the complicated geometry, especially the six petals, we implement the algorithm on a $120\times120$ mesh.
Some approximate curves generated in the BFGS iterations are presented in Figure \ref{case2:cir_to_star} from which we can see a quick evolution of the numerical
curve towards to the target curve for this inverse\slash design problem even with a complicated geometry, and this suggests a benefit of the accurate gradient formula available for the proposed IFE method.

\begin{figure}[H]
\label{test1_1}
  \centering
    \includegraphics[width=4.7in]{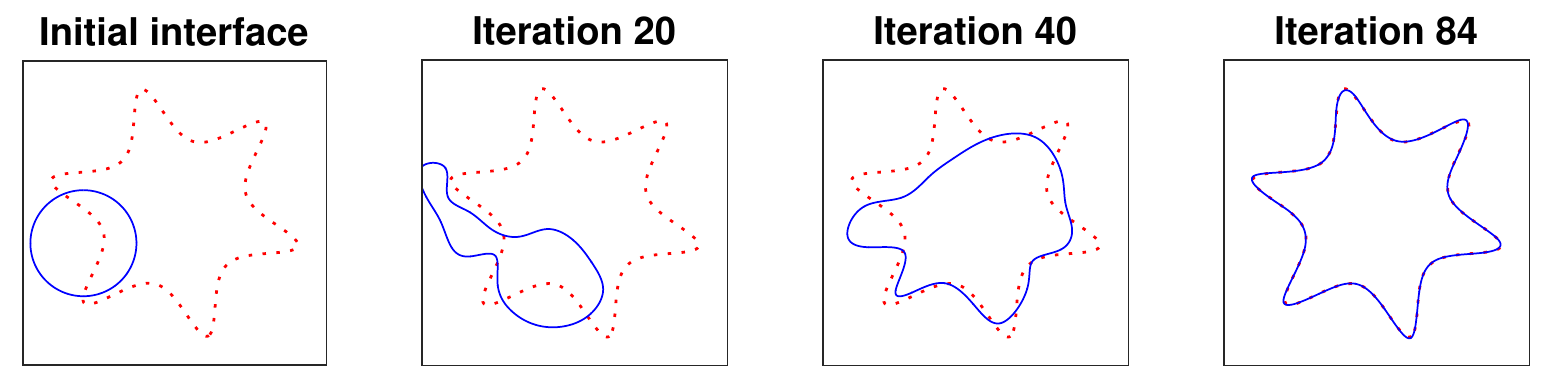}
      \caption{Optimization process for case 1 }
      \label{case2:cir_to_star} 
\end{figure}


\noindent\textbf{Case 2}: We demonstrate how the proposed algorithm can handle an interface inverse\slash design problem whose target interface
consists of multiple closed curves. For this purpose and for simplicity, we consider the case in which interface $\Gamma$ is formed by the two simple curves such that $S(x,y)=0$. We denote the sub-domain inside the upper-left dotted curve (in red color) by $\Omega^1$, the sub-domain inside the lower-right dotted curve (in red color) by $\Omega^2$, and denote sub-domain outside these two closed dotted curve by $\Omega^3$, see Figure \ref{case3:cir_to_multi_interface}. The
interface problem described by \eqref{inter_prob_0} and \eqref{jump_cond_0} and its corresponding IFE discretization can be readily modified to suit the present interface configuration in which the parameter $\beta$ is a piecewise constant function such that its value on $\Omega_i$ is $\beta^i$, $i=1,2,3$. The data $\bar{u}$ for this problem is given on the whole domain $\Omega$ on which a $80\times80$ mesh is utilized. Each numerical curve component is a parametric cubic spline with 15 control points and 30 control points in total. As demonstrated in Figure \ref{case3:cir_to_multi_interface}, the approximate curves by the proposed IFE method evolve from the initial approximate curve components $S_0^1$ and $S_0^2$ to the target curve components after 110 iterations. We notice that the numerical curve component started from $S_0^1$ converges to the exact curve component much faster than that started from $S_0^2$. After 10 iterations, the first numerical
curve component is already quite close to the target curve, while the evolution of the second numerical curve component is obviously less. We believe
the objective function is more sensitive to the design variables for the first numerical curve component than the second because the jump $\beta^3/\beta^1$ is much larger than $\beta^3/\beta^2$ in this example, and the gradient in the proposed IFE method is capable to capture this kind of subtle dependence of the
objective function on the design variables.
\begin{figure}[H]
\label{test1_3}
  \centering
    \includegraphics[width=4.7in]{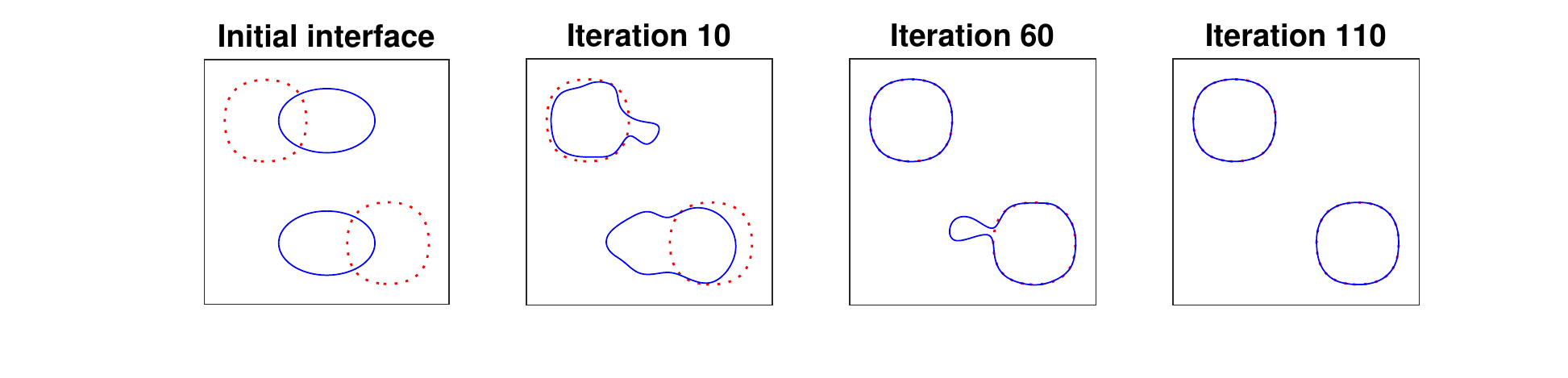}
      \caption{Optimization process for case 2 }
          \label{case3:cir_to_multi_interface} 
\end{figure}

\noindent\textbf{Case 3}: The data function $\bar u$ is given in proper sub-domain $\Omega_0=[-0.5,1]\times[0,1]$ in
the upper-right of $\Omega$ illustrated in Figure \ref{case4:partial_domain}, with $\Omega_0^s = \Omega_0 \cap \Omega^s, s = \pm$. We also implement the algorithm on a $80\times80$ mesh for this example. The numerical curve is a parametric cubic spline with 20 control points.
As presented in Figure \ref{case4:partial_domain}, the numerical curve converges in about 80 iterations. We observe that the converged numerical
curve is a much better approximation to the target interface curve $\Gamma$ inside $\Omega_0$ than outside, and we believe this is a reasonable consequence
of the available data function $\bar u$ given only on $\Omega_0$, and we think this example suggests again that the gradient in the proposed IFE method can capture the nature of the interface inverse problem in accordance with the available data.
\begin{figure}[H]
\label{test1_4}
  \centering
    \includegraphics[width=4.7in]{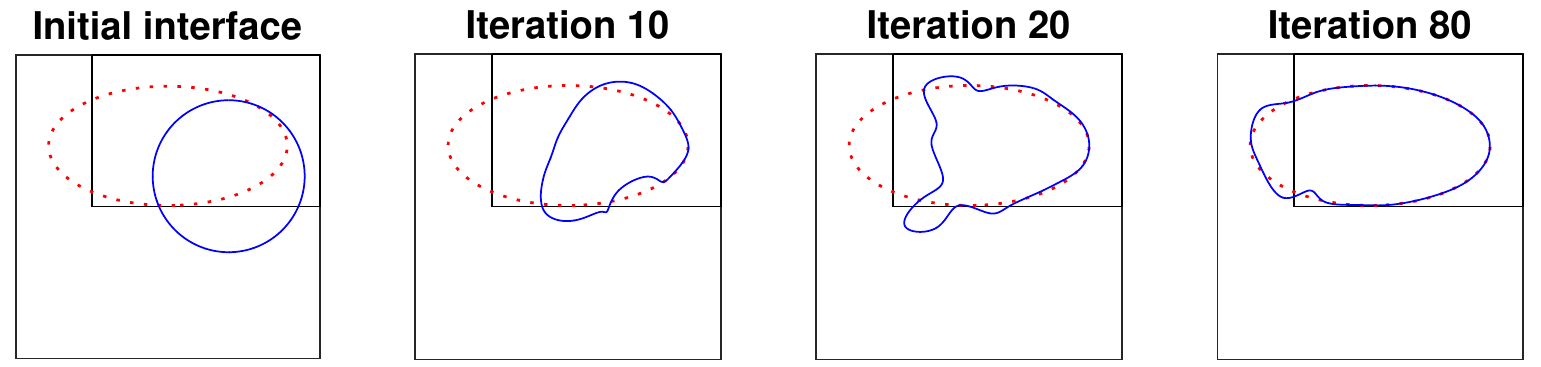}
      \caption{Optimization process for case 3}
      \label{case4:partial_domain} 
\end{figure}


\subsection{The Dirichlet-Neumann Problem with a Single Measurement}
In this group of numerical examples, we apply the propose IFE method to the popular but challenging inverse Dirichlet-Neumann problem in which
we try to recover the interface $\Gamma$ from one Neumann data $g_N$ provided on the boundary for an interface forward problem of the elliptic equation described by \eqref{inter_prob_0}-\eqref{beta} with a pure Dirichlet boundary condition. This type of inverse problems have a wide range of applications in electronic impendence tomography (EIT) \cite{2013BelhachmiMeftahi,2005HolderDavid,1998Lionheart} where one wishes to detect a material interface by injecting the voltage potential $g_D$ on $\partial\Omega$ and measuring the current density $g_N$ on (or a portion of) $\partial\Omega$. When the charge source $f=0$, it is referred as the Calder\'on's inverse conductivity problem \cite{2006Calderon} which is well-known ill-conditioned since only the data on the boundary $\partial \Omega$ is available for the reconstruction of $\Gamma$.

\commentout{
By the interface problem with the Dirichlet boundary condition, one can consider the so called Dirichlet-to-Neumann operator $\Lambda_{\beta}:H^{\frac{1}{2}}(\partial\Omega) \rightarrow H^{-\frac{1}{2}}(\partial\Omega)$ such that
\begin{equation}
\label{D_N_map}
\Lambda_{\beta}: u|_{\partial\Omega}=g_D \rightarrow \frac{\partial u}{\partial \mathbf{ n}}|_{\partial\Omega} = g_N.
\end{equation}
Kohn and Vogelius \cite{1984KohnVogelius} proved that the operator \eqref{D_N_map} can uniquely determine the piecewise analytical coefficients $\beta$. Then the analyticity assumption was weakened by a series of articles \cite{1996Nachman,1987SylvesterUhlmann}. Recently, the authors in \cite{2006Astala} proved the uniqueness by only assuming the coefficients $\beta^s\in L^{\infty}(\Omega^s)$, $s=-, +$ and bounded away from zero and infinity. If $\beta$ is assumed to be a piecewise $C^2$ function and only the coefficient $\beta^-$ on $\Omega^-$ included inside $\Gamma$ is the unknown, then the uniqueness of $\Gamma$ and $\beta^-$ can be also guaranteed from the operator \eqref{D_N_map} \cite{1988Isakov}.

Our focus here is the inverse interface problem for recovering the interface $\Gamma$ from a pair of Dirichlet-Neumann data function $g_D$ and $g_N$ assuming the values of the piecewise constant coefficient $\beta$ are known.
}
We formulate this inverse problem as a shape optimization problem with a Kohn-Vogelius type functional \cite{1987KohnVogelius,2012AntonioAlfredo}:
\begin{equation}
\label{J_2_DN}
\mathcal{J}(u^1(\Gamma),u^2(\Gamma), \Gamma)=\int_{\Omega}|u^1 - u^2|^2dX,
\end{equation}
where as in \cite{2013BelhachmiMeftahi}, $u^1$ and $u^2$ are the solutions of the following interface forward problems:
\begin{alignat*}{2}
  \begin{aligned} & \begin{cases}
 -\nabla\cdot(\beta\nabla u^1)=f,  \\
[u^1]|_{\Gamma}=0,\\
[\beta\nabla u \cdot\mathbf{ n}]|_{\Gamma}=0,  \\
  u^1=g^1_D = g_D, ~~~ \textrm{on}\; \partial\Omega, \\
  \end{cases}\\
  \end{aligned}
    & \hskip 2em &
  \begin{aligned} & \begin{cases}
-\nabla\cdot(\beta\nabla u^2)=f,  \\
[u^2]|_{\Gamma}=0,\\
[\beta\nabla u^2 \cdot\mathbf{ n}]|_{\Gamma}=0,  \\
u^2 = g^2_D = g_D, ~~~ \textrm{on}\; \partial\Omega_D, ~~\frac{\partial u^2}{\partial\mathbf{ n}}=g_N^2, ~~~ \textrm{on}\; \partial\Omega_N,
  \end{cases} 
  \end{aligned}
\end{alignat*}
and $\int_{\Omega} u^2 dx =u_0$ needs to be imposed when $\partial\Omega_N = \partial \Omega$.
\commentout{
The global uniqueness for this inverse problem has been established for several classes of interfaces, such as the convex cylinders or unions of discs in \cite{1990IsakovPowell} and convex polyhedrons in \cite{1994BarcelFabes}. A local uniqueness result is given in \cite{1995Alessandrini} by assuming some conditions on boundary data which guarantee $\nabla u$ has zero index on $\Gamma$.

Again, we follow \eqref{discrt_obj_1} to discretize the objective functional in {\color{blue} \eqref{J_2_DN}} by replacing $u^k, k = 1, 2$ with their IFE solutions $u_h^k, k = 1, 2$ given in \eqref{ppife_mat_unified}, and
}
Again, we employ the IFE method proposed in \eqref{discrt_obj_2} to solve this interface inverse problem by seeking the design variable $\bfalpha^*$ that minimizes the following objective function
\begin{equation}
\begin{split}
\label{J_2_h}
&\mathcal{J}_{h} (\bfu_h^1(\bfalpha), \bfu_h^2(\bfalpha), \bfalpha)=\int_{\Omega_0} J_h(\bfu_h^1(\bfalpha), \bfu_h^2(\bfalpha), X(\bfalpha), \bfalpha)dX, \\
\text{subject to}~~~~~~ & \mathbf{ A}^k(X(\bfalpha), \bfalpha)\bfu_h^k(\alpha) - \bfF^k(X(\bfalpha), \bfalpha) =\mathbf{ 0},~~k = 1, 2,
\end{split}
\end{equation}
where $J_{h}(\bfu_h^1(\bfalpha), \bfu_h^2(\bfalpha), X(\bfalpha), \bfalpha) = \big({\tilde J}_{h}(\bfu_h^1(\bfalpha), \bfu_h^2(\bfalpha), X(\bfalpha), \bfalpha)\big)^2$ with
\begin{equation}
\label{J_2_h_2}
{\tilde J}_{h}(\bfu_h^1(\bfalpha), \bfu_h^2(\bfalpha), X(\bfalpha), \bfalpha) = \sum^{|\mathring{\mathcal{N}}_h|}_{i=1} u^1_i \phi_{i}(X) +\sum^{|\mathcal{N}_h^m|}_{i=|\mathring{\mathcal{N}}_h|+1}g_{D}(X_i)\phi_i(X) - \sum^{|\mathcal{N}^m_h|}_{i=1} u^2_i \phi_{i}(X).
\end{equation}
\commentout{
We note that $\mathcal{J}_{h} (\bfu_h^1(\bfalpha), \bfu_h^2(\bfalpha), \bfalpha)$ can be evaluated quickly by matrix-vector operations as follows:
\begin{align}
&\mathcal{J}_{h} (\bfu_h^1(\bfalpha), \bfu_h^2(\bfalpha), \bfalpha) = \left[ \left(\begin{array}{c} \mathbf{ u}^1_h \\ \mathbf{ g}^1_D \end{array}\right) - \left(\begin{array}{c} \mathbf{ u}^2_h \\ \mathbf{ g}^2_D \end{array}\right) \right]^T \mathbf{ M}  \left[ \left(\begin{array}{c} \mathbf{ u}^1_h \\ \mathbf{ g}^1_D \end{array}\right) -  \left(\begin{array}{c} \mathbf{ u}^2_h \\ \mathbf{ g}^2_D \end{array}\right) \right], \label{J_2_h_matrix_version} \\
&\mathbf{ g}^1_D=\left( g_D(X_i) \right)^{|\mathcal{N}_h|}_{i=|\mathring{\mathcal{N}_h}|+1}, ~~\mathbf{ g}^2_D=\left( g_D(X_i) \right)^{|\mathcal{N}|_h}_{i=|\mathcal{N}^m_h|+1},~~\mathbf{ M}=\left( \int_{\Omega} \phi_i\phi_j dX \right)_{i=1,j=1}^{|\mathcal{N}_h|,|\mathcal{N}_h|}.
\end{align}
}
By a similar argument to \eqref{optimal_shape_fun_eq_3}, we can show this discretized shape functional still has the optimal second order accuracy for approximating the original one on an fixed mesh. Also, similar to \eqref{eq:J_1_DJDX_DJDalpha}-\eqref{J_1_h_grad_2} in the output-least-squares problem discussed in Section \ref{sec:Output-Least-Squares}, formulas for $\nabla J_h$, $\frac{\partial J_h}{\partial \alpha_j}, j \in \mathcal{D}$ as well as
$\frac{\partial\mathcal{J}_{h}}{\partial \mathbf{ u}^1_h}$ and $\frac{\partial\mathcal{J}_{h}}{\partial \mathbf{ u}^2_h}$ can be readily derived and implemented
in the IFE framework, and these preparations can then be used in the proposed IFE Shape Optimization Algorithm presented in Section \ref{Implementation}.

\commentout{

Also, formulas similar to those in \eqref{eq:J_1_DJDX_DJDalpha} can be easily derived to compute $\frac{\partial J_h}{\partial X}$ and
$\frac{\partial J_h}{\partial \alpha_j}, j \in \mathcal{D}$. Furthermore, by direct calculations, we have
\begin{align}
&\frac{\partial\mathcal{J}_{h}}{\partial \mathbf{ u}^1_h}=2 \mathbf{ M}_1 \left[\left(\begin{array}{c} \mathbf{ u}^1_h \\ \mathbf{ g}^1_D \end{array}\right)  - \left(\begin{array}{c} \mathbf{ u}^2_h \\ \mathbf{ g}^2_D \end{array}\right) \right] , ~~~ \textrm{and} ~~~
\frac{\partial\mathcal{J}_{h}}{\partial \mathbf{ u}^2_h}=2\mathbf{ M}_2 \left[ \left(\begin{array}{c} \mathbf{ u}^2_h \\ \mathbf{ g}^2_D \end{array}\right) - \left(\begin{array}{c} \mathbf{ u}^1_h \\ \mathbf{ g}^1_D \end{array}\right) \right], \label{J_2_h_grad} \\
&\mathbf{ M}_1=\left( \int_{\Omega} \phi_i\phi_j dX \right)_{i=1,j=1}^{|\mathring{\mathcal{N}}_h|,|\mathcal{N}_h|}, ~~~ \textrm{and}~~~ \mathbf{ M}_2=\left( \int_{\Omega} \phi_i\phi_j dX \right)_{i=1,j=1}^{|\mathcal{N}^m_h|,|\mathcal{N}_h|}. \label{J_2_h_matrices}
\end{align}
}


For the Dirichlet-Neumann inverse problem, we report $3$ experiments that are configured with the target curve $S(x, y) = 0$ and exact solution to the
interface forward problem $u(x, y)$ given in Table \ref{cases_for_DN_problem}. We note that $u(x, y)$ given in this table is used only to generate the
Dirichlet and Neumann boundary data for the related inverse problem. As before, the BFGS algorithm \cite{J.Nocedal_S.Wright_2006} is employed to carry out the shape optimization described by \eqref{J_2_h} according to the proposed IFE Shape Optimization Algorithm, for which, $S_0(x,y) = 0$ given in Table \ref{cases_for_DN_problem} is the initial curve that is
plotted as a solid curve (in blue color) in the related figures as all other presented approximate curves in the BFGS iterations. In addition, we refine the mesh once the optimization has stalled at a certain numerical curve, i.e., the reconstructed interface curve is not moving, to obtain better reconstruction.

\begin{table}[H]
\begin{center}
\footnotesize
\renewcommand{\arraystretch}{1.5}
  \begin{tabular}{ | @{}Sc | @{}Sc | @{}Sc | @{}Sc | }
    \hline
    Cases & $\beta$                                 & Interface $S$ and initial guess & Exact $u$ \\ \hline
    {\bf Case 1} & $\left.\begin{array}{c} \beta^-=1 \\ \beta^+=10 \end{array}\right.$ & $\left.\begin{array}{c} S=r-1, ~ \textrm{where} ~ r=(100(x-0.4)^2+36(y+0.3)^2)/\pi^2 \\ S_0=(x-0.1)^2+y^2-(\pi/4)^2 \end{array}\right.$  & $\begin{aligned} & u  =  \frac{3600}{\pi^4\beta^s}(r^{\frac{5}{2}}-1) \\
    & +\frac{3600}{\pi^4\beta^-}  ~\textrm{in}~ \Omega^s ~ s=\pm
\end{aligned}$  \\ \hline
    {\bf Case 2} & $\left.\begin{array}{c} \beta^-=1 \\ \beta^+=2 \end{array}\right.$ & $\left.\begin{array}{c} S=(2((x+0.5)^2+y^2)-x-0.2)^2-((x+0.5)^2+y^2)+0.3 \\ S_0=64(x-0.5)^2+16y^2-\pi^2 \end{array}\right.$ & $\begin{aligned}u=&S/\beta^s  ~\textrm{in}~\Omega^s, \\  &s=\pm  \end{aligned}$ \\ \hline
    {\bf Case 3} & $\left.\begin{array}{c} \beta^-=1, \\ \beta^+=2 \end{array}\right.$ & $\left.\begin{array}{c} S=\sin{(\pi x)}+\frac{\pi}{1.5}y+0.1 \\ S_0=y+0.15/\pi \end{array}\right.$  & $\begin{aligned} u  =&S/\beta^s  ~\textrm{in}~\Omega^s, \\  &s=\pm
\end{aligned}$  \\
    \hline
  \end{tabular}
  \caption{Configuration for the Dirichlet-Neumann Problem}
  \renewcommand{\arraystretch}{1}
\end{center}
\label{cases_for_DN_problem}
\end{table}

\noindent\textbf{Case 1}: The Neumann data is given on the whole $\partial \Omega$. The numerical
curve is a parametric cubic spline with 20 control points. Some representative approximate curves generated in the optimization are plotted as solid curves (with blue color) in Figure \ref{examp2_case1} in which the dotted curve (in red color) is the target curve to be recovered. These numerical results demonstrate that the propose IFE method can handle a large shape change, as illustrated in Figure \ref{examp2_case1} from the initial interface to the one generated by the third iteration. We note that such a large shape change often causes mesh distortion when body fitting mesh is used, but the proposed IFE method totally avoid this issue by the using a fixed interface independent mesh. The first 30 iterations are generated on a $80\times80$ mesh. The final result is generated on a $320\times320$ mesh. Clearly, with a finer mesh we can obtain more accurate reconstruction. The numerical curve quickly converges to the target curve after about $80$ iterations, and this demonstrates again the benefit of
the fact that the objective function in the proposed algorithm is a good approximation of the continuous objective functional defined by \eqref{J_2_DN} and the gradient of the objective function in the proposed algorithm provides a good sensitivity with respect to design variables in the numerical curve.

\commentout{
In the first case, we let the true interface be an ellipse $100(x-0.4)^2+36(y+0.3)^2=\pi^2$ plotted in red color in Figure \ref{examp2_case1}, and let the Dirichlet data and Neumann data be generated by the function \eqref{example_u} with $a=\pi/10$, $b=\pi/6$ and $\beta^-=1$, $\beta^+=10$. We start the BFGS iteration from an initial curve which is a circle of radius $\pi/4$ centered at $(0.1, 0)$, see the blue curve in the first plot of Figure \ref{examp2_case1}. {\color{blue}The numerical
curve is a parametric cubic spline with 20 control points.} Some representative approximate curves in the optimization are provided in Figure \ref{examp2_case1} with blue color. These numerical results demonstrate that the propose IFE method can handle large shape change, as illustrated in Figure \ref{examp2_case1} from the initial interface to the one generated by the second iteration. We note that such a large shape change often causes mesh distortion if body fitting mesh is used, but the proposed IFE method totally avoid this issue by the using a fixed interface independent mesh. We also observe that the numerical curve quickly converges to the target curve after about $80$ iterations, and this demonstrates again the benefit of
the fact that the objective function in the proposed algorithm is a good approximation of the continuous objective functional defined by
\eqref{J_2_DN} and the gradient of the objective function in the proposed algorithm provides a good sensitivity with respect to design variables in the numerical curve.

And 20 control points are used in the optimization. The reconstruction process is given by the Figure \ref{examp2_case1}. As we mentioned before, our algorithm can handle large shape change, as illustrated in the Figure \ref{examp2_case1} from the initial interface to the second iteration. The large shape change will yield mesh distortion if body fitting mesh is used. Our algorithm can totally eliminate this issue by the using the IFE method on the interface independent mesh. However we observe that after the first iteration, the curve intersects itself. So we use the mechanic described above to restart the optimization by a new curve as shown on the iteration 3 in Figure \ref{examp2_case1}. The reconstruction curve quickly find the location of the true interface after only one iteration. It indicates that our formulas to compute the sensitivity are very accurate. And the reconstructed curve converges to the true interface after about 80 iterations.
However, as shown by the last graph in Figure \ref{examp2_case1}, we observe that the convergence on the right side of the true interface is a little better than the left side. We guess this phenomenon is due to the fact that the right side of the true interface is more close to the boundary on which the data is only available.
}
\begin{figure}[H]
  \centering
    \includegraphics[width=4.7in]{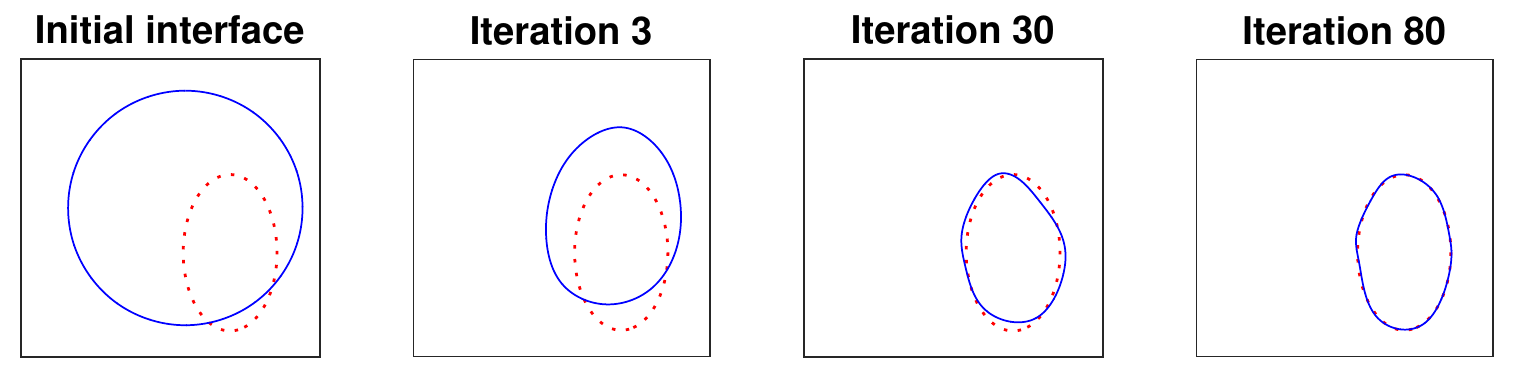}
      \caption{Reconstruction process for case 1}
       \label{examp2_case1}
\end{figure}

\commentout{
\noindent\textbf{Case 2} In the second case, we consider a non-conical but still convex interface, which is given by a level set function $S(x,y)=0$, where $(x-0.4)^4+(y-0.3)^4=(\pi/8)^4$. It is a square-like curve, as shown on the upper-right in $\Omega$ by Figure \ref{examp2_case2}. And the boundary data $g_N$ and $g_D$ are given by the function
\begin{equation}
\label{example2_u_2}
u(x,y) =\left\{
\begin{aligned}
&\frac{S(x,y)}{\beta^1},  &\; (x,y)\in \Omega^1,\\
&\frac{S(x,y)}{\beta^2}, & \; (x,y) \in \Omega^2,
\end{aligned}
\right.
\end{equation}
where $\beta^-=1$, $\beta^+=2$ and $\Omega^1$ is the sub-domain inside the true interface $S(x,y)=0$. And the initial interface is a circle $x^2+y^2=(\pi/4)^2$. We use 20 control points for the optimization. The initial mesh size is 80-by-80; we uniformly refine it to be 160-by-160 after 44 iterations and to be 320-by-320 after 70 iterations, as shown by Figure \ref{examp2_case2}. The reconstructed interface converges to the exact one after about 145 iterations. And we can clearly see that the reconstructed curve can find outline of the exact interface even on a relatively coarse mesh. However finer meshes have to be used to push the reconstruction more precise. The similar behavior was also observed in \cite{2015GiacominiPantz} in which the authors dynamically increased the number of elements especially near the interface during the optimization process.

\begin{figure}[H]
  \centering
    \includegraphics[width=4.7in]{graph/test2_2/test2_2-eps-converted-to.pdf}
      \caption{Reconstruction process for case 2}
       \label{examp2_case2} 
\end{figure}
}


\noindent\textbf{Case 2}: We now consider a more difficult Dirichlet-Neumann interface inverse problem whose exact solution interface curve $\Gamma$ is
non-conical and non-convex with a kidney-like shape plotted as dotted curve (in red color) in Figure \ref{examp2_case3}, and, to the best of our knowledge, there is no general theory to guaranty the uniqueness of the solution to this inverse problem with only one single pair of Dirichelt and Neumann data. The Neumann data is given on the whole $\partial \Omega$.
The numerical curve is a parametric cubic spline with 20 control points. In this case, we present 4 plots in Figure \ref{examp2_case3}: the first one is the initial guess, and the 2nd, 3rd and 4th plots are obtained on a $80\times80$, $160\times160$ and $320\times320$ mesh, respectively. Even on a relatively coarse $80\times80$ mesh, our algorithm can capture the basic feature of the target curve to be recovered. Again, we can observe that finer mesh can push the numerical curve to the exact target curve, but this process takes far more iterations than \textbf{Case 1}. We believe this is caused by the challenging nature of this interface inverse problem whose exact solution is non-convex; nevertheless, the proposed IFE method still produces an approximate solution quite satisfactory to a certain extend.
\begin{figure}[h]
  \centering
    \includegraphics[width=4.7in]{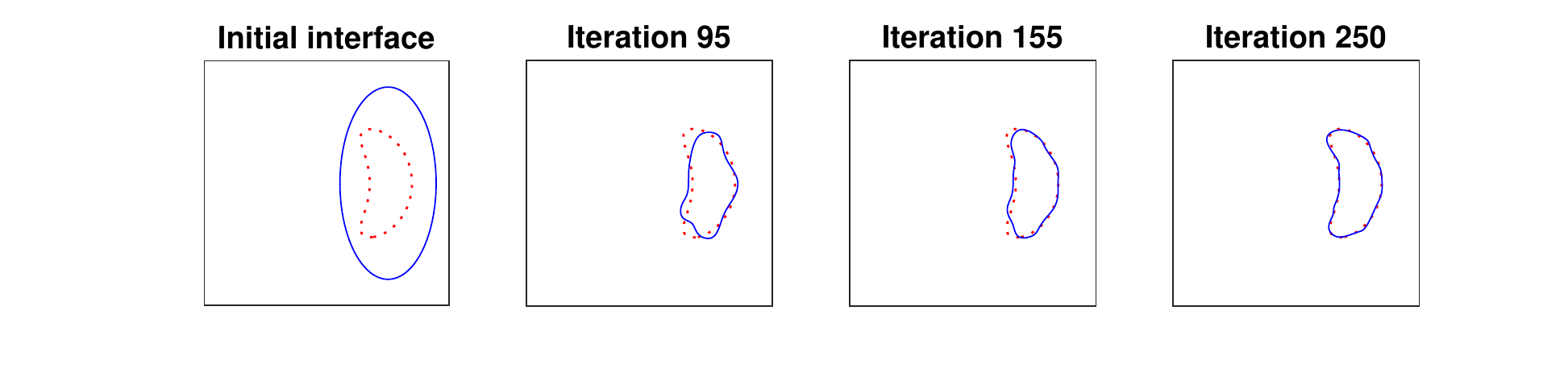}
      \caption{Reconstruction process for case 2}
       \label{examp2_case3} 
\end{figure}
\commentout{
 We also use 20 control points for the optimization. Again the initial mesh size is 80-by-80; and it is refined to be 160-by-160 and 320-by-320 after 95 iterations and 155 iterations. On a coarse mesh, we could see the reconstructed curve can recover the left side convex part, but it is difficult to recover the right-side non-convex part and the two corners. And it can significantly improve the reconstruction by refining the mesh twice. This case is more challenging than the previous ones in the sense that it takes much more iterations to converge especially on the non-convex part and the two corners. Actually, in theory, researchers have only shown the global uniqueness for discs and convex polygons \cite{1994BarcelFabes,1990IsakovPowell}, which, in another point of view, suggests the difficulties for reconstructing non-convex interfaces. But the numerical result indicates our algorithm is also successful even for this non-convex interface.
}


\noindent\textbf{Case 3}: In this case, the Neumann data is provided only on a proper subset of the boundary
$\partial \Omega$. Specifically, the true interface $\Gamma$ is the level set $S(x,y)=0$ plotted as the dotted curve (in red color) in Figure \ref{examp2_case4} that
separates $\Omega$ into two sub-domains $\Omega^-$ and $\Omega^+$ below and above $\Gamma$, and the Neumann data function $g_N$ is given only on the lower and upper edge of the square domain $\Omega$. The numerical interface is a 1-D cubic spline $y=y(t)$, $t\in[-1,1]$ with 10 control points whose end points match the exact interface. The first plot in Figure \ref{examp2_case4} shows the initial guess and  the 2nd, 3rd and 4th plots are obtained on a $80\times80$, $160\times160$ and $320\times320$ mesh, respectively. Again, the algorithm can produce a quite good reconstruction even on a coarse mesh ($80\times80$) to a certain extend. The last plot in Figure \ref{examp2_case4} shows that the numerical curve after $140$ iterations matches the exact curve well, and this demonstrates
that the proposed IFE method can treat a Dirichlet-Naumann interface inverse problem that has a limited Neumann data measured on part of the boundary
of $\Omega$. We also test the case in which the Neumann data is on the left and right boundary of $\Omega$ instead of the lower and upper edges, but the result is not as satisfactory as the one presented here.

\commentout{with $\beta^-=1$, $\beta^+=2$. Since the true interface $\Gamma$ has a simple 1-D geometry, i.e., it can be represented by $y=-1.5/\pi(\sin{(x)}+0.1)$,
we represent the numerical interface by a 1-D cubic spline $y=y(t)$, $t\in[-1,1]$ {\color{red}with 10 control points} whose end points match the exact interface. Consequently, the design variables are the $y$-coordinates of the control points. We start the BFGS iteration from the horizontal line $y=-0.15/\pi$ plotted in blue color in the first plot of Figure \ref{examp2_case4} which contains some numerical curves constructed in the optimization iterations.
The last plot in Figure \ref{examp2_case4} shows that the numerical curve after $140$ iterations matches the exact curves well, and this demonstrates
that the proposed IFE method can treat a Dirichlet-Naumann interface inverse problem that has a limited Neumann data measured on part of the boundary
of $\Omega$.

{\color{blue}We also test the case that the Neumann data is on the left and right boundary of $\Omega$ instead of the lower and upper edges. However the result is not as satisfactory as the one presented here. We think this is due to the fact the concave down or up parts are difficult to detect on the boundary.} {\color{red}Some comments about its limitations such as when it might fail????}

The evolution of the reconstructed interface is given by the Figure \ref{examp2_case4}. The initial mesh size is 80-by-80; and it is uniformly refined to be 80-by-80 and 160-by-160 after 41 iterations and 94 iterations, respectively. According to the numerical results, the algorithm can quickly find the location of the true interface; but it takes much more iterations to improve the convergence to the exact interface. The similar convergence behavior is one of the features of the ill-conditioned inverse problems, which is due to the fact that objective functional is very unsensitive to the errors.
}
\begin{figure}[H]
  \centering
    \includegraphics[width=4.7in]{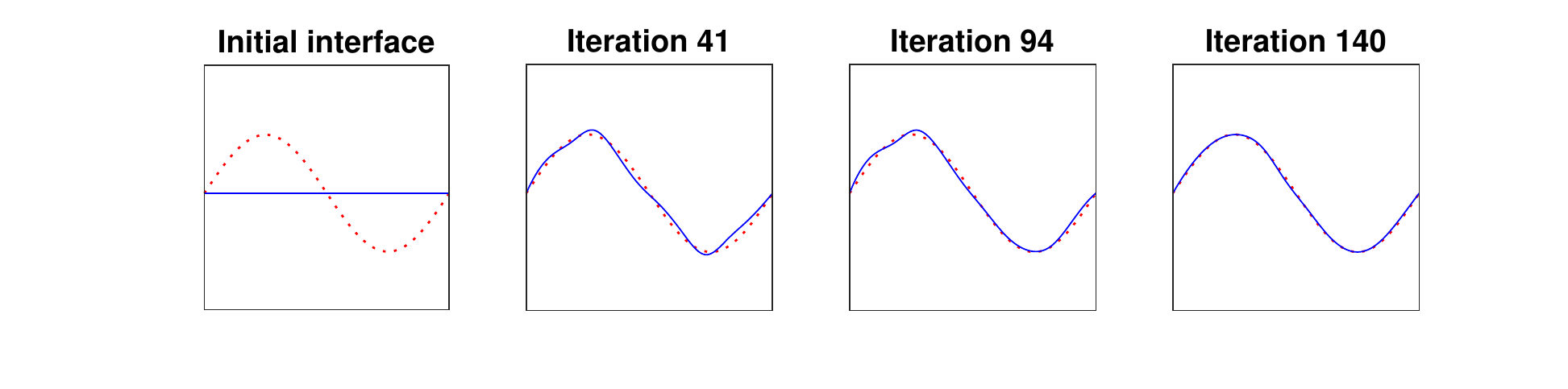}
      \caption{Reconstruction process for case 3}
       \label{examp2_case4} 
\end{figure}



\subsection{The Heat Dissipation Problem}
We now consider an application of the proposed IFE method to an optimal design problem for a heat system in which the goal is to minimize the overall heat dissipation by optimally distributing two materials in a domain \cite{2008GaoZhangZhu,2006Gersborg,2011ZhangLiuQiao}. This thermal design problem has wide applications such as cooling fins \cite{2001AttetkovVolkov,2002SasikumarBalaji} and high-conductivity channel of electronic components \cite{1997Bejan}.
\commentout{
Similar shape optimization problems are also very popular in elasticity equations, i.e., minimizing the elastic compliance \cite{2004AllaireJouveToader,2012Shojaee}. Besides, in \cite{2007Pedregal}, the authors investigated the existence to this type of optimal design problem by introducing the so called Div-Curl Young measures and considering a suitable relaxed version of the original optimization problem. Roughly speaking, they proved the existence of the infimum in a generalized sense.
}

In the steady heat conduction situation, this design problem is to find an optimal curve $\Gamma^*$ separating two chosen materials that can minimize the following objective functional \cite{2008GaoZhangZhu}:
\begin{equation}
\label{J_3}
\mathcal{J} (u^1(\Gamma), \Gamma)=\int_{\Omega} \nabla u^1\cdot (\beta\nabla u^1) dx \text{~~~subject to~~~} |\Omega^1|\leqslant \theta |\Omega|,
\end{equation}
where $u^1$ is the solution to the interface problem described by \eqref{inter_prob_0}-\eqref{jump_cond_0} with $K = 1$ with a Dirichlet boundary condition,
$\Omega^1$ is the sub-domain filled with the high conductivity material, and $\theta\in(0,1)$ is prescribed design parameter. By the proposed IFE method
\eqref{discrt_obj_2}, we seek a design variable $\bfalpha^*$ that minimizes the following objective function
\begin{equation}
\begin{split}
\label{J_3_h}
&\mathcal{J}_{h} (\bfu_h^1(\bfalpha), \bfalpha)=\int_{\Omega_0} J_h(\bfu_h^1(\bfalpha), X(\bfalpha), \bfalpha)dX, \\
\text{subject to}~~~~~~ & \mathbf{ A}^1(X(\bfalpha), \bfalpha)\bfu_h^1(\alpha) - \bfF^1(X(\bfalpha), \bfalpha) =\mathbf{ 0}, \text{~~and~~} |\Omega^1|\leqslant \theta |\Omega|,
\end{split}
\end{equation}
where
\begin{equation}
\label{J_3_h_0_fun}
J_{h} (\bfu_h^1(\bfalpha),X(\bfalpha),\bfalpha) = \beta \left|\sum^{|\mathring{\mathcal{N}}_h|}_{i=1} u_i^1 \nabla\phi_{i}(X) +\sum^{|\mathcal{N}_h|}_{i=|\mathring{\mathcal{N}}_h|+1}g_{D}(X_i) \nabla\phi_i(X)\right|^2.
\end{equation}
Since the objective functional involves gradients, following the idea in derivation \eqref{optimal_shape_fun_eq_3}, we can show that the discretized functional can approximate the true shape functional with an optimal first order accuracy independent of the interface shape and location. Also, similar to \eqref{eq:J_1_DJDX_DJDalpha}-\eqref{J_1_h_grad_2} again, formulas can be derived for $\mathcal{J}_{h}(\bfu_h^1(\bfalpha), \bfalpha)$, $\frac{\partial J_h}{\partial \alpha_j}, j \in \mathcal{D}$ and
$\frac{\partial \mathcal{J}_{h}}{\partial \mathbf{ u}_h^1}$ within the IFE framework. In particular, we have $\nabla J_h = {\bf 0}$.
\commentout{
The objective function $\mathcal{J}_{h}(\bfu_h^1(\bfalpha), \bfalpha)$ can be evaluated efficiently by matrix-vector operations as follows:
\begin{equation}
\label{J_3_h_matrix_form}
\mathcal{J}_{h}(\bfu_h^1(\bfalpha), \bfalpha) = \left(\begin{array}{c} \mathbf{ u}_h^1 \\ \mathbf{ g}_D \end{array}\right)^T \mathbf{ K} \left(\begin{array}{c} \mathbf{ u}_h^1 \\ \mathbf{ g}_D \end{array}\right)
\end{equation}
where
\begin{equation}
\label{stiffness_mat}
\mathbf{ K}=\left(\begin{array}{cc} \mathbf{ K}_{11} & \mathbf{ K}_{12} \\ \mathbf{ K}_{12} & \mathbf{ K}_{22} \end{array}\right) \in \mathbb{R}^{|\mathcal{N}_h|\times|\mathcal{N}_h|}
\end{equation}
is the global stiffness matrix in block form according to the vector partition in \eqref{J_3_h_matrix_form}. Since
$\phi_i, 1\leq i \leq \abs{\mathcal{N}_h}$ are linear IFE basis functions, by \eqref{J_3_h_0_fun}, it is obvious that $\frac{\partial J_h}{\partial X} = {\bf 0}$,
and a formula similar to the one in \eqref{eq:J_1_DJDX_DJDalpha} can be easily derived to compute
$\frac{\partial J_h}{\partial \alpha_j}, j \in \mathcal{D}$. Furthermore, by direct calculations, we can see
\begin{equation}
\label{J_3_h_grad}
\frac{\partial \mathcal{J}_{h}}{\partial \mathbf{ u}_h^1} = \mathbf{ K}_{11} \mathbf{ u}_h + \mathbf{ K}_{12} \mathbf{ g}_D.
\end{equation}
}
These preparations can then be employed in the proposed IFE Shape Optimization Algorithm together with the SQP (sequential quadratic programming) method \cite{J.Nocedal_S.Wright_2006} to carry out the constrained optimization numerically.

\commentout{
 to compute the gradient $D_{\bfalpha}\mathcal{J}_h$ according to the formula \eqref{material_deri}, and the gradient is used in a method such as the SQP (sequential quadratic programming) method \cite{J.Nocedal_S.Wright_2006} to carry out the constrained optimization numerically.
}

We test the proposed IFE method on a specific design problem configured in the domain $\Omega$ that contains a design independent heat source $f=-1$
on a center square $[-0.1,0.1]\times[-0.1,0.1]$, the boundary temperature is fixed to be $u=0$ and $\theta=0.5$, see the illustration in Figure \ref{test3_1}.
The two materials separated by the curve $\Gamma$ are such that $\beta^- = 1$ and $\beta^+=10^{-3}$. We start the SQP iteration from a circle $x^2+y^2=0.82^2$ plotted
as a solid curve (in blue color) in Figure \ref{optimal_design}(a), and the numerical curve in the optimization is a parametric cubic spline with 20 control points. We use the $160\times160$ mesh in this example. After 28 iterations, the proposed algorithm generates a design shown in Figure \ref{test3_1}(b) whose patten is very similar to the one reported in \cite{2008GaoZhangZhu}.

\begin{figure}[H]
  \centering
  \subfigure[The heating conditions and initial guess]{
    \label{heat_condition_initial_guess} 
    \includegraphics[width=1.8in]{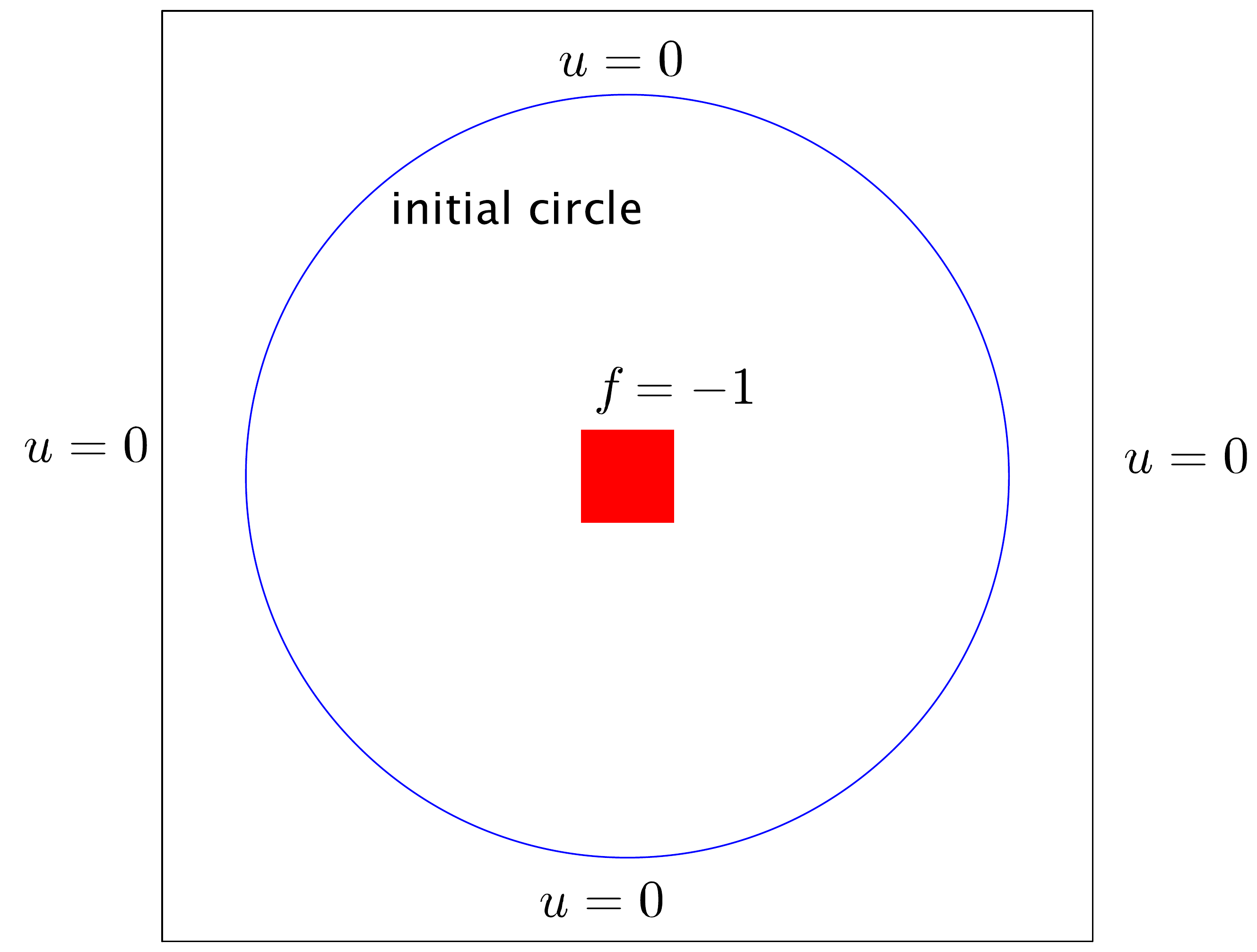}}
  \hspace{0.01in}
  \subfigure[The optimal design]{
    \includegraphics[width=1.35in]{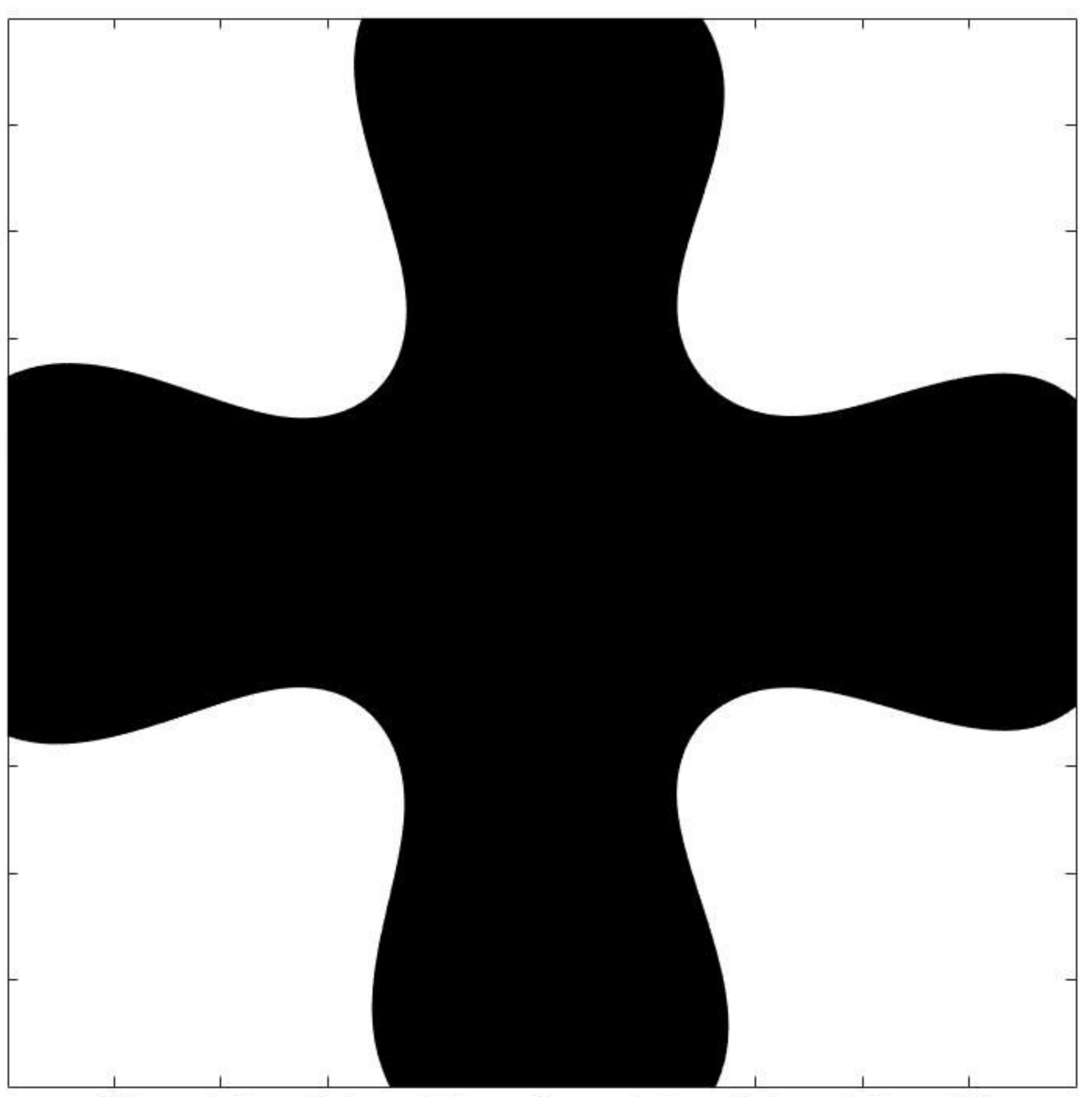}}
     \label{optimal_design}
  \caption{The heat dissipation problem}
  \label{test3_1} 
\end{figure}

\commentout{
The filled sub-domain is occupied by the high conductivity material, which has to find the shortest heat dissipation channels from the center to four edges.
}



\commentout{
\section{Conclusion}

We develop a shape optimization algorithm for solving interface inverse problems based on the IFE method. The proposed algorithm is suitable for inverse geometric problems in the sense that it is based on interface independent meshes. On such interface independent meshes, we build a $H^1$ \textit{velocity field} by establishing a explicit function relating the design variables and spacial variables. And the \textit{velocity field} vanishes on all the non-interface elements, which results in efficient computation of sensitivities. Also, we compute the shape derivatives of the coefficients in the IFE basis functions analytically. Then we combine these results with a discretized adjoint method to evalute the sensitivities. The proposed algorithm is efficient and accurate for inverse geometric problems because (1) no mesh regeneration procedure is needed and the \textit{velocity field} can be analytically established; (2) the forward solvers based on the IFE method have optimal accuracy; and (3) the sensitivity can be analytically computed on the discretization level and almost all the computation only need to be carried on interface elements.

For applications, we consider three problems: the output-least-squares problem where the data is available on the whole or portion of $\Omega$; the Dirichlet-Neumann problem where the data is only available on $\partial\Omega$; and a real optimal design problem for heat dissipation. We formulate these three inverse problems as interface optimization problems with different shape objective functionals. It turns out our method can be easily applied to different objective functionals where the only difference is the partial derivatives of the objective functions with respect to the IFE solution vector $\mathbf{ u}_h$, which can be also explicitly formulated in these three examples.

The numerical results turn out to be very promising. In general, our method can handle large shape change without any issue about the mesh distortion since the interface independent mesh is used. And we observe that the method can quickly find the location or the outline of the true interface during the first several iterations, which suggests that our sensitivity computation based the IFE discretization and fixed mesh is very accurate. Specifically, in the first problem, the reconstruction is very good even for interface with complicated geometry and multiple interface curves. Our other numerical experiments also suggest the reconstruction is successful for both large and small contrast of coefficients $\beta$. In the second problem, we carry out the numerical experiments on both convex and non-convex inclusions $\Omega^1$. The reconstruction of the convex one is clearly better than the non-convex ones, which agrees with the theoretical analysis in the literature \cite{1994BarcelFabes,1990IsakovPowell}. And the reconstruction for a 1-D interface with Neumann data only available on a portion of $\Omega$ is still successful. Besides, for the Dirichlet-Neumann problems, we have to uniformly refine the mesh to guarantee better reconstruction. The similar technique is also used in \cite{2015GiacominiPantz} to improve the reconstruction. Finally in the last example, for a bench mark problem from \cite{2007Pedregal} about minimizing the heat dissipation on $\Omega$, we obtain very similar design as in that article.
}

\bibliographystyle{plain}

\end{document}